\documentclass[11pt,oneside]{amsart}

\usepackage{amsmath,ifthen, amsfonts, amssymb,
srcltx, amsopn, color}
\usepackage{verbatim}
\usepackage{graphicx}

\usepackage{paralist}
\usepackage{float}
\restylefloat{figure}

\makeatletter
\def\section{\@startsection{section}{1}%
\z@{.5\linespacing\@plus\linespacing}{.3\linespacing}%
{\normalfont\scshape\centering}}
\def\subsection{\@startsection{subsection}{2}%
\z@{.3\linespacing\@plus.5\linespacing}{-.3em}%
{\normalfont\bfseries}}
\def\subsubsection{\@startsection{subsubsection}{3}%
\z@{.3\linespacing\@plus.5\linespacing}{-.3em}%
{\normalfont\itshape}}
\makeatother

\usepackage[final=true]{hyperref}

\newcommand{\showcomments}{yes}

\newsavebox{\commentbox}
{\ifthenelse{\equal{\showcomments}{yes}}%
{\footnotemark
        \begin{lrbox}{\commentbox}
        \begin{minipage}[t]{1.25in}\raggedright\sffamily\tiny
        \footnotemark[\arabic{footnote}]}
{\begin{lrbox}{\commentbox}}}%
{\ifthenelse{\equal{\showcomments}{yes}}%
{\end{minipage}\end{lrbox}\marginpar{\usebox{\commentbox}}}
{\end{lrbox}}}

\newtheorem{thm}{Theorem}[section]
\newtheorem{lem}[thm]{Lemma}

\newtheorem{cor}[thm]{Corollary}

\newtheorem{prop}[thm]{Proposition}
\newtheorem*{thmA}{Theorem~A}
\newtheorem*{thmB}{Theorem~B}
\newtheorem*{thmC}{Theorem~C}
\newtheorem*{corD}{Corollary~D}
\newtheorem*{thmE}{Theorem~E}

\theoremstyle{definition}
\newtheorem{defn}[thm]{Definition}
\newtheorem{rem}[thm]{Remark}
\newtheorem{exmp}[thm]{Example}

\newtheorem{prob}[thm]{Problem}
\newtheorem{question}[thm]{Question}

\newtheorem{claim*}{Claim}

\DeclareMathOperator{\dimension}{dim}

\DeclareMathOperator{\image}{im}

\DeclareMathOperator{\Aut}{Aut}

\DeclareMathOperator{\stabilizer}{Stab}
\DeclareMathOperator{\diam}{diam}

\DeclareMathOperator{\gproj}{Proj_{_{\contact X}}}

\DeclareMathOperator{\psprod}{\boxtimes}

\newcommand{\coll}{\;\;\makebox[0pt]{$\bot$}\makebox[0pt]{$\smile$}\;\;}

\newcommand{\field}[1]{\mathbb{#1}}
\newcommand{\integers}{\ensuremath{\field{Z}}}

\newcommand{\naturals}{\ensuremath{\field{N}}}
\newcommand{\reals}{\ensuremath{\field{R}}}

\newcommand{\interior} [1] {{\ensuremath \text{\rm Int}(#1) }}

\makeatletter

\newcommand{\Rmnum}[1]{\mathbf{{\expandafter\@slowromancap\romannumeral #1@}}}

\newcommand{\simp}{\ensuremath{\partial_{_{\triangle}}}}

\newcommand{\contact}[1]{\ensuremath{\Gamma\mathbf{#1}}}
\newcommand{\crossing}[1]{\ensuremath{\Delta\mathbf{#1}}}

\newcommand{\dive}[2]{\ensuremath{{\mathbf{div}}({#1},{#2})}}
\newcommand{\diver}[4]{\ensuremath{\overline{\mathbf{div}}({#1},{#2},{#3};{#4})}}
\newcommand{\divers}[2]{\ensuremath{\widetilde{\mathbf{div}}({#1},{#2})}}

\makeatother

\let\oldmarginpar\marginpar
\renewcommand\marginpar[1]{\-\oldmarginpar[\raggedleft\footnotesize #1]%
{\raggedright\footnotesize #1}}

\begin{document}
\title[The simplicial boundary]{The simplicial boundary of a CAT(0) cube complex}
\author[Mark~F.~Hagen]{Mark F. Hagen}
           \address{Dept. of Math. \& Stat.\\
                    McGill University \\
                    Montreal, Quebec, Canada H3A 2K6 }
           \email{markfhagen@gmail.com}

\date{\today}
\maketitle

\begin{abstract}
For a CAT(0) cube complex $\mathbf X$, we define a simplicial flag complex $\simp\mathbf X$, called the \emph{simplicial boundary}, which is a natural setting for studying non-hyperbolic behavior of $\mathbf X$.  We compare $\simp\mathbf X$ to the Roller, visual, and Tits boundaries of $\mathbf X$, give conditions under which the natural CAT(1) metric on $\simp\mathbf X$ makes it isometric to the Tits boundary, and prove a more general statement relating the simplicial and Tits boundaries.  $\simp\mathbf X$ allows us to interpolate between studying geodesic rays in $\mathbf X$ and the geometry of its \emph{contact graph} $\contact X$, which is known to be quasi-isometric to a tree, and we characterize essential cube complexes for which the contact graph is bounded.  Using related techniques, we study divergence of combinatorial geodesics in $\mathbf X$ using $\simp\mathbf X$.  Finally, we rephrase the rank-rigidity theorem of Caprace-Sageev in terms of group actions on $\contact X$ and $\simp\mathbf X$ and state characterizations of cubulated groups with linear divergence in terms of $\contact X$ and $\simp\mathbf X$.
\end{abstract}
\vspace{-5pt}
\small
\tableofcontents
\normalsize
\vspace{-15pt}
\section*{Introduction}
Since their introduction as a class of examples in~\cite{Gromov87}, CAT(0) cube complexes have become increasingly ubiquitous in group theory, and the class of groups acting on cube complexes is vast.  The original examples come from Bass-Serre theory, since trees are 1-dimensional cube complexes, but, in the sense of Gromov's density model, many more groups act on cube complexes than split (compare the results of~\cite{OllivierWiseDensity} to those of~\cite{DahmaniGuirardelPrzytycki}).  Sageev provided a general means of obtaining actions on cube complexes by constructing, in~\cite{Sageev95}, a $G$-cube complex from a \emph{semi-splitting} of the group $G$, which notion arose in work of Houghton and Scott~\cite{Houghton,Scott}.  Many groups have been cubulated using Sageev's construction and related techniques, including Coxeter groups~\cite{NibloReeves03}, Artin groups of type FC~\cite{CharneyDavis95b}, diagram groups, including Thompson's group $V$~\cite{Farley2003,Farley2005}, small-cancellation groups~\cite{WiseSmallCanCube04}, random groups at sufficiently low density in Gromov's model~\cite{OllivierWiseDensity}, and groups with quasiconvex hierarchies, including one-relator groups with torsion, limit groups, and fundamental groups of Haken hyperbolic 3-manifolds~\cite{WiseIsraelHierarchy}.  Right-angled Artin groups were shown to act on cube complexes by more direct means: the Salvetti complex of a right-angled Artin group was shown to be a nonpositively curved cube complex by Charney and Davis~\cite{CharneyDavis94}.

The structure of the cube complex $\mathbf X$ is governed by the \emph{hyperplanes}, which are convex subspaces, described in~\cite{Sageev95,Chepoi2000}, that separate $\mathbf X$ into exactly two complementary components.  In this sense, CAT(0) cube complexes are ``generalized trees''.  Another salient property of a tree is that every geodesic triangle is a geodesic tripod.  Generalizing this leads to the notion of a \emph{median graph}, and in fact cube complexes and median graphs are essentially equivalent generalizations of trees: the 1-skeleton of $\mathbf X$ is a median graph, and any median graph is the 1-skeleton of a uniquely-determined cube complex.

It is natural to focus on how the hyperplanes interact.  The \emph{contact graph} $\contact X$ of $\mathbf X$, introduced in~\cite{HagenQuasiArb}, is the intersection graph of the family of hyperplane carriers, and was shown in that paper to be quasi-isometric to a tree.  $\contact X$ naturally generalizes the \emph{crossing graph} $\crossing X$ of $\mathbf X$, which is the intersection graph of the hyperplanes themselves, and is a spanning subgraph of $\contact X$.  Now, whenever $\mathbf X$ decomposes as a nontrivial product, $\contact X$ decomposes as a join and is thus a quasi-tree simply by virtue of being bounded.  However, there are numerous infinite examples for which $\contact X$ is bounded and $\mathbf X$ does not have a cubical product structure; a natural example is the smallest subcomplex of the standard tiling of $\reals^2$ by 2-cubes that contains $\{(x_1,x_2):x_2<x_1\}$, whose contact graph has diameter 3.  The motivating question in this paper is: \emph{Under what geometric conditions on $\mathbf X$ is $\contact X$ (or $\crossing X$) bounded, and what relevance does this have for groups acting on cube complexes?}

Studying this question led to the introduction of the \emph{simplicial boundary} $\simp\mathbf X$ of $\mathbf X$, a simplicial complex ``at infinity'' encoding much information about the non-hyperbolic behavior of $\simp\mathbf X$.  Our goal in this paper is to understand the simplicial boundary and to indicate some of its uses in cubical geometry.

We now summarize our results, which are related to this question, and briefly discuss our methods.  Since the key objects in this paper are defined independently of the CAT(0) geometry, we mostly work in the 1-skeleton, taking advantage of the fact that it is a median graph (see the discussions in~\cite{BandeltChepoi_survey,Chepoi2000,EppsteinFalmagneOvchinnikov,ImKl,Isbell,NicaCubulating04,Roller98,vandeVel_book}). A large part of our use of the higher-dimensional cubes is in the application of disc diagram techniques, introduced by Casson and developed by Sageev in his thesis and by Wise in~\cite{WiseIsraelHierarchy}.

\vspace{2mm}
\textbf{The simplicial boundary.}  In view of Theorem~B below, it is natural to attempt to detect unboundedness of $\contact X$ by finding a geodesic ray in $\mathbf X$ that contains points far from each hyperplane and does not lie in a cubical ``flat sector''.  The \emph{simplicial boundary} of a CAT(0) cube complex $\mathbf X$, whose construction is this paper's main innovation, is designed to keep track of infinite bounded subgraphs of $\contact X$ corresponding to unbounded flat subcomplexes.

The set $\mathcal W(\gamma)$ of hyperplanes that cross the geodesic ray $\gamma$ has several salient properties which are abstracted in the definition of a \emph{unidirectional boundary set (UBS)} of hyperplanes in Section~\ref{sec:ubs}.  The simplices of $\simp\mathbf X$ are defined to be equivalence classes of UBSs, where two UBSs are equivalent if their symmetric difference is finite.  The definition of $\simp\mathbf X$ is enabled by Theorem~\ref{thm:structureofboundarysets}, which says that each UBS decomposes as the disjoint union of \emph{minimal} UBSs in an essentially unique way.  Crucially, the subgraph of $\crossing X$ generated by a UBS corresponding to an $n$-simplex of $\simp\mathbf X$ is a spanning subgraph of an infinite complete $(n+1)$-partite graph.  When $n\geq 1$, this subgraph is sufficiently connected to ensure that the corresponding subgraph of $\contact X$ is bounded.

Theorem~\ref{thm:boundaryproperties} establishes basic properties of $\simp\mathbf X$: it is a flag complex, each simplex is contained in a finite maximal simplex, and the relationship between UBSs and $\contact X$ described above, together with the proof of~\cite[Theorem~7.6]{HagenQuasiArb}, implies that $\simp\mathbf X$ is totally disconnected when $\mathbf X$ is unbounded and hyperbolic.

The definition of a UBS is motivated by sets of hyperplanes of the form $\mathcal W(\gamma)$, but the \emph{eighth-flat} shown in Figure~\ref{fig:raybound} shows that not all UBSs are sets of hyperplanes dual to some ray.  However, Theorem~\ref{thm:visiblesimplex} says that maximal simplices of $\simp\mathbf X$ are \emph{visible}, i.e. represented by UBSs determined by geodesic rays.  The property of being \emph{fully visible} -- every simplex of $\simp\mathbf X$ is visible -- enables additional conclusions.  In particular, by Theorem~\ref{thm:visiblefracflat}, if $\mathbf X$ is fully visible, then each $n$-simplex of $\simp\mathbf X$ records the existence of an isometrically embedded copy of the standard tiling of $[0,\infty)^{n+1}$ by $(n+1)$-dimensional Euclidean unit cubes, and thus corresponds to a genuine infinite $(n+1)$-partite subgraph of $\crossing X$.

This is one of several analogues between $\simp\mathbf X$ -- which is an invariant of the median graph $\mathbf X^{(1)}$ -- and the Tits boundary of $\mathbf X$.  Another similarity is given by Theorem~\ref{thm:productsandjoins}, which equates the existence of product decompositions of $\mathbf X$ with simplicial join decompositions of $\simp\mathbf X$.  While this is very similar to the corresponding theorem about spherical join decompositions of the Tits boundary of a CAT(0) space (Theorem~II.9.24 of~\cite{BridsonHaefliger}), there are important differences in the hypotheses under which these results hold, reflecting the fact that $\mathbf X$ generally contains many more combinatorial geodesics than CAT(0) geodesics.

However, $\simp\mathbf X$ can be endowed with a CAT(1) metric by declaring each simplex to be a right-angled spherical simplex, and, if $\mathbf X$ is fully visible and satisfies an additional technical condition, then this CAT(1) realization of $\simp\mathbf X$ is isometric to the Tits boundary $\partial_T\mathbf X$, i.e. $\simp\mathbf X$ is a triangulation of $\partial_T\mathbf X$.  In Section~\ref{sec:titscompare}, we give conditions ensuring that $\simp\mathbf X$ is an isometric triangulation of the Tits boundary, explain the relationship between the two boundries in general, and give examples showing that the two boundaries are in general slightly different.

\vspace{2mm}
\textbf{Boundedness of $\contact X$.}  The motivating question is addressed in Section~\ref{sec:boundedcontact}, where we prove:

\begin{thmA}[Corollary~\ref{cor:boundedcontactsimple}, Theorem~\ref{thm:boundedcontact}]\label{thm:thmD}
Let $\mathbf X$ be a strongly locally finite, essential, one-ended CAT(0) cube complex.  Then
\[\diam(\contact X)\leq\diam(\crossing X)\leq\diam((\simp\mathbf X)^1)\leq 2(\diam(\crossing X)-1),\]
so that $\crossing X$ has finite diameter if and only if the 1-skeleton of $\simp\mathbf X$ is bounded, and, if $\simp\mathbf X$ is bounded, then $\contact X$ is bounded.  If, in addition, $\mathbf X$ has finite degree, then $\contact X$ is bounded if and only if $\simp\mathbf X$ has bounded 1-skeleton.  Finally, if $\mathbf X$ is a strongly locally finite CAT(0) cube complex and $\contact X$ is bounded, then for each isolated 0-simplex $v$ of $\simp\mathbf X$, $v$ lies in the image of $\simp H$ for some hyperplane $H$.
\end{thmA}

If $\crossing X$ is bounded, $\mathbf X$ decomposes as an \emph{iterated pseudoproduct}, a notion generalizing that of a product and providing a nexus between $\crossing X$ and $\simp\mathbf X$.  If $\gamma$ is a combinatorial geodesic ray in $\mathbf X$, then the \emph{projection} $\gproj\gamma$ of $\gamma$ to the contact graph is the embedded ray in $\contact X$ that traverses the vertices corresponding to hyperplanes crossing $\gamma$, in the order that they cross $\gamma$.  Projecting rays to $\contact X$ allows one to detect unboundedness of $\contact X$:

\begin{thmB}[Theorem~\ref{thm:trichotomy1}, Theorem~\ref{thm:trichotomy2}]\label{thm:thmB}
Let $\gamma$ be a combinatorial geodesic ray, let $\gproj\gamma\subseteq\contact X$ be its projection to the contact graph, and let $\Lambda(\gamma)$ be the full subgraph of $\contact X$ generated by $\gproj\gamma$.  Suppose there exists $B<\infty$ such that, if $K_{p,p}\subseteq\Lambda(\gamma)$, then $p\leq B$.  Suppose, moreover, that for all $R\geq 0$, there exists $T$ such that for all hyperplanes $H$, any subpath of $\gamma$ that lies in $N_R(N(H))$ has length at most $T$.  Then the inclusion $\gproj\gamma\hookrightarrow\contact X$ is a quasi-isometric embedding.  Moreover, if $\diam_{_{\contact X}}(\gproj\gamma)<\infty$, then either $\gamma$ lies in an isometrically embedded \emph{eighth-flat} or $\gamma$ lies in a uniform neighborhood of some hyperplane.
\end{thmB}

Theorem~B is used in conjunction with essentiality and local finiteness to show that rank-one combinatorial geodesic rays correspond to 0-simplices of the boundary, and is of interest in its own right.  The remainder of Theorem~A uses the duality between cube complexes and wallspaces, disc diagram arguments, and basic facts about $\simp\mathbf X$.

\vspace{2mm}
\textbf{Combinatorial divergence.}  Recently, Behrstock and Charney studied divergence in CAT(0) cube complexes associated to right-angled Artin groups.  Motivated by their results, we examine, in Section~\ref{sec:divergence}, combinatorial divergence and divergence of geodesics in the strongly locally finite CAT(0) cube complex $\mathbf X$, without reference to any group action.  For example, we prove:

\begin{thmC}[Theorem~\ref{thm:lineardivergence}]\label{thm:thmC}
Let $\mathbf X$ be strongly locally finite.  Then $\simp\mathbf X^{(1)}$ is bounded if and only if $\mathbf X$ has weakly uniformly linear divergence.
\end{thmC}

Theorem~C means that there exists a constant $A=A(\mathbf X)$ such that for all combinatorial geodesic rays $\gamma,\gamma'$ with common initial 0-cube, and for all $r\geq 0$, there exists a path $P$ joining $\gamma(r),\gamma'(r)$ that avoids the ball of radius $r-1$ about $\gamma(0)$ and has length bounded by $Ar+B$, where only $B$ depends on $\gamma,\gamma'$.  The situation is similar for divergence of combinatorial geodesics:

\begin{corD}[Corollary~\ref{cor:superlinear1}]\label{cor:corD}
Let $\alpha:\reals\rightarrow\mathbf X$ be a bi-infinite combinatorial geodesic representing the simplices $v^{\pm}$ of $\simp\mathbf X$.  Then $v^-$ and $v^+$ lie in different components of $\simp\mathbf X$ if and only if the divergence of the geodesic $\alpha$ is super-linear.
\end{corD}

Simplicial boundary considerations can apparently be used to make more refined statements about divergence functions of cube complexes in terms of the simplicial structure of $\simp\mathbf X$.

\vspace{2mm}
\textbf{Rank-one isometries and divergence.}  The slightly restricted notion of divergence used in stating the above results aligns with the standard definition (i.e. that introduced in~\cite{GerstenDivergence,GerstenDivergence2}) when $\mathbf X$ is combinatorially geodesically complete and admits a geometric group action~\cite[Lemma~7.3.5]{HagenPhD}.  In Section~\ref{sec:groupdiverge}, we observe the following:

\begin{thmE}\label{thm:thmE}
Let $G$ act properly, cocompactly, and essentially on the combinatorially geodesically complete CAT(0) cube complex $\mathbf X$.  If $\contact X$ is bounded, then $G$ has linear divergence if and only if $\contact X$ decomposes as the join of two infinite proper subgraphs (i.e. $\simp\mathbf X$ decomposes as the simplicial join of two proper subcomplexes).  Otherwise, $G$ has at least quadratic divergence.
\end{thmE}

Corollary~5.4 of~\cite{BehrstockCharney} characterizes right-angled Artin groups with linear divergence: they are exactly those for which the defining graph decomposes as a nontrivial join.  It is evident that the defining graph $\Theta$ of the right-angled Artin group $G_{\Theta}$ decomposes as a join if and only if the contact graph (equivalently, simplicial boundary) of the universal cover of the Salvetti complex of $G_{\Theta}$ decomposes as a join.  In this sense, either the contact graph or the simplicial boundary is the analogue of the defining graph of a right-angled Artin group needed to extend their characterization to cubulated groups that are not RAAGs.

A key ingredient in the proof of Theorem~\ref{thm:thmE} is the \emph{rank-rigidity theorem} of Caprace-Sageev~\cite[Theorem~A,Theorem~B]{CapraceSageev}, and in Section~5, we translate this result into the language of the contact graph and the simplicial boundary.  Roughly speaking, the dichotomy established by rank-rigidity -- if $\mathbf X$ is not a product, then a group acting sufficiently nicely has a rank-one element -- corresponds to the fact that $\contact X$ is either bounded or unbounded.  We also note, in Theorem~\ref{thm:rankonecontact}, that whether an isometry is rank-one is detected by the contact graph.\\

\vspace{-2mm}
\textbf{Acknowledgements.}  I am grateful to Jason Behrstock, Ruth Charney, Victor Chepoi, Michah Sageev, and  
Dani Wise for helpful discussions about cube complexes, boundaries, and divergence.  I enthusiastically thank 
the anonymous referee for making valuable comments and encouraging me to investigate in greater detail the 
material discussed in Section~\ref{sec:titscompare}.  Finally, I thank Mike Carr for extremely helpful comments 
on that section, and Abdul Zalloum for a correction to Lemma 3.18.

\section{Preliminaries}\label{sec:prelim}
We refer the reader to~\cite{BehrstockCharney,CapraceSageev,ChatterjiNiblo04,Chepoi2000,HagenPhD,HaglundSemisimple,Roller98,Sageev95,WiseIsraelHierarchy} for more comprehensive discussions of the background given in this section.
\vspace{-0.3em}
\subsection{CAT(0) cube complexes and contact graphs}\label{sec:cubeprelim}
Throughout this paper, $\mathbf X$ is a CAT(0) cube complex, i.e. a simply connected CW-complex whose cells are Euclidean cubes of the form $[-\frac{1}{2},\frac{1}{2}]^d$ with $0\leq d<\infty$, attached so that any two distinct cubes either have empty intersection or intersect in a common face.  Also, the link of each 0-cube is a flag complex.  The \emph{dimension} of $\mathbf X$ is the supremum of the set of dimensions of cubes of $\mathbf X$.  The \emph{degree} of $x\in\mathbf X^0$ is the number of vertices in its link, and the \emph{degree} of $\mathbf X$ is the supremum of the set of degrees of its 0-cubes.  Note that the degree of $\mathbf X$ is at least the dimension.

\subsubsection{Hyperplanes}
For $d\geq 1$, a \emph{midcube} of the cube $[-\frac{1}{2},\frac{1}{2}]^d$ is a subspace obtained by restricting exactly one coordinate to 0.  A \emph{hyperplane} $H$ of $\mathbf X$ is a connected subspace such that, for each cube $c$, either $H\cap c=\emptyset$ or $H\cap c$ is a midcube of $c$.  The \emph{carrier} $N(H)$ of $H$ is the union of all closed cubes $c$ such that $H\cap c\neq\emptyset$.  The ambient CAT(0) cube complex is assumed to be countable and to contain at least two distinct 0-cubes, and hence at least one hyperplane.

In~\cite[Theorems~4.10,4.11]{Sageev95}, Sageev proved the following for each hyperplane $H$: $H$ is a CAT(0) cube complex with $\dimension(H)<\dimension(\mathbf X)$; $H$ is \emph{2-sided} in the sense that $N(H)\cong H\times[-\frac{1}{2},\frac{1}{2}]$; the complement $\mathbf X-H$ has exactly two nonempty components, $H^+$ and $H^-$.  These are the \emph{halfspaces} associated to $H$.  These facts were established independently, from a slightly different viewpoint, by Chepoi in~\cite{Chepoi2000}.

Subspaces $A,B\subset X$ lying in distinct halfspaces associated to $H$ are \emph{separated} by $H$.  In particular, a 1-cube $c$ whose endpoints are separated by $H$ is \emph{dual to $H$} and $H$ is \emph{dual to $c$}.  The relation ``dual to the same hyperplane'' on the set of 1-cubes coincides with the \emph{Djokovi\'{c}-Winkler relation}, i.e. the transitive closure of the relation which contains $(c,c')$ if the 1-cubes $c$ and $c'$ lie on opposite sides of a 4-cycle in $\mathbf X^1$.

The distinct hyperplanes $H$ and $H'$ \emph{contact} if no third hyperplane $H''$ separates $H$ from $H'$.  $H$ and $H'$ contact if and only if $N(H)\cap N(H')\neq\emptyset$, and this can happen in one of two ways.  First, $H$ and $H'$ can \emph{cross}, i.e. $N(H)\cap N(H')$ contains a 2-cube $s$ whose boundary 4-cycle $c_1c_2c_1'c_2'$ has the property that $c_1$ and $c'_1$ are dual to $H$ and $c_2$ and $c'_2$ to $H'$.  Equivalently, $H$ and $H'$ cross if the four \emph{quarter-spaces} $H^+\cap(H')^+,H^+\cap(H')^-,H^-\cap(H')^+,H^-\cap(H')^-$ are all nonempty.  If $H,H'$ contact and do not cross, then they \emph{osculate}: there are 1-cubes $c,c'$, dual to $H,H'$ respectively, such that $c$ and $c'$ have a common 0-cube and the path $cc'$ does not lie on the boundary path of any 2-cube.  We use the notation $H\bot H'$ to mean that $H$ and $H'$ cross and $H\coll H'$ to mean that they contact.

If $H_1,\ldots,H_n$ are distinct pairwise-crossing hyperplanes, then there is an $n$-cube $c\subset\bigcap_{i=1}^nN(H_i)$ containing a 1-cube dual to each $H_i$.  We say that $c$ is \emph{dual} to the collection $\{H_i\}_{i=1}^n$.  If $\{H_i\}$ is a maximal family of pairwise-crossing hyperplanes, then $c$ is unique.  More generally, if $H_1,\ldots,H_n$ pairwise-contact, then there is at least one 0-cube $c$ with an incident 1-cube dual to each $H_i$.  Thus $\dimension\mathbf X$ is the cardinality of a largest family of pairwise-crossing hyperplanes, and the degree of $\mathbf X$ is the cardinality of a largest family of pairwise-contacting hyperplanes.

\subsubsection{Strong local finiteness}
The locally finite CAT(0) cube complex $\mathbf X$ is \emph{strongly locally finite} if every family of pairwise-crossing hyperplanes is finite.    Note that $\mathbf X$ can be strongly locally finite with infinite dimension and degree.  There are CAT(0) cube complexes that are locally finite but not strongly locally finite: these contain infinite families of pairwise-crossing hyperplanes corresponding to infinite cubes ``at infinity''.  The cube complex dual to the wallspace shown in~\cite[Figure~5]{HruskaWiseAxioms} is of this type.  We often assume that $\mathbf X$ contains no infinite family of pairwise-crossing hyperplanes without hypothesizing strong local finiteness.

\subsubsection{Metrics, geodesics, and hyperplane-equivalence}
Gromov showed that, if $\dimension\mathbf X<\infty$, then by treating each cube as a Euclidean unit cube, one obtains a piecewise-Euclidean geodesic metric $\mathfrak d_{_{\mathbf X}}:\mathbf X\times\mathbf X\rightarrow[0,\infty)^2$ such that $(\mathbf X,\mathfrak d_{_{\mathbf X}})$ is a CAT(0) space~\cite{Gromov87}; more general results of this type were proved by Bridson~\cite{BridsonThesis}.  Recently, Leary showed that this CAT(0) metric also exists when $\mathbf X$ is infinite-dimensional~\cite[Theorem~41]{LearyInfiniteCubes}.

We almost always use a different metric.  A \emph{combinatorial path} in $\mathbf X$ is a continuous map $P:I\rightarrow\mathbf X^1$ of an interval $I$ such that $P(i)\in\mathbf X^0$ for each $i\in I\cap\integers$ and $P$ maps each $[n,n+1]$ homeomorphically to a 1-cube.  The \emph{path-metric} $d_{_{\mathbf X}}:\mathbf X^1\times\mathbf X^1\rightarrow[0,\infty)$ is defined by letting $d_{_{\mathbf X}}(x,y)$ be the length of a shortest combinatorial path joining $x$ to $y$.  If $x,y\in\mathbf X^0$ are 0-cubes, then $d_{_{\mathbf X}}(x,y)$ counts the hyperplanes separating $x$ from $y$.

Haglund showed that $d_{_{\mathbf X}}$ extends to all of $\mathbf X$ in such a way that $d_{_{\mathbf X}}$ restricts to the $\ell^1$ metric on each cube and to the path-metric on $\mathbf X^1$.  For $x,y\in\mathbf X$, $d_{_{\mathbf X}}(x,y)$ is the infimum of the lengths of paths joining $x$ to $y$ that are parallel to the 1-skeleton~\cite{HaglundSemisimple}.  We usually work in the 1-skeleton and consider only combinatorial paths, and unless stated otherwise, we use this ``cubical $\ell_1$'' metric $d_{_{\mathbf X}}$, restricted to the 1-skeleton.  It follows from~\cite[Lemma 2.2]{CapraceSageev} that, if $\dimension\mathbf X<\infty$, then $(\mathbf X,\mathfrak d_{_{\mathbf X}})$ is quasi-isometric to $(\mathbf X,d_{_{\mathbf X}})$.  This fact is sometimes useful for applying facts about the CAT(0) metric in combinatorial contexts and vice versa.

We denote by $N_R(A)$ the smallest subcomplex containing the $R$-neighborhood of $A^0$ in $\mathbf X^0$, where $A$ is a subcomplex.  The subcomplex $Y\subset\mathbf X$ is isometrically embedded if and only if $Y^1\hookrightarrow\mathbf X^1$ is an isometric embedding.  Given any connected subcomplex $Y$, the hyperplane $H$ \emph{crosses} $Y$ if and only if $Y\cap H^+$ and $Y\cap H^-$ are nonempty, and $Y$ is isometrically embedded if and only if $Y\cap H$ is connected for each hyperplane $H$ that crosses $Y$.  Thus $P:I\rightarrow\mathbf X^1$ is a geodesic if and only if $P(I)$ contains at most one 1-cube dual to each hyperplane.  In this case, we also denote the isometric subcomplex $P(I)$ by $P$.  If $I$ is finite, then $P$ is a \emph{(combinatorial) geodesic segment}.  If $I=\reals$, then $P$ is a (bi-infinite) \emph{geodesic}.  Otherwise, $P$ is a \emph{geodesic ray}.

The connected full subcomplex $Y$ is \emph{convex} if for each pair of hyperplanes $H,H'$ that cross $Y$, if $H\coll H'$ then $N(H)\cap N(H')\cap Y\neq\emptyset$ and, if in addition $H\bot H'$, then $H\cap Y$ and $H'\cap Y$ cross.  ($Y$ is \emph{full} if it contains every cube $c$ of $\mathbf X$ whose 1-skeleton appears in $Y$.  \emph{Convex subcomplex} in this paper is understood to denote a full subcomplex with the preceding property.  Convex subcomplexes are therefore CAT(0).)  This can be proved using the fact that $H$ and $N(H)$ are convex in $\mathbf X$, which was proved by Sageev~\cite{Sageev95}, and which is true for either metric since the connected subcomplex $Y$ is CAT(0)-convex if and only if it is combinatorially convex~\cite{HaglundSemisimple}.  From this characterization of convexity, or from the median graph viewpoint, it follows that $\mathbf X$ has the \emph{Helly property}: if $Y_1,\ldots,Y_n$ are pairwise-intersecting convex subcomplexes of $\mathbf X$, then $\cap_{i=1}^nY_n\neq\emptyset$.

Henceforth, $\mathcal W$ will denote the set of all hyperplanes in $\mathbf X$, and $\mathcal W^{\pm}$ the set of all halfspaces.  For any isometrically embedded subcomplex $Y$, the set of hyperplanes that cross $Y$ is $\mathcal W(Y)$.  Occasionally, we shall use the notation $\mathcal W(\alpha)$, where $\alpha$ is a CAT(0) geodesic path in $\mathbf X$, to denote the set of hyperplanes crossing $\alpha$.

\begin{defn}[Hyperplane-equivalent, consuming, almost-equivalent]\label{defn:almostequivalent}
The isometric subcomplex $Y_1$ \emph{consumes} the isometric subcomplex $Y_2$ if $\mathcal W(Y_2)\subseteq\mathcal W(Y_1)$.  If $Y_1$ consumes $Y_2$ and $Y_2$ consumes $Y_1$, then $Y_1$ and $Y_2$ are \emph{hyperplane-equivalent}.  If $|\mathcal W(Y_1)\triangle\mathcal W(Y_2)|<\infty$, then $Y_1$ and $Y_2$ are \emph{almost-equivalent}.
\end{defn}

Almost-equivalence is an equivalence relation, and we shall mostly be concerned with almost-equivalence of combinatorial geodesic rays.

\subsubsection{Medians}
Chepoi proved that the class of graphs that are 1-skeleta of CAT(0) cube complexes is precisely the class of \emph{median graphs}~\cite{Chepoi2000}.  By viewing hyperplanes as Djokovi\'{c}-Winkler classes of 1-cubes, he also established the properties of hyperplanes described above.  A synopsis of the correspondence between these viewpoints appears in~\cite[Section~3]{ChepoiHagen}.

\begin{defn}[Median graph]\label{defn:mediangraph}
The graph $M$ is a \emph{median graph} if for all triples $x,y,z\in M$ of distinct vertices, there exists a unique vertex $m=m(x,y,z)$, the \emph{median} of $x,y,z$, such that $d(x,y)=d(x,m)+d(m,y),\,d(x,z)=d(x,m)+d(m,z),$ and $d(y,z)=d(y,m)+d(m,y).$
\end{defn}

A detailed discussion of the relationship between median algebras, median graphs, and cube complexes appears in~\cite{Roller98}, and the construction of a median graph from a \emph{wallspace} is lucidly discussed in~\cite{NicaCubulating04}.  It is crucial to observe that the invariants of $\mathbf X$ discussed in this paper -- the crossing graph, the contact graph, and the simplicial boundary -- are invariants of $\mathcal W$ and hence of the median graph $\mathbf X^1$ only.

\subsubsection{Essentiality and compact indecomposability}\label{sec:essentialprelim}
Following~\cite{Sageev95}, we define the hyperplane $H$ to be \emph{essential} if each of $H^+$ and $H^-$ contains points arbitrarily far from $H$.  If every hyperplane of $\mathbf X$ is essential, then $\mathbf X$ is \emph{essential}.  $\mathbf X$ is \emph{compactly decomposable} if there exists a compact, convex subcomplex $K$ such that $\mathbf X-K$ is disconnected.  Otherwise, $\mathbf X$ is \emph{compactly indecomposable}.  If $\mathbf X$ is essential and locally finite, then $\mathbf X$ is compactly indecomposable if and only if $\mathbf X^1$ is one-ended.

\subsubsection{Cubulating wallspaces}
Sageev showed that a $G$-cube complex can be constructed from an appropriately chosen family of \emph{codimension-1} subgroups of a group $G$~\cite[Theorem~3.1]{Sageev95}.  In the more general context of a \emph{wallspace} provided by Haglund and Paulin~\cite{HaglundPaulin98}, one can construct a CAT(0) cube complex in essentially the same way. We summarize the approach taken in~\cite{ChatterjiNiblo04}.

A \emph{wallspace} consists of an \emph{underlying set} $\mathcal S$ equipped with a set $\mathbb W$ of \emph{walls}.  A \emph{wall} $W\in\mathbb W$ is a pair $(\mathfrak h(W),\mathfrak h^*(W))$ of disjoint nonempty subsets of $\mathcal S$ whose union is all of $\mathcal S$.  The wall $W$ \emph{separates} $s,s'\in\mathcal S$ if $s\in\mathfrak h(W)$ and $s'\in\mathfrak h^*(W)$.  The walls $W,W'$ \emph{cross} if each of the four possible intersections of their halfspaces is nonempty.  We require that any two elements of $\mathcal S$ be separated by only finitely many walls.  For $s,s'\in\mathcal S$, denote by $\#(s,s')$ the number of walls separating $s$ and $s'$.

Let $\mathbb H$ be the set of halfspaces, and let $\pi:\mathbb H\rightarrow\mathbb W$ be the map that sends $\mathfrak h(W)$ and $\mathfrak h^*(W)$ to $W$ for each $W\in\mathbb W$.  The \emph{dual cube complex} $\mathbf X(\mathcal S,\mathbb W)=\mathbf X$ is formed as follows.  A \emph{consistent orientation} is a section $x:\mathbb W\rightarrow\mathbb H$ of $\pi$, i.e. a choice of exactly one halfspace for each wall, such that $x(W)\cap x'(W)\neq\emptyset$ for all walls $W,W'$.  If $W,W'$ cross, any section $x$ of $\pi$ satisfies $x(W)\cap x(W')\neq\emptyset$.  The consistent orientation $x$ is \emph{canonical} if there exists $s\in\mathcal S$ such that $s\in x(W)$ for all $W\in\mathbb W$.  Each $s\in\mathcal S$ yields a canonical consistent orientation $x_s$, by letting $x_s(W)$ be whichever of $\mathfrak h(W),\mathfrak h^*(W)$ contains $s$.  Clearly, $x_{s}$ and $x_{s'}$ differ exactly on the finite set of walls $W$ that separate $s$ from $s'$.

The 0-skeleton $\mathbf X^0$ consists of all consistent orientations $x$ such that $x$ differs from some (any) $x_s$ on finitely many walls.  The 0-cubes $x,y$ are joined by a 1-cube if and only if $x(W)\neq y(W)$ for a single wall $W$.  The resulting graph $\mathbf X^1$ is median and thus the 1-skeleton of a CAT(0) cube complex.  There is a bijection $\beta:\mathbb W\rightarrow\mathcal W$, and the hyperplanes $\beta(W),\beta(W')$ cross if and only if the walls $W,W'$ cross.  Similarly, $\beta(W)$ and $\beta(W')$ are separated by a hyperplane $\beta(W'')$ if and only if $W''$ separates $W$ from $W'$, and $x_s,x_{s'}$ are separated by $\beta(W)$ if and only if $s,s'$ are separated by $W$.

Consider the wallspace in which $\mathcal S$ is the 0-skeleton of a CAT(0) cube complex and $\mathbb W$ is the set of bipartitions of the 0-skeleton induced by the hyperplanes.  We view a 0-cube $x\in\mathbf X^0$ as consistent, canonical orientation of all walls.  A section $x$ of $\pi$ that is consistent but not canonical is a \emph{0-cube at infinity}.

A corollary of this construction is the existence of \emph{restriction quotients} (see e.g.~\cite{CapraceSageev}).  Let $\mathcal V\subseteq\mathcal W$ be a set of hyperplanes of $\mathbf X$.  Then there exists a cubical quotient map $\rho:\mathbf X\rightarrow\mathbf X(\mathcal V)$, where $\mathcal X(\mathcal V)$ is the cube complex dual to the wallspace whose underlying set is $\mathbf X^0$ and whose walls are induced by the set $\mathcal V$ of hyperplanes.  $\mathbf X(\mathcal V)$ is the \emph{restriction quotient} associated to $\mathcal V$.

If $H,H'\in\mathcal V$ contact in $\mathbf X$, then $\rho(H)\coll\rho(H')$ in $\mathbf X(\mathcal V)$.  The converse holds for crossings but not for osculations.  More precisely, if $\rho(H)\bot\rho(H')$, then $H\bot H'$.  However, if $H,H'\in\mathcal V$, and every hyperplane $U$ separating $H$ from $H'$ in $\mathbf X$ belongs to $\mathcal W-\mathcal V$, then $\rho(H)\coll\rho(H')$.

\subsubsection{Contact graphs and crossing graphs}\label{sec:contactgraphprelim}
The \emph{crossing graph} $\crossing X$ of $\mathbf X$ is a graph whose vertex set is the set $\mathcal W$ of hyperplanes of $\mathbf X$, and we use the same notation for a vertex as for the corresponding hyperplane.  The vertices $H$ and $H'$ are adjacent if and only if $H\bot H'$.  Equivalently, $\crossing X$ is the intersection graph of the family of hyperplanes.  $\crossing X$ appears as the \emph{transversality} graph in work of Niblo and of Roller, and an equivalent construction appears in work of Bandelt, Dress, van de Vel, and others.

The \emph{contact graph} $\contact X$ of $\mathbf X$, introduced in~\cite{HagenQuasiArb}, has vertex set $\mathcal W$, but now $H,H'$ are adjacent if and only if $H\coll H'$.  In particular, $\crossing X$ is a spanning subgraph of $\contact X$.  Equivalently, $\contact X$ is the intersection graph of the family of hyperplane-carriers in $\mathbf X$.  We denote by $H\coll H'$ the closed edge of $\Gamma$ corresponding to the contacting of $H$ and $H'$.

Observe that for each $n$-clique in $\crossing X$, there is an $n$-cube in $\mathbf X$ and for each $n$-clique in $\contact X$, there is a 0-cube in $\mathbf X$ of degree at least $n$.  Hence the dimension of $\mathbf X$ is the clique number of $\crossing X$, and the degree of $\mathbf X$ is the clique number of $\contact X$.

Let $\gamma:I\rightarrow\mathbf X$ be a combinatorial geodesic path, and for each $i$, let $H_i$ be the hyperplane dual to the 1-cube $\gamma([i,i+1])$.  Then $H_i\coll H_{i+1}$ for each $i$, and we obtain a path $\gproj(\gamma)$ in $\contact X$, called the \emph{projection of $\gamma$ to $\contact X$}, that is the concatenation of these edges.  Since $\gamma$ passes through each hyperplane at most once, $\gproj(\gamma)$ is an embedded path in $\contact X$, and spans the subgraph $\Lambda(\gamma)$ generated by $\mathcal W(\gamma)$.  This projection shows that, unlike the crossing graph, $\contact X$ is necessarily connected.

Detailed discussions of the geometric and combinatorial properties of $\contact X$ can be found in~\cite{HagenQuasiArb,HagenPhD} and~\cite{ChepoiHagen}.  Some variants and uses of contact graphs are discussed in~\cite{ChepoiNiceLabeling}.  The main geometric property of contact graphs is Theorem~4.1 of~\cite{HagenQuasiArb}, which says that there exist constants $M\geq 1,C\geq 0$, independent of $\mathbf X$, and a tree $\mathcal T=\mathcal T(\mathbf X)$, such that $\contact X$ is $(M,C)$-quasi-isometric to $\mathcal T$.

We need the following two propositions, which can easily be proved by cubulating wallspaces.  The first is essentially given in Section~2 of~\cite{CapraceSageev}, albeit not in contact graph terms, and characterizes cubical products.  For graphs $A_1,A_2$, we denote their join by $A_1\star A_2$.

\begin{prop}\label{prop:productchar}
The following are equivalent for a CAT(0) cube complex $\mathbf X$:
\begin{compactenum}
\item There exist nonempty convex subcomplexes $\mathbf X_1,\mathbf X_2\subset\mathbf X$ such that $\mathbf X\cong\mathbf X_1\times\mathbf X_2$.
\item There exist nonempty disjoint subgraphs $A_1,A_2\subset\crossing X$ such that $\crossing X\cong A_1\star A_2$.
\item There exist nonempty disjoint subgraphs $A'_1,A'_2\subset\contact X$ such that $\contact X\cong A'_1\star A'_2$, and for all $H_1\in A'_1,H_2\in A'_2$, the edge $H_1\coll H_2$ is a crossing-edge.
\end{compactenum}
\end{prop}

The second proposition requires the following notion, which also features in Section~\ref{sec:boundary}.

\begin{defn}[Inseparable set]\label{defn:inseparable}
The set $\mathcal U\subseteq\mathcal W$ is \emph{inseparable} if for all $U,U'\in\mathcal U$, every $H\in\mathcal W$ that separates $U$ from $U'$ belongs to $\mathcal U$.
\end{defn}

Given a set $\mathcal U$ of hyperplanes, let $\Lambda(\mathcal U)$ denote the full subgraph of $\contact X$ induced by $\mathcal U$.  Given an isometrically embedded subcomplex $Y\subseteq\mathbf X$, let $\Lambda(Y)=\Lambda(\mathcal W(Y))$.  Note that $\Lambda(Y)$ is the contact graph of the convex hull of $Y$, and $\Lambda(Y)\cap\crossing X$ is its crossing graph.  Note also that the subcomplex $Y$ is compact if and only if $\mathcal W(Y)$ is finite, and that the convex hull of a compact subcomplex is again compact.  The following is proved in~\cite{HagenPhD} by constructing the cube complex dual to the wallspace $\left(\mathbf X^0,\mathcal U\right)$ and showing that this complex embeds convexly in $\mathbf X$.

\begin{prop}\label{prop:recovery}
Let $\mathcal U$ be a finite inseparable set of hyperplanes in the locally finite CAT(0) cube complex $\mathbf X$.  Then there exists a compact, convex subcomplex $Y\subset\mathbf X$ such that $\mathcal W(Y)=\mathcal U$.
\end{prop}

There are other versions of Proposition~\ref{prop:recovery} in which, under additional assumptions, $\mathcal U$ is allowed to be infinite, but it is in general much more problematic to obtain a subcomplex of $\mathbf X$ from a subgraph of $\contact X$ than it is to obtain a quotient.

\subsection{Actions on CAT(0) cube complexes}\label{sec:actionprelim}
If $\mathbf X$ is locally finite, then $\mathbf X^1$ is a proper metric space.  We say that the group $G$ \emph{acts on $\mathbf X$} to mean that $G$ acts by cubical automorphisms and thus by isometries (with respect to both $d_{_{\mathbf X}}$ and $\mathfrak d_{_{\mathbf X}}$).  We shall always assume that $G$ acts \emph{nontrivially}, i.e. without a global fixed point.  $G$ acts \emph{properly} if, for any $x\in\mathbf X^0$ and any infinite sequence $(g_n\in G)_{n\geq 0}$, we have that $d_{_{\mathbf X}}(x,g_nx)\rightarrow\infty$ as $n\rightarrow\infty$.  Equivalently, $G$ acts properly if the stabilizer $G_c$ of any cube $c$ is finite, provided $\mathbf X$ is locally finite.

Let $G$ act on $\mathbf X$.  For each $W\in\mathcal W$, let $G_W$ denote the stabilizer of $W$ and note that $G_W=\stabilizer_{\mathbf X}(N(W))$.  For all $g\in G$, and for all $W\in\mathcal W$, it is evident that $gW$ is again a hyperplane, and $gW\coll gW'$ (respectively, $gW\bot gW'$) if and only if $W\coll W'$ (respectively, $W\bot W'$).  Hence $G$ acts by isometries on $\contact X$, and the stabilizer of a vertex coincides with the stabilizer of the corresponding hyperplane.

\subsubsection{Essential actions}
If $G$ acts on $\mathbf X$, then the hyperplane $H$ is \emph{$G$-essential} if for some $x\in\mathbf X^0$, each halfspace $H^{\pm}$ contains points of the orbit $Gx$ arbitrarily far from $H$.  %In~\cite{CapraceSageev}, Caprace-Sageev also discuss \emph{half-essential} and \emph{trivial} hyperplanes, but we shall not need these.
Note that, if $H$ is $G$-essential, then $H$ is essential.  If each hyperplane is $G$-essential, then $G$ acts \emph{essentially} on $\mathbf X$.  If $G$ acts essentially on $\mathbf X$, then $\mathbf X$ is essential.  The \emph{$G$-essential core} $\mathbf Y\subseteq\mathbf X$ is a (possibly empty) convex $G$-invariant subcomplex on which $G$ acts essentially.  The \emph{essential core theorem}~\cite[Proposition~3.5]{CapraceSageev} shows that if $G$ acts properly on the locally finite cube complex $\mathbf X$, either with finitely many orbits of hyperplanes or without a fixed point on the visual boundary, then the $G$-essential core is unbounded.  %This allows one to pass from a proper action with finitely many orbits of hyperplanes (respectively, no fixed point at infinity) to a proper essential action with finitely many orbits of hyperplanes (respectively, no fixed point at infinity).  For simplicity, we shall hypothesize an essential action where required.

\subsubsection{Semisimplicity and rank-one isometries}
Let $G$ act on $\mathbf X$.  The element $g\in G$ is \emph{elliptic} if $g$ stabilizes a cube $c$ and \emph{combinatorially hyperbolic} if there exists a combinatorial geodesic $\gamma:\reals\rightarrow\mathbf X^1$ and some $\tau>0$ such that $g\gamma(t)=\gamma(t+\tau)$ for all $t\in\reals$.  The path $\gamma$ is a \emph{combinatorial axis} for $g$, and $\tau$ is the \emph{translation length} of $g$.  The following result of Haglund~\cite[Theorem~1.4]{HaglundSemisimple}, is fundamental:

\begin{thm}[Semisimplicity]\label{thm:semisimpleHAGLUND}
Let $G$ act on the CAT(0) cube complex $\mathbf X$.  Then each $g\in G$ is either combinatorially hyperbolic, stabilizes a cube, or stabilizes a hyperplane $H$ and exchanges the halfspaces associated to $H$.
\end{thm}

By passing to the first cubical subdivision of $\mathbf X$, Haglund shows that one can guarantee that $g\in G$ is either elliptic or combinatorially hyperbolic.  If $\mathbf X$ is locally finite and finite-dimensional, and $g\in\Aut(\mathbf X)$ has a combinatorial geodesic axis $\alpha:\reals\rightarrow\mathbf X$, then $g$ has a CAT(0) geodesic axis $\beta:\reals\rightarrow\mathbf X$ that fellow-travels with $\alpha$ (with respect to either metric).  The converse is not quite true, but it is easy to deduce the following from Haglund's classification:

\begin{prop}\label{prop:combinatoriallyelliptichyperbolic}
Let $g\in\Aut(\mathbf X)$.  If $g$ is CAT(0) hyperbolic and $\mathbf X$ contains no infinite family of pairwise-crossing hyperplanes, then $g^k$ is combinatorially hyperbolic for some $k\in\integers$.
\end{prop}

%\begin{proof}
%Suppose that $g$ is a CAT(0) hyperbolic isometry.  Since the CAT(0) translation length of $g$ is positive, $g$ does not fix a point, and therefore does not stabilize a cube.  By Theorem~\ref{thm:semisimpleHAGLUND}, either $g$ is combinatorially hyperbolic, or there exists a hyperplane $H_0$ such that $g(H_0)=H_0$, and $g(H_0^+)=H_0^-,\,g(H_0^-)=H_0^+$.  Let $\mathbf X_1\cong H_0$ be a copy of $H_0$ bounding the carrier $N(H_0)\cong H_0\times[-\frac{1}{2},\frac{1}{2}]$.  Then $\mathbf X_1$ is a subcomplex of $\mathbf X$ stabilized by $g^2$.  Now, $g^2$ does not fix a 0-cube of $\mathbf X_1$, and thus either $g^2$ has a combinatorial axis in $\mathbf X_1\subset\mathbf X$, or $g$ exchanges the halfspaces of a hyperplane $H_1$ that crosses $H_0$.  Indeed, the former happens when $g$ acts hyperbolically on $H$, and the latter if not, by Theorem~\ref{thm:semisimpleHAGLUND}.  Therefore either $g^{2^n}$ is combinatorially hyperbolic for some $n$, or there is a maximal set $\{H_i\}_{i=0}^d$ of pairwise-crossing hyperplanes with each $H_i$ stabilized by $g$.  Indeed, such a pairwise-crossing family is finite by assumption.  Hence $g$ stabilizes a $d$-cube $c$ in $\mathbf X$, and thus fixes a point, namely the barycenter of $c$, contradicting the fact that $g$ is CAT(0)-hyperbolic.
%\end{proof}

The hyperbolic isometry $g$ of $\mathbf X$ is \emph{rank-one} if some (and hence any) CAT(0) geodesic axis for $g$ does not bound an isometrically embedded Euclidean half-plane in $\mathbf X$.

\subsection{Divergence and divergence of geodesics}\label{sec:divergeprelim}
The notion of the divergence function of a geodesic space was introduced by Gersten~\cite{GerstenDivergence,GerstenDivergence2}.  The present exposition roughly follows that in~\cite{BehrstockCharney}.  Let $(M,d)$ be a geodesic space, and let $\rho(k)=Ak-B$ with $A\in(0,1),\,B\geq 0$.  Given points $a,b,c\in M$, let $k=d(c,\{a,b\})$ and let $\diver{a}{b}{c}{\rho}=\inf_P|P|,$ where $P$ varies over all paths in $M-B_{\rho(k)}(c)$ that join $a$ to $b$, where $B_{\rho(k)}(c)$ is the open ball of radius $\rho(k)$ centered at $c$.  By convention, $\diver{a}{b}{c}{\rho}$ takes the value $+\infty$ if no such path exists.  The \emph{divergence of $M$} (with respect to $\rho$) is the function:
\[\divers{M}{\rho}(r)=\sup\left\{\diver{a}{b}{c}{\rho}\,\mid\,a,b,c\in M,\,d(a,b)\leq r\right\}.\]
Given a function $f:[0,\infty)\rightarrow[0,\infty]$, we say that $M$ has \emph{divergence at most $f$} if there exists $\rho$ such that $\divers{M}{\rho}(r)\leq f(r)$ for all $r\geq 0$.  If $\gamma:\reals\rightarrow\mathbf X$ is a bi-infinite geodesic, then the \emph{divergence of $\gamma$} is $\diver{\gamma(r)}{\gamma(-r)}{\gamma(0)}{\rho}$ for each $r\geq 0$.  If $M$ has divergence at most $f$, then each geodesic has divergence at most $f$, and if some geodesic has super-$f$ divergence, then $M$ has super-$f$ divergence.

\section{Projections of rays to the contact graph}\label{sec:contactproj}
We study subgraphs of $\contact X$ induced by projections of geodesic rays.  See~\cite{WiseIsraelHierarchy} and~\cite{HagenQuasiArb} for a discussion of disc diagrams as they are used in the proof of Theorem~\ref{thm:trichotomy1}.

\subsection{Combinatorial fractional flats and rank-one rays}\label{sec:fracflat}
A \emph{(combinatorial) flat} is a cube complex isomorphic to the standard Euclidean tiling of $\reals^2$ by 2-cubes.  A \emph{non-diagonal half-flat} is a cube complex isomorphic to the standard tiling of $\reals\times[0,\infty)$ by 2-cubes.  A \emph{non-diagonal quarter-flat} is a cube complex isomorphic to the standard tiling of $[0,\infty)^2$ by 2-cubes.

Let $f:[0,\infty)\rightarrow[0,\infty)$ be an unbounded, nondecreasing function.  The \emph{$f$-sector} is the subspace $\{(x,y)\in[0,\infty)^2\mid 0\leq y\leq f(x)\}$.  An \emph{eighth-flat} is a CAT(0) cube complex isomorphic to the smallest subcomplex of the standard tiling of $[0,\infty)^2$ containing a given $f$-sector.  The set of hyperplanes of an eighth-flat $\mathbf E$ can be partitioned into two inseparable sets, $\mathcal H$ and $\mathcal V$, such that each $H\in\mathcal H$ crosses all but finitely many $V\in\mathcal V$.  The unique combinatorial geodesic ray $\gamma$ in $\mathbf E$ that crosses every hyperplane is the \emph{top bounding ray}, and the unique ray $\gamma'$ corresponding to the $x$-axis in $[0,\infty)^2$ is the \emph{bottom bounding ray}. See Figure~\ref{fig:raybound}.  A \emph{diagonal quarter-flat} is a CAT(0) cube complex obtained from a pair $\mathbf E_1,\mathbf E_2$ of eighth-flats by identifying the bottom bounding ray of $\mathbf E_1$ with an infinite sub-ray of the bottom bounding ray of $\mathbf E_2$, using a cubical isometry.  Let $\sigma$ be a bi-infinite combinatorial geodesic in the standard tiling of $\reals^2$ that crosses each hyperplane.  The closure of a component of $\reals^2-\sigma$ is a CAT(0) cube complex called a \emph{diagonal half-flat}.  A cube complex $\mathbf F$ of one of the above types is a \emph{fractional flat}.

\begin{rem}
A non-diagonal fractional flat $\mathbf F$ is the product of two 1-dimensional complexes and thus $\diam(\contact F)=\diam(\crossing F)=2$.  For $\mathbf F$ a diagonal half-flat, this value is 3.  If $\mathbf F$ is an eighth-flat or diagonal quarter-flat, then $\contact F$ has finite diameter, and $\mathbf F$ has a co-finite subcomplex whose crossing graph has diameter 3.
\end{rem}

The combinatorial geodesic ray $\gamma:[0,\infty)\rightarrow\mathbf X$ \emph{bounds a fractional flat} if there exists a cubical isometric embedding $\mathbf E\rightarrow\mathbf X$ of an eighth-flat whose image contains $\gamma$.  This includes any case in which $\gamma$ lies in an isometric fractional flat.  Otherwise, $\gamma$ is (combinatorially) \emph{rank-one}.

\begin{figure}[h]
  \includegraphics[width=3in]{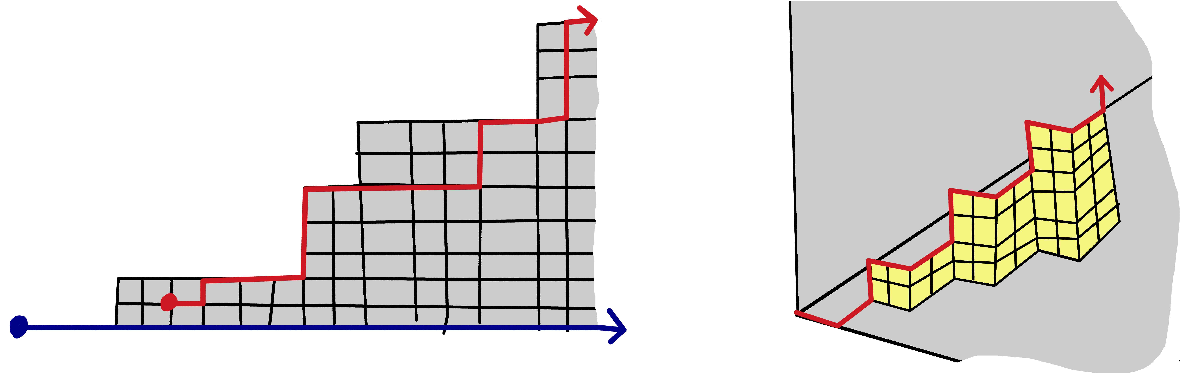}\\
  \caption{The arrowed rays at left bound an eighth-flat contained in the larger eighth-flat.  The arrowed ray at right is the top bounding ray of an eighth-flat in the standard tiling of $\reals^3$.}\label{fig:raybound}
\end{figure}

\subsection{Rays of bounded projection}\label{sec:boundedprojection}
Theorem~\ref{thm:trichotomy1} gives the only possible obstructions to a combinatorial geodesic ray $\gamma$ projecting to a quasi-geodesic ray in the contact graph.  By instead asking only that $\gamma$ project to an unbounded ray, we reach the conclusion that either $\gamma$ bounds an eighth-flat or $\gamma$ lies uniformly close to a single hyperplane, which is Theorem~\ref{thm:trichotomy2}.

\begin{thm}[Projection trichotomy I]\label{thm:trichotomy1}
Let $\mathbf X$ be a CAT(0) cube complex, and let $\gamma:\reals_+\rightarrow\mathbf X$ be a combinatorial geodesic ray.  Let $\gproj(\gamma)$ be the projection of $\gamma$ to the contact graph, and let $\Lambda(\gamma)$ be the full subgraph of $\contact X$ generated by $\gproj(\gamma)$.  Then one of the following holds:
\begin{compactenum}
\item The inclusions $\gproj(\gamma),\Lambda(\gamma)\hookrightarrow\contact X$ are quasi-isometric embeddings.
\item For all $p\geq 0$, we have $K_{p,p}\subset\Lambda(\gamma)$.
\item There exists $R\geq 0$ such that, for all $t\geq 0$, there exists a hyperplane $H$ such that $\gamma\cap N_R(H)$ contains a connected subpath of length at least $t$.
\end{compactenum}
\end{thm}

\begin{proof}
For $i\geq 0$, let $H_i$ be the hyperplane dual to $\gamma([i,i+1])$.  Note that, for all $i,j$, we have $d_{\gproj(\gamma)}(H_i,H_j)=|i-j|$.  Moreover, for all $i$, the hyperplanes $H_i$ and $H_{i+1}$ contact, so that $d_{_{\contact X}}(H_i,H_j)\leq|i-j|$.  Suppose that for all $0<\lambda^{-1}\leq 1$, there exist $i,j\geq 0$ such that
\[N=d_{_{\contact X}}(H_i,H_j)<\lambda^{-1}|i-j|-\lambda^{-1}.\]
Let $H_i=U_0\coll U_1\coll\ldots U_N=H_j$ be a shortest path in $\contact X$ joining $H_i$ to $H_j$.  We shall verify that, for some $R\leq 2B$, where $B$ is the constant from Lemma~\ref{lem:crossingbound} below, there exists a hyperplane $H$ and a subpath $\gamma''\subset\gamma$ such that $\gamma''\subset N_R(N(H))$ and $|\gamma''|\geq\lambda^{-1}$.  Since $B$ is independent of $\lambda$, and $\lambda$ can be chosen arbitrarily large, this implies that $(3)$ holds.  To this end, let $\gamma'$ be the subpath of $\gamma$ between, but not including, the 1-cubes dual to $H_i$ and $H_j$.  Note that $|\gamma'|=|i-j|-1$.

\textbf{The disc diagram $D$:}  For $0\leq m\leq N$, let $A_m\rightarrow N(U_m)$ be a combinatorial geodesic path, and let these be chosen in such a way that $A_m$ ends on the initial 0-cube of $A_{m+1}$ for each $0\leq m\leq N-1$, and the concatenation $A=A_0A_1\ldots A_N$ has the same endpoints as $\gamma'$.  Let $D\rightarrow\mathbf X$ be a disc diagram bounded by $A(\gamma')^{-1}$.  We choose $D$ subject to the following minimality constraints.  First, $D$ has minimal area among all disc diagrams with boundary path $A(\gamma')^{-1}$.  Second, the paths $A_m$ are chosen, subject to the constraint that each $A_m\rightarrow N(U_m)$, in such a way that the area of $D$ is as small as possible and, among such minimal-area $D$, the length of $A$ is as short as possible.  Finally, the path $U_0\coll\ldots\coll U_N$ is chosen among all $\contact X$-geodesics joining $H_i$ to $H_j$ so that the resulting $D$ has area as small as possible.

Let $K$ be a dual curve in $D$.  As illustrated in Figure~\ref{fig:diagramD}, $K$ must travel from $A$ to $\gamma'$.  Indeed, $K$ cannot have two ends on $\gamma'$, since $\gamma'$ is a geodesic segment.  Likewise, if $K$ travels from $A_m$ to $A_m$, for some $m$, then $A_m$ is not a geodesic, a contradiction.  If $K$ travels from $A_m$ to $A_{m+1}$, then we can employ hexagon moves (see e.g.~\cite{WiseIsraelHierarchy}), and remove spurs, to reduce the area of $D$, and then the length of $A$, without affecting our path in $\contact X$.  See the discussion of diagrams with \emph{fixed carriers} in Section~2 of~\cite{HagenQuasiArb}.  Hence, if $K$ has two ends on $A$, then $K$ travels from $A_m$ to $A_k$, with $|m-k|\geq2$.  Let $U$ be the hyperplane to which $K$ maps.  If $|m-k|>2$, then
\[H_i=U_0\coll U_1\coll\ldots\coll U_m\bot U\bot U_k\coll\ldots\coll U_N=H_j\]
is a path in $\contact X$ of length less than $N$, a contradiction.  If $m=k+2$, then replacing $U_{k+1}$ by $U$ results in path from $H_i$ to $H_j$ whose minimal-area diagram is a proper subdiagram of $D$.  Indeed, as in Figure~\ref{fig:diagramD}, we can replace the part of $A$ between and including the 1-cubes dual to $K$ by a path in the carrier of $K$ in $D$.

\begin{figure}[h]
  \includegraphics[width=4in]{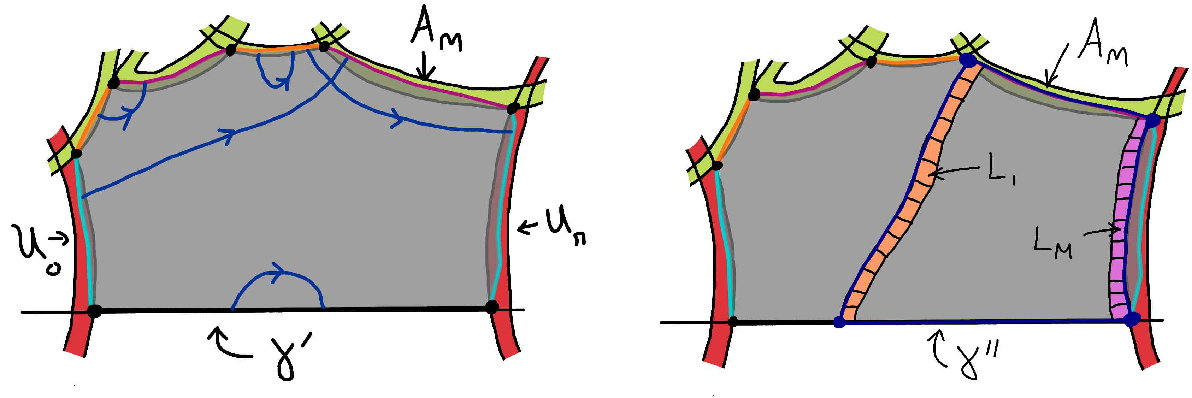}\\
  \caption{A heuristic picture of $D$, showing the hyperplane-carriers to which its boundary path maps. At left are impossible dual curves: the single-arrowed dual curves are illegal, and the crossing pair is illegal.  At right are the dual curves $L_1$ and $L_M$ and their carriers.}\label{fig:diagramD}
\end{figure}

Next, note that if $K,K'$ are dual curves, both of which emanate from the path $A_m$, then $K$ and $K'$ do not cross, for otherwise they intersect in a 2-cube $s$ of $D$ mapping to a 2-cube $\bar s$ of $N(U_m)$, and we can employ hexagon moves to pass to a lower-area diagram.

Thus $|A|=|i-j|-1$, and hence there exists $m$ with
\[M=|A_m|\geq\frac{|i-j|-1}{N}>\lambda^{-1}.\]
Let $E$ be the subdiagram of $D$ bounded by $A_m$, the carriers of the dual curves $L_1,L_M$ dual to the initial and terminal 1-cubes of $A_m$, and the subpath $\gamma''$ of $\gamma'$ between and including the 1-cubes dual to $L_1$ and $L_M$.  Then every dual curve in $E$ travels from $A_m$ to $\gamma''$ -- there are $M$ of these -- from $L_1$ to $L_M$, or from $\gamma''$ to $L_1$ to $L_M$.  See Figure~\ref{fig:diagramE}.  Indeed, no dual curve $K$ can travel from $A_m$ and cross $L_M$ (or $L_1$), for otherwise $K$ and $L_M$ ($L_1$) would cross in $D$, which is impossible since both emanate from $A_m$.

\begin{figure}[h]
  \includegraphics[width=1.75in]{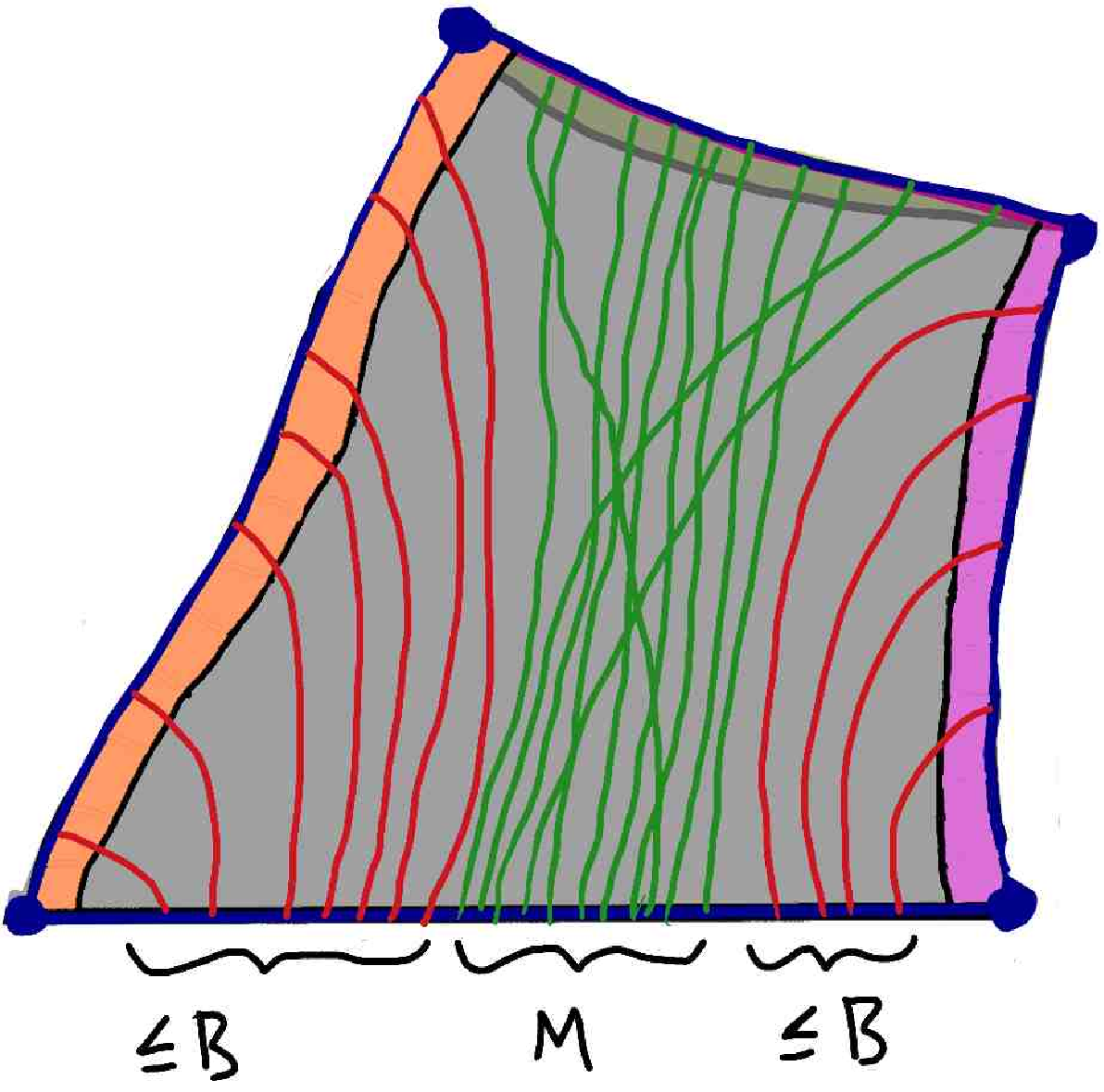}\\
  \caption{The diagram $E$ has three types of dual curve.}\label{fig:diagramE}
\end{figure}

Suppose that there exists a constant $B<\infty$ such that, for all $i',j'\geq 0$, the hyperplanes $H_{i'}$ and $H_{j'}$ cannot cross if $|i'-j'|>B$.  Hence $M\leq|\gamma''|\leq M+2B$: the first inequality follows from the fact that each dual curve emanating from $A_m$ ends on $\gamma''$, while the second follows from the fact that at most $2B$ dual curves emanating from $\gamma''$ can cross $L_1$ or $L_M$.

Let $y\in\gamma''$ be a 0-cube, and let $K$ be a dual curve in $E$ emanating from one of the 1-cubes of $\gamma''$ containing $y$.  If $K$ ends on $A_m$, then there is a path in $E$ of length at most $|K|$ joining $y$ to a 0-cube of $A_m$.  But every dual curve $L$ in $E$ that crosses $K$ has no end on $A_m$, since distinct dual curves emanating from $A_m$ do not cross.  Thus $L$ crosses $L_1$ or $L_M$, so that $|K|\leq |L_1|+|L_M|$. Therefore, $d_{_{\mathbf X}}(y,U_m)\leq |L_1|+|L_M|$.  If $K$ ends on $L_1$ (or $L_M$), then $y$ is at distance at most $B$ from $L_1$ (say), and thus $d_{_{\mathbf X}}(y,U_m)\leq B+|L_1|$.

But every dual curve crossing $L_1$ crosses $\gamma''$, and there are at most $B$ of these, since $L_1$ maps to a hyperplane crossing $\gamma''\subset\gamma$.  Hence $|L_1|+|L_M|\leq 2B$.  Therefore, $\gamma''$ is a path of length at least $M>\lambda^{-1}$ lying in the $2B$-neighborhood of $N(U_m)$.  Since $B$ is independent of $\lambda$, we have reached conclusion~$(3)$, or $B$ does not exist and, by Lemma~\ref{lem:crossingbound}, either~$(2)$ or~$(3)$ holds.  We conclude that if~$(2)$ and~$(3)$ fail, then there exists $\lambda\in[1,\infty)$ such that
\[|i-j|\geq d_{_{\contact X}}(H_i,H_j)\geq\lambda^{-1}\left(|i-j|-1\right)=\lambda^{-1}d_{\gproj(\gamma)}(H_i,H_j)-\lambda^{-1}\]
for all $i,j\geq 0$, and hence $\gproj(\gamma)\hookrightarrow\contact X$ is a quasi-isometric embedding.  Since $\Lambda(\gamma)$ is a subgraph of $\contact X$ spanned by $\gproj(\gamma)$, the inclusion $\Lambda(\gamma)\hookrightarrow\contact X$ is also a q.i.e.
\end{proof}

\begin{lem}\label{lem:crossingbound}
Let $\mathbf X$ and $\gamma$ be as in Theorem~\ref{thm:trichotomy1}.  Suppose that conclusions $(2)$ and $(3)$ of Theorem~\ref{thm:trichotomy1} do not hold.  Then there exists $B<\infty$ such that, if $H_i\coll H_j$, then $|i-j|\leq B$.
\end{lem}

\begin{proof}
Since Theorem~\ref{thm:trichotomy1}.(2) does not hold, there exists $B'<\infty$ such that, if the subgraph $\Lambda(\gamma)$ of $\contact X$ generated by $\mathcal W(\gamma)$ contains $K_{p,p}$, then $p\leq B'$.  Since statement $(3)$ does not hold, for each $r\geq 0$, there exists $B''(r)\in[r,\infty)$ such that, if $\gamma'\subset\gamma$ is a connected subpath lying in $N_r(N(H))$, for any hyperplane $H$, then $|\gamma'|\leq B''(r)$.

Suppose that $H_i\coll H_j$ (hence $|i-j|\geq 1$).  Let $\gamma'\subset\gamma$ be the subpath of length $|i-j|-1$ lying between the 1-cube dual to $H_i$ and that dual to $H_j$, and joining $a_i\in N(H_i)\cap\gamma$ to $a_j\in N(H_j)\cap\gamma$.  Let $x$ be a closest 0-cube of $N(H_i)\cap N(H_j)$ to $\gamma'$.  If $a_i=a_j$, there is nothing to prove.  If $x=a_i$, then by convexity of $N(H_i)$, the path $\gamma'$ lies in $N(H_i)$ and thus $N(H_i)$ contains a subpath of $\gamma$ of length at least $|i-j|-1$.  Thus $|i-j|-1\leq B''(0)$.

Hence suppose that $x,a_i,a_j$ are all distinct and let $m$ be their median.  Since the interval between $a_i$ and $x$ lies in $N(H_i)$, by convexity, $m\in N(H_i)$.  Similarly $m\in N(H_j)$.  Any hyperplane separating $m$ from $\gamma'$ separates $m$ from $a_i$ and $a_j$ and hence separates $x$ from $a_i$ and $a_j$.  Thus $m\in N(H_i)\cap N(H_j)$ is at least as close to $\gamma'$ as is $x$, and therefore $m=x$, so that $x$ lies on a geodesic segment joining $a_i$ to $a_j$.  It follows that every hyperplane $H_k$, with $i<k<j$ (say), separates exactly one of $a_i$ and $a_j$ from $x$.  Moreover, every hyperplane separating $x$ from $a_i$ or $a_j$ separates $a_i$ from $a_j$, since $x$ lies on a geodesic from $a_i$ to $a_j$.  Let $A_i\rightarrow N(H_i)$ be a geodesic segment joining $a_i$ to $x$, and let $A_j\rightarrow N(H_j)$ be a geodesic segment joining $x$ to $a_j$, chosen so that $A_iA_j$ is a geodesic segment.

Let $e$ be the midpoint of $\gamma'$.  Let $\gamma_i$ be the subpath of $\gamma'$ joining $a_i$ to $e$, and let $\gamma_j$ be the subpath joining $e$ to $a_j$.  Let $h_i$ be the number of hyperplanes crossing $\gamma_i$ and separating $a_i$ and $x$ from $e$ and let $h_j$ be the number of hyperplanes crossing $\gamma_j$ and separating $a_j$ and $x$ from $e$.  If $H$ separates $e$ from $a_j$ and $x$, then $H$ separates $a_j$ and $x$ from $a_i$.  Similarly, if $H'$ separates $e$ from $a_i$ and $x$, then $H'$ separates $x$ and $a_i$ from $a_j$.  Hence $H\bot H'$ for all such $H,H'$.  Thus $\Lambda(\gamma)$ contains $K_{h_i,h_j}$.  Thus $\min(h_i,h_j)\leq B'$.  As shown in Figure~\ref{fig:aiaj}, $\gamma_i$ lies in the $h_i$-neighborhood of $N(H_i)$, and hence $|\gamma_i|\leq B''(h_i)$.  For the same reason, $|\gamma_j|\leq B''(h_j)$.  But $|\gamma_i|=|\gamma_j|$, since $e$ is the midpoint of $\gamma'$.  Thus $|i-j|=|\gamma|+1\leq \max(2B''(B')+1,B''(0))=B$.
\end{proof}

\begin{figure}[h]
  \includegraphics[width=1.75in]{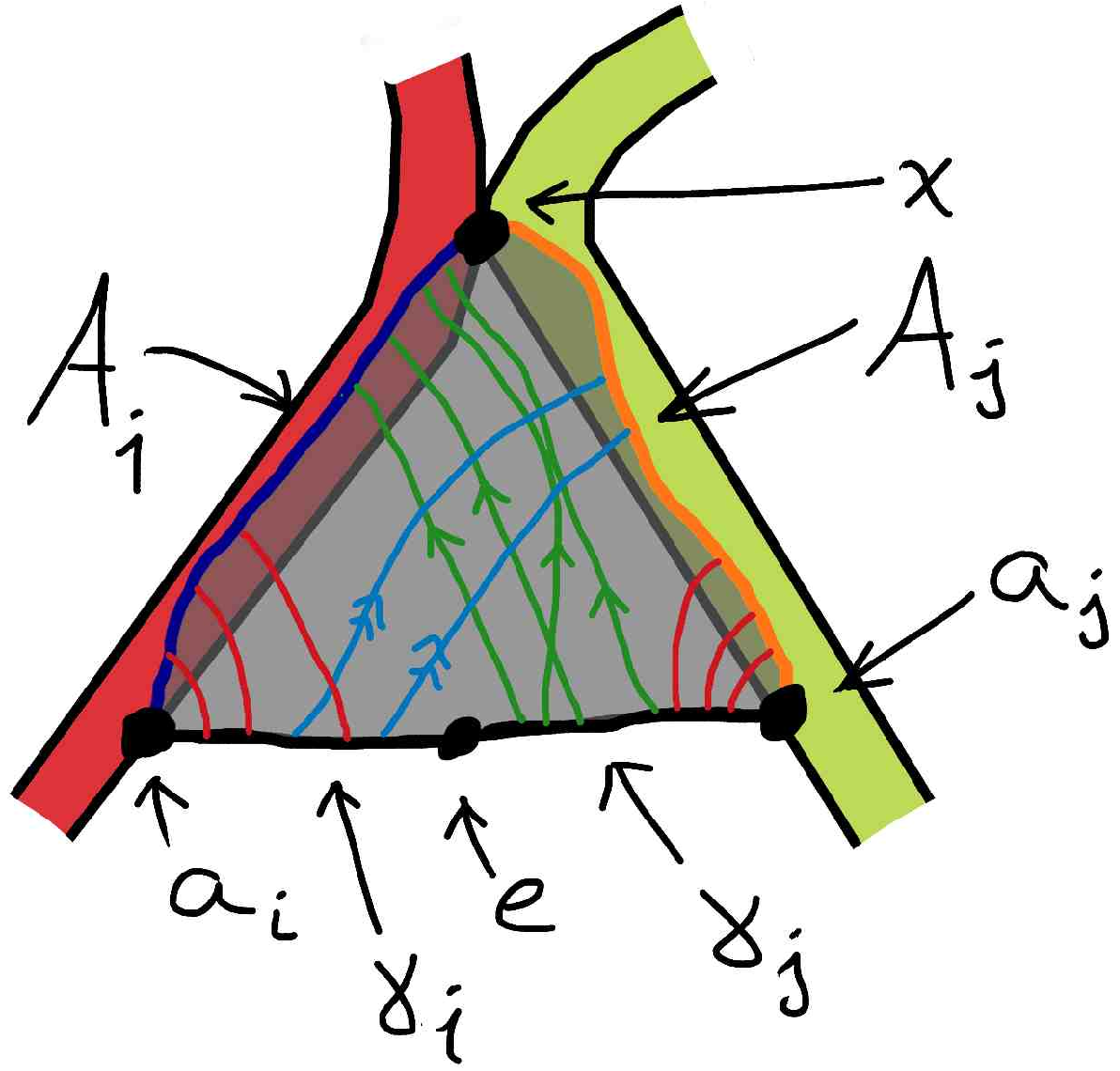}\\
  \caption{Either $\gamma'$ contains a long path lying in the carrier of a hyperplane or a family of dual curves corresponding to a large biclique in $\crossing X$.}\label{fig:aiaj}
\end{figure}

\begin{thm}[Projection trichotomy II]\label{thm:trichotomy2}
Let $\mathbf X$ and $\gamma$ be as in Theorem~\ref{thm:trichotomy1}, and assume that there is no infinite collection of pairwise-crossing hyperplanes.  Suppose that $\gproj(\gamma)$ is bounded in $\contact X$.  Then either $\gamma\subset N_R(H)$ for some $R\geq 0$ and some hyperplane $H$, or $\gamma$ bounds an eighth-flat in $\mathbf X$.
\end{thm}

\begin{proof}
Let $H_i$ be the hyperplane dual to $\gamma([i,i+1])$.  Assume that there exists an infinite increasing sequence $0\leq n_0<n_1<\ldots$ so that for all $j\geq 1$, the hyperplane $H_{n_j}$ separates $H_{n_{j-1}}$ from $H_{n_{j+1}}$.  Since $H_{n_i}$ separates $H_{n_0}$ from $H_{n_j}$ for $0<i<j$, any path in $\contact X$ from $H_0$ to $H_{n_j}$ contains a hyperplane $V$ that crosses $H_{n_i}$.  Moreover, if $U$ crosses $H_{n_0}$ and $H_{n_j}$, then $U$ crosses $H_{n_i}$.  Hence, if $(1)$ does not hold, it follows that there exists $K\geq 0$ such that for all $k>K$, some hyperplane $V_k$ crosses $H_{n_i}$ for all $i\geq k$.  Let $\mathcal H=\{H_{n_i}\}$ and let $\mathcal V$ be the set of all hyperplanes $V$ that cross all but finitely many $H_{n_i}$.

\textbf{Verifying $(3)$:}  Suppose that $|\mathcal V\cap\mathcal W(\gamma)|<\infty$.  By deleting a finite initial segment from $\gamma$, we may assume that no hyperplane in $\mathcal V$ crosses $\gamma$, and some $V\in\mathcal V$ crosses $H_n$ for all $n\geq 0$.  Indeed, if $V$ crosses every $H_{n_i}$ and does not cross some $H_n$, then $H_n\in\mathcal V\cap\mathcal W(\gamma)$, and we can truncate the initial part of $\gamma$ to remove the 1-cube dual to $H_n$.

For each $n\geq 0$, if a hyperplane $U$ separates $\gamma(n)$ from $N(V)$, then $U$ either separates all of $\gamma$ from $N(V)$ or $U$ crosses $\gamma$.  If $U$ crosses $\gamma([0,n])$, then $U$ cannot cross $\gamma([n,\infty))$, and hence $U$ crosses $H_m$ for $m\geq n$.  In particular, $U\in\mathcal V\cap\mathcal W(\gamma)$, a contradiction.  Thus $U$ crosses $\gamma([n,\infty))$, or $U$ separates $\gamma$ from $N(V)$.  Hence $U$ is one of the finitely many hyperplanes separating $N(V)$ from $\gamma(0)$, and thus $d_{_{\mathbf X}}(\gamma(n),N(V))\leq d_{_{\mathbf X}}(\gamma(0),N(V))$ for all $n\geq 0$, and~$(3)$ holds.

\textbf{Building an eighth-flat:}  If for some $n$ the ray $\gamma([n,\infty))$ bounds an eighth-flat, then so does $\gamma$.  Also, we have reduced to the case in which $\mathcal V\cap\mathcal W(\gamma)$ is infinite.  Hence, by truncating, assume that $H_0\in\mathcal V$.  For each $n\geq 0$, let $x_n=\gamma(n)$, and let $g_n$ be the \emph{gate} of $x_n$ in $N(H_0)$, i.e. the unique closest 0-cube of the convex set $N(H_0)^0$ to $x_n$.  Define a function $f$ by $f(n)=d_{_{\mathbf X}}(x_n,g_n)$.  Now, any hyperplane separating $x_n$ from $x_m$ either crosses $H_0$ and separates $g_m$ from $g_n$, or separates exactly one of $x_m$ or $x_n$ from $N(H_0)$, since $N(H_0)$ is convex.  Thus $f$ is a non-decreasing function.  If $f$ is bounded, then~$(3)$ holds.

For each $n\geq 0$, any hyperplane $H$ separating $g_n$ from $g_{n+1}$ crosses $H_0$, and hence does not separate $x_m$ or $x_n$ from $H_0$.  Thus $H$ separates $x_{n+1}$ from $x_n$, so there is a unique combinatorial geodesic ray $\sigma\rightarrow N(H_0)$ such that $\sigma(n)=g_n$ for each $n\geq 0$.  ($\sigma$ is the closest point projection of $\gamma$ to the convex subcomplex $N(H)$.  A discussion of gates in median graphs appears in~\cite{vandeVel_book}.)

For each $n\geq 0$, let $\omega_n$ be a combinatorial geodesic joining $x_n$ to $g_n$.  Observe that, if $g_{n+1}=g_n$, then $\omega_{n+1}=\gamma([n,n+1])\omega_n$.  If $g_n\neq g_{n+1}$, then $|\omega_n|=|\omega_{n+1}|$, and $\omega_n$ and $\omega_{n+1}$ lie on opposite sides of an isometrically embedded ``strip'' $S_n\cong\omega_n\times[-\frac{1}{2},\frac{1}{2}]$ in $\mathbf X$.  Indeed, the set of hyperplanes separating $x_n$ from $g_n$ coincides with that separating $x_{n+1}$ from $g_{n+1}$ in this case.  From the definition, we see that $\mathbf E=\bigcup_nS_n$ is an embedded eighth-flat in $\mathbf X$ whose bottom bounding ray is $\sigma$ and whose top bounding ray is $\gamma$.

For each $x\in\mathbf E^0$, choose a smallest $n$ so that $x$ lies on $\omega_n$.  Let $\alpha_1(x)=n$, and let $\alpha_2(x)\leq f(n)$ be the distance from $x$ to $g_n$ in $\mathbf E$.  Then for each $x,y\in\mathbf X^0$,
\[d_{\mathbf E}(x,y)=|\alpha_1(x)-\alpha_1(y)|+|\alpha_2(x)-\alpha_2(y)|\geq d_{_{\mathbf X}}(x,y).\]
Let $n=\alpha_1(x),m=\alpha_1(y)$.  Since $x$ lies on a geodesic from $x_n$ to $g_n$, $g_n$ is the gate of $x$ in $N(H_0)$.  Likewise, $g_m$ is the gate of $y$ in $N(H_0)$.  Hence the distance in $\mathbf X$ between the gates in $N(H_0)$ of $x$ and $y$ is $|m-n|$.  Since $S_m,S_n$ are isometrically embedded, the distance from $y$ to $g_m$ in $\mathbf X$ is $\alpha_2(y)=m$.  Thus the number of hyperplanes that separate exactly one of $x,y$ from $N(H_0)$ in $\mathbf X$ is $|\alpha_2(x)-\alpha_2(y)|$, so that $\mathbf E$ is isometrically embedded.

\textbf{Existence of $\{H_{n_i}\}$:}  To conclude, we must show that the claimed family $\{H_{n_i}\}_{i\geq 0}$ of pairwise non-crossing hyperplanes exists.  Were this not the case, then there would exist $N$ such that $H_m\bot H_n$ for all $m,n\geq N$, a contradiction.
\end{proof}

A combinatorially hyperbolic isometry of $\mathbf X$ is \emph{combinatorially rank-one} if some (and hence every) combinatorial geodesic axis does not lie in a combinatorial half-flat.  The following simple observation relates this definition to rank-one rays.

\begin{prop}\label{prop:combinatorialrankone}
Let $\mathbf X$ be locally finite and let $g\in\Aut(\mathbf X)$ be a combinatorially hyperbolic isometry.  Then $g$ is combinatorially rank-one if and only if for some (and hence for every) combinatorial geodesic axis $\gamma$ for $g$, and for all $m\in\integers$, the ray $\gamma([m,\infty))$ is rank-one.
\end{prop}

\begin{proof}
Some sub-ray of $\gamma$ bounds an eighth-flat.  Translating this eighth-flat by elements of $\langle g\rangle$ yields an increasing union of isometric eighth-flats, so that $\gamma$ lies in a half-flat.
\end{proof}

\section{The simplicial boundary}\label{sec:boundary}
We now turn to the definition and essential properties of the \emph{simplicial boundary} $\simp\mathbf X$ of $\mathbf X$.

\subsection{Unidirectional boundary sets}\label{sec:ubs}
We assume in this section that $\mathbf X$ contains no infinite family of pairwise-crossing hyperplanes, but not that $\mathbf X$ is finite-dimensional or locally finite.

The distinct hyperplanes $H_1,H_2,H_3$ form a \emph{facing triple} if for each $i\in\{1,2,3\}$, there exists a halfspace $\mathfrak h(H_i)$ that contains $H_j$ and $H_k$, for $j,k\in\{1,2,3\}-\{i\}$.  Equivalently, no set of three halfspaces, one associated to each of the $H_i$, is totally ordered by inclusion.

\begin{defn}[Nested, semi-nested]\label{defn:seminestet}
The set $\mathcal U$ of hyperplanes is \emph{nested} if there exists consistent orientation $\chi$ of $\mathcal W$, corresponding to a 0-cube or a 0-cube at infinity, such that $\{\chi(U)\mid U\in\mathcal U\}$ is totally ordered by inclusion.  Note that the hyperplanes in a nested set are pairwise non-crossing.  $\mathcal U$ is \emph{semi-nested} if there exists a consistent orientation $\chi$ of $\mathcal W$ such that $\{\chi(U)\mid U\in\mathcal U\}$ is partially ordered by inclusion, and, for all $U,U'\in\mathcal U$ such that $\chi(U)\not\subset\chi(U')$, either $\chi(U')\subseteq\chi(U)$ or $U\bot U'$.
\end{defn}

Given a consistent orientation $\chi$ of $\mathcal W$, define a partial ordering $\prec$ on $\mathcal U$ by $U\prec U'$ if and only if $\chi(U')\subseteq\chi(U)$.  Then $\mathcal U$ is nested if for some $\chi$, $\prec$ is a total ordering, and $\mathcal U$ is semi-nested if for some $\chi$, any two distinct hyperplanes in $\mathcal U$ cross or are $\prec$-comparable.  Nested sets are semi-nested, and a semi-nested set cannot contain a facing triple.

\begin{defn}[Unidirectional]\label{defn:unidirectional}
Let $\mathcal U$ be a set of hyperplanes.  Then $\mathcal U$ is \emph{unidirectional} if for each $U\in\mathcal U$, at least one halfspace associated to $U$ contains only finitely many elements of $\mathcal U$.
\end{defn}

\begin{lem}\label{lem:nesting}
A set $\mathcal U\subset\mathcal W$ that is unidirectional, infinite, inseparable, and without a facing triple is semi-nested.
\end{lem}

\begin{proof}
$\mathcal U$ must contain an infinite set of pairwise non-crossing hyperplanes.  Since $\mathcal U$ is unidirectional, there exists a hyperplane $U_0\in\mathcal U'$ such that the halfspace $U_0^-$ does not contain any element of $\mathcal U$.  Since $\mathcal U$ does not contain a facing triple, $\mathcal U$ contains a set $\mathcal U'=\{\mathcal U_i\}_{i\geq 0}$ such that, for all $i\geq 1$, the hyperplanes $U_{i-1}$ and $U_{i+1}$ are separated by $U_i$.  For each $i\geq 0$, let $\chi(U_i)$ be the halfspace associated to $U_i$ that contains $U_{i+1}$.  If some $U\in\mathcal W$ separates some $U_i$ from some $U_j$, with $j>i$, let $\chi(U)$ be the halfspace containing $U_j$.  Since $\mathcal U$ is inseparable, such a $U$ belongs to $\mathcal U$.

More generally, let $\mathcal A$ be the set of all hyperplanes $H\in\mathcal W$ with $H\subset\mathbf X-\chi(U_0)$.  Let $\mathcal B$ be the set of all hyperplanes $H'\in\mathcal W$ with $H'\subset\chi(U_0)$.  Let $\mathcal C$ be the set of hyperplanes $H''\in\mathcal W$ that cross $U_0$. For each $H\in\mathcal A$, let $\chi(H)$ be the halfspace associated to $H$ that contains $U_0$, i.e. orient $H$ toward $U_0$.  If $H\in\mathcal U\cap\mathcal B$, let $\chi(H)$ be the halfspace that does not contain $U_0$.  If $H\in\mathcal B-\mathcal U$, then for all $U\in\mathcal U$, either $U\bot H$ or $U$ and $U_0$ lie in the same halfspace associated to $H$.  In either event, let $\chi(H)$ be the halfspace containing $U_0$.  In other words, for all hyperplanes $H$ that do not cross $U_0$, $\chi$ orients $H$ away from $U_0$ exactly when $H\not\in\mathcal U$.  If $H$ does not cross $U$, and $H'\neq U$ does not cross $U$, then $\chi(H)\cap\chi(H')\neq\emptyset$.

If $H\bot U$, then any orientation of $H$ is consistent with $\chi|_{\mathcal A}$, since $\chi$ orients each $A\in\mathcal A$ toward $H\cap U$.  Similarly, any orientation of $H$ is consistent with $\chi|_{\mathcal B-\mathcal U}$.  If $H\subset(\mathbf X-\chi(U_i))$ for some $i\geq 1$, let $\chi(H)$ be the halfspace containing $U_i$.  This is consistent with $\chi|_{\mathcal B\cap\mathcal U}$.  Moreover, if $H'$ is another hyperplane that crosses $H$ but fails to cross some $U_j$, than $\chi(H')\cap\chi(H)$ contains $U_{\max(i,j)}$, so $\chi$ is consistent across all such $H,H'$.  The remaining hyperplanes are those $H$ that cross $U_i$ for all $i\geq 0$.  Choose a fixed 0-cube $x_0\in N(U_0)$ and let $\chi$ orient all of these $H$ toward $x_0$.  If $H$ is such a hyperplane, and $H'\bot U$ and $H'\not\bot U_i$, then $N(H)\cap N(U_i)$ lies in $\chi(H')$ and contains a 0-cube in $\chi(H)$.  Thus $\chi$ is a consistent orientation of all hyperplanes of $\mathbf X$.  Since $\bigcap_{i\geq 0}\chi(U_i)=\emptyset$, $\chi$ is a 0-cube at infinity.  Finally, for $U,U'\in\mathcal U$, let $U\prec U'$ if and only if $\chi(U')\subseteq\chi(U)$.  This is clearly a partial ordering, and every $\prec$-chain in $\mathcal U$ has a minimum since $\mathcal U$ is unidirectional.

Let $\mathcal C$ be the set of hyperplanes $C$ such that $C\bot U_i$ for all $i\geq 0$.  Recall that $\chi(C)$ is the halfspace containing $x_0$, and that the orientation $\chi$ is consistent on all of $\mathcal W$ for arbitrarily chosen $x_0\in N(H_0)$.  Observe that since $\mathcal U$ is inseparable and contains no facing triple, $\prec$ fails to have the desired comparability property on $\mathcal U$ if and only if there exist $C,C'\in\mathcal U\cap\mathcal C$ that do not cross, with $x_0\in N(C)\cap N(C')$.  Hence the comparability requirement holds for $\chi$ except on the set $\mathcal U\cap\mathcal W(\chi(U_0)\cap\chi(C))$.  Replacing $\chi$ by the 0-cube $\chi_1$ at infinity that coincides with $\chi$ on all hyperplanes except $C$ preserves consistency, but $\chi_1$ yields a partial order satisfying the comparability requirement on a subset of $\mathcal U$ that properly contains $\mathcal U\cap\mathcal W(\chi(U_0)\cap\chi(C))$.  We can thus write $\mathcal U$ as an increasing union of subsets $\mathcal U^n$ such that, for each $n$, there is a 0-cube $\chi_n$ at infinity such that for all $U,U'\in\mathcal U^n$, either $U\bot U'$ or $\chi(U)\subseteq\chi(U')$ or $\chi(U')\subseteq\chi(U)$.  Furthermore, the restriction of $\chi_m$ to $\mathcal U^n$ coincides with $\chi_n$ whenever $m\geq n$.  We thus define a 0-cube $\eta$ at infinity by $\eta(H)=\chi_n(H)$ for some (and hence every) $n$ for which $H\in(\mathcal W-\mathcal U)\cup\mathcal U_n$.  By construction, $\eta$ is consistent and the associated partial order on $\mathcal U$ is as required.
\end{proof}

We now abstract the essential properties of the set $\mathcal U$ in Lemma~\ref{lem:nesting}.

\begin{defn}[Unidirectional boundary set, almost-equivalent]\label{defn:UBS}
A \emph{unidirectional boundary set (UBS)} $\mathcal U$ is an infinite, inseparable, unidirectional set of hyperplanes containing no facing triple.  The UBSs $\mathcal U'$ and $\mathcal U$ are \emph{almost-equivalent} if their symmetric difference is finite.  The UBS $\mathcal U$ is \emph{minimal} if any UBS $\mathcal U'\subseteq\mathcal U$ is almost-equivalent to $\mathcal U$.  Almost-equivalence of UBSs is an equivalence relation.
\end{defn}

\begin{exmp}\label{exmp:ray}
If $\gamma:[0,\infty)\rightarrow\mathbf X$ is a geodesic ray, then $\mathcal W(\gamma)$ is a UBS.  If $\gamma'$ is another ray, then $\gamma$ and $\gamma'$ are almost-equivalent if and only if $\mathcal W(\gamma)$ and $\mathcal W(\gamma')$ are almost equivalent UBSs.  An associated 0-cube at infinity orients every $H\in\mathcal W-\mathcal W(\gamma)$ toward $\gamma(0)$, and orients each $H\in\mathcal W(\gamma)$ away from $\gamma(0)$.
\end{exmp}

\begin{defn}[Inseparable closure]\label{defn:inseparableclosure}
Let $\mathcal U\subset\mathcal W$.  The \emph{inseparable closure} $\overline{\mathcal U}$ of $\mathcal U$ is the intersection of all inseparable sets $\mathcal V$ such that $\mathcal U\subseteq\mathcal V$.  Note that $\overline{\mathcal U}$ consists of $\mathcal U$, together with the set of hyperplanes $V$ for which there exist $U,U'\in\mathcal U$ such that $V$ separates $U$ and $U'$.
\end{defn}

\begin{lem}\label{lem:UBSesexist}
If $\mathbf X$ contains no infinite family of pairwise-crossing hyperplanes and $\mathcal W$ contains a UBS, then $\mathcal W$ contains a minimal UBS.  Moreover, every essential hyperplane is contained in a UBS.
\end{lem}

\begin{rem}
The hypotheses are satisfied if $\mathbf X$ is unbounded and strongly locally finite, or, more generally, if $\mathbf X$ has no infinite family of pairwise-crossing hyperplanes and contains at least one deep halfspace.
\end{rem}

\begin{proof}[Proof of Lemma~\ref{lem:UBSesexist}]
By hypothesis, there exists an infinite set $\mathcal U'=\{U_i\}_{i\geq 0}$ of hyperplanes such that, for all $i\geq 1$, the hyperplane $U_i$ separates $U_{i-1}$ from $U_{i+1}$.  Let $\mathcal U$ be the inseparable closure of $\mathcal U'$.  Clearly $\mathcal U$ is infinite and inseparable.

\textbf{$\mathcal U$ contains no facing triple:}  Since $\mathcal U'$ does not contain a facing triple, any facing triple in $\mathcal U$ is of the form $U_i,U_j,V_1$ or $U_i,V_1,V_2$ or $V_1,V_2,V_3$, where $V_1,V_2,V_3\in\mathcal U-\mathcal U'$.  If $U_i,U_j,V_1$ is a facing triple, then $V_1$ separates some $U_k$ from $U_i$, and hence from $U_j$.  Thus $U_i,U_j,U_k$ is a facing triple, a contradiction.  If $U_i,V_1,V_2$ is a facing triple, then $V_1$ separates some $U_j$ from $U_i$, and thus $U_i,U_j,V_2$ form a facing triple, contradicting the previous statement.  By like reasoning, $V_1,V_2,V_3$ cannot form a facing triple.

\textbf{$\mathcal U$ is unidirectional:}  Let $U\in\mathcal U$ have the property that both associated halfspaces contain infinitely many elements of $\mathcal U$.  Since there exist $U_i,U_j$ separated by $U$, with $i<j$, and only finitely many hyperplanes separate $U_i$ from $U_j$, either the halfspace $U_i^-$ not containing $U$ contains infinitely many hyperplanes, or infinitely many elements of $\mathcal U$ cross $U_i$ and do not cross $U$.  Now, no element of $\mathcal U$ crosses $U_0$, by the definition of the inseparable closure.  Hence infinitely many hyperplanes $V\in\mathcal U$ cross $U_i$ but do not cross $U$ or $U_0$.  Either $U,U_0,V$ form a facing triple, or $V$ is among the finitely many hyperplanes separating $U$ from $U_0$.  This is a contradiction, and it follows that $U$ is unidirectional.

\textbf{$\mathcal U$ contains a minimal UBS:}  Let $\mathcal S\subseteq\mathcal U$ be a UBS.  If $\mathcal S$ does not contain $U_i$, then $\mathcal S$ cannot contain any $U\in\mathcal U$ that separates $U_i$ from $U_0$.  Thus there exists some minimal $I<\infty$ such that $\mathcal S$ contains $U_i$ for $i\geq I$.  If $U\in\mathcal U$ does not separate $U_I$ from $U_0$ and does not cross $U_i$, then inseparability of $\mathcal S$ requires that $U\in\mathcal S$.  Hence every element of $\mathcal U-\mathcal S$ either separates $U_0$ from $U_I$ or crosses $U_I$ and separates $U_0$ from $U_i$ for some $i\geq I$.

Indeed, if $U\in\mathcal U$, then $U$ cannot cross all but finitely many of the $U_i$, for otherwise $\mathcal U-\{U\}$ would be a smaller inseparable set containing $\mathcal U'$.  Let $\mathcal C_1$ be the set of hyperplanes in $\mathcal U$ that cross $U_i$.  If $\mathcal C_1$ is finite, then $\mathcal U-\mathcal S$ has cardinality $|\mathcal C_1|+d_{_{\mathbf X}}(N(U_0),N(U_i))$ and so $\mathcal S$ is almost-equivalent to $\mathcal U$.  If $\mathcal C_1$ is infinite, then let $\mathcal U_1=\mathcal U-\mathcal S$.  Note that $\mathcal C_1$ is unidirectional and contains no facing triple, since those properties are inherited by subsets.  If $C,C'\in\mathcal C_1$ are separated by some hyperplane $W$, then $W$ crosses $U_I$ and crosses each hyperplane $U_j$ crossed by both $C$ and $C'$.  Moreover, since $W$ does not cross $C$ or $C'$, $W$ does not cross $U_j$ for $j$ sufficiently large.  Hence $W\in\mathcal C_1$, which is therefore a UBS.  Likewise, $\mathcal U_1$ is infinite, unidirectional, and contains no facing triple.  By the definition of $\mathcal C_1$, $\mathcal U_1$ is also inseparable.  Hence $\mathcal U_1$ is a UBS.  Moreover, every element of $\mathcal U_1$ crosses all but finitely many elements of $\mathcal C_1$.

Hence, if $\mathcal U$ is not minimal, then $\mathcal U$ contains two disjoint UBSs, $\mathcal U_1$ and $\mathcal C_1$, such that every hyperplane in $\mathcal U_1$ crosses all but finitely many hyperplanes in $\mathcal C_1$.  If $\mathcal U_1$ is not minimal, then by the same reasoning, $\mathcal U_1$ is almost-equivalent to a disjoint union $\mathcal U_2\sqcup\mathcal C_2$ of UBSs such that each $U\in\mathcal U_2$ crosses all but finitely many $C\in\mathcal C_2$, and each of $U$ and $C$ crosses all but finitely many $C'\in\mathcal C_1$.  Apply the same argument to $\mathcal U_2$, and continue in this way. Since $\mathbf X$ does not contain an infinite set of pairwise-crossing hyperplanes, we must, after finitely many applications of this argument, find a minimal UBS contained in $\mathcal U$.

\textbf{Essential hyperplanes:}  If $U\in\mathcal W$ is essential, then by definition there exists an infinite nested set of hyperplanes contained in a halfspace associated to $U$, and we argue as above.
\end{proof}

\begin{rem}
The proof of Lemma~\ref{lem:UBSesexist} shows that if $\mathbf X$ contains no infinite collection of pairwise-crossing hyperplanes, and $\mathcal U\subset\mathcal W$ is a UBS, then $\mathcal U$ contains a minimal UBS.  Indeed, $\mathcal U$ must contain a unidirectional nested set of hyperplanes, whose inseparable closure is contained in $\mathcal U$.
\end{rem}

The following theorem allows us to define the simplicial boundary of $\mathbf X$.

\begin{thm}\label{thm:structureofboundarysets}
Suppose that every collection of pairwise-crossing hyperplanes in $\mathbf X$ is finite.  Let $v$ be an almost-equivalence class of UBSs.  Then $v$ has a representative of the form $\mathcal V=\bigsqcup_{i=1}^k\mathcal U_i,$ where $k<\infty$, each $\mathcal U_i$ is minimal, and for all $1\leq i<j\leq k$, if $H\in\mathcal U_j$, then $H$ crosses all but finitely many elements of $\mathcal U_i$.

Moreover, if $\mathcal V'=\bigsqcup_{i=1}^{k'}\mathcal U'_i$ is almost-equivalent to $\mathcal V$, and each $\mathcal U'_i$ is minimal, then $k=k'$ and, up to re-ordering, $\mathcal U_i$ and $\mathcal U'_i$ are almost equivalent for all $i$.
\end{thm}

\begin{proof}
If $\mathcal V$ is a minimal UBS, then by definition any UBS $\mathcal U\subset\mathcal V$ is almost-equivalent to $\mathcal V,$ and the desired decomposition into minimal UBSs exists and is unique in the required sense.  We now proceed inductively.  Let $\mathcal U_1$ be a minimal UBS contained in $\mathcal V.$  Let $\mathcal V_1=\mathcal V-\mathcal U_1$.  If $\mathcal V_1$ is finite, then $\mathcal V$ is almost-equivalent to $\mathcal U_1$, and $\mathcal V$ is already minimal.

Let $V\in\mathcal V_1$ be a hyperplane.  If $V$ crosses infinitely many hyperplanes of $\mathcal U_1$, then $V$ crosses all but finitely many hyperplanes of $\mathcal U_1$ since, for all $U\in\mathcal U_1$, all but finitely many hyperplanes of $\mathcal U_1$ lie in a common halfspace associated to $U.$  Let $\mathcal V'_1\subset\mathcal V_1$ be the set of hyperplanes in $\mathcal V-\mathcal U_1$ that cross only finitely many elements of $\mathcal U_1$.  If $\mathcal V'_1$ is infinite, then since $\mathbf X$ is strongly locally finite, $\mathcal V'_1$ contains an infinite nested set $\{V_i\}_{i\geq 0}$ with initial hyperplane $V_0$ such that for all but finitely many $U\in\mathcal U_1$, $V_0$ separates $U$ from each $V_i$ with $i\geq 1$.  This implies that $\mathcal U_1\cup\mathcal V'_1$ contains a nested set with no initial hyperplane, contradicting the fact that $\mathcal V$ is a UBS.  Thus $\mathcal V'_1$ is finite and can be assumed to be empty without affecting the almost-equivalence class of $\mathcal V$.  Hence each $V\in\mathcal V_1$ crosses all but finitely many elements of $\mathcal U_1$, and thus $\mathcal V_1$ contains an infinite inseparable set.  $\mathcal V_1$ thus contains a minimal UBS $\mathcal U_2$, by Lemma~\ref{lem:UBSesexist}.  Moreover, each $U\in\mathcal U_2$ crosses all but finitely many of the hyperplanes in $\mathcal U_1.$

By induction, we obtain a (a priori, infinite) maximal index set $K\subseteq\naturals$ and a subset $\bigsqcup_{k\in K}\mathcal U_k\subseteq\mathcal V,$ where each $\mathcal U_k$ is a minimal UBS and for all $k<n,$ each $U\in\mathcal U_n$ crosses all but finitely many elements of $\mathcal U_k$.  We can thus choose, for each $k\in K$, $U_k\in\mathcal U_k$ such that the hyperplanes $U_k,U_j$ cross whenever $j\neq k$.  Hence $|K|<\infty$ since $\mathbf X$ is strongly locally finite.  (If $\dimension\mathbf X<\infty$, then $|K|\leq\dimension(\mathbf X)$).

Consider the set $\mathcal F=\mathcal V-\bigsqcup_{k\in K}\mathcal U_k\subseteq\mathcal V.$  If $\mathcal F$ is infinite, $\mathcal F$ contains a UBS, contradicting maximality of $K$.  Hence $\mathcal F$ is finite, and we may remove $\mathcal F$ from $\mathcal V$ without affecting the almost-equivalence class $v$ of $\mathcal V.$

Finally, suppose $\mathcal V'=\bigcup_{k'\in K'}\mathcal U'_{k'}$ is another such decomposition into minimal inseparable sets, with $\mathcal V'$ almost-equivalent to $\mathcal V$.  Suppose there exist $k'_1$ and $k'_2\in K'$ and $k\in K$ such that $\mathcal U_k$ has infinite intersection with $\mathcal U'_{k'_i}$ for $i=1,2$.  By inseparability, for $i\in\{1,2\}$, we have that $\mathcal U_k\cap\mathcal U'_{k'_i}$ contains all but finitely many elements of $\mathcal U_k$ and $\mathcal U'_{k'_i}$ (see below).  Without affecting the almost-equivalence class, we can remove finitely many hyperplanes from the symmetric difference of these sets to conclude that $\mathcal U_k=\mathcal U'_{k_1'}=\mathcal U'_{k_2'}$, and hence that $k_1=k_2$.  Thus $\mathcal U_k$ is almost-equivalent to at most one of the sets $\mathcal U'_{k'}$.

Suppose $\mathcal U_k$ is not almost-equivalent to $\mathcal U'_{k'}$ for any $k'\in K'$.  Then $\mathcal U_k$ has finite intersection with each $\mathcal U'_{k'}$.  But $\mathcal U_k$ is infinite and contained in the finite union $\bigcup_{k'\in K'}\mathcal U'_{k'}$, a contradiction.  Thus $\mathcal U_k$ is almost-equivalent to at least, and thus exactly, one of the $\mathcal U'_{k'}$.

We used the fact that, if $\mathcal U_1,\mathcal U_2$ are minimal UBSs, then $\mathcal U_1\cap\mathcal U_2$ is infinite if and only if $\mathcal U_1$ and $\mathcal U_2$ are almost-equivalent.  Indeed, if $\mathcal U_1$ is almost-equivalent to $\mathcal U_2,$ then, since each is infinite and their symmetric difference finite, their intersection must be infinite.  Conversely, if $\mathcal U_3=\mathcal U_1\cap\mathcal U_2$ is infinite, then it contains a UBS $\mathcal U_3'$.  But $\mathcal U'_3$ is almost-equivalent to $\mathcal U_1$.  Likewise, $\mathcal U'_3$ is almost-equivalent to $\mathcal U_2$, whence $\mathcal U_1\triangle\mathcal U_2$ is finite.
\end{proof}

\subsection{Simplices at infinity}\label{sec:simplicesatinfinity}
The unidirectional boundary set $\mathcal V$ is \emph{$k$-dimensional} if the decomposition of $\mathcal V$ into minimal UBSs given by Theorem~\ref{thm:structureofboundarysets} has $k$ factors.  Theorem~\ref{thm:structureofboundarysets} implies that, if $v$ is an almost-equivalence class of UBSs, then there is a unique integer $k$ such that every UBS representing $v$ is $k$-dimensional.  A \emph{0-simplex at infinity} is an almost-equivalence class of minimal UBSs.  More generally, a \emph{$k$-simplex at infinity} is an almost-equivalence class of UBSs whose representatives are $(k+1)$-dimensional.  If $u,v$ are simplices at infinity, let $u\leq v$ if and only if there exist representatives $\mathcal U$ of $u$ and $\mathcal V$ of $v$ such that $\mathcal U\subseteq\mathcal V$.  By Theorem~\ref{thm:structureofboundarysets}, $\leq$ is a partial ordering on the set of almost-equivalence classes.

\begin{defn}[Simplicial boundary]\label{defn:simplicialboundary}
Suppose that each collection of pairwise-crossing hyperplanes in $\mathbf X$ is finite.  The \emph{simplicial boundary} $\simp\mathbf X$ of $\mathbf X$ is the geometric realization of the abstract simplicial complex whose set of simplices is the set of simplices at infinity, in which $u$ is a face of $v$ if and only if $u\leq v$.

A 0-simplex $v\in\simp\mathbf X$ is \emph{isolated} if $v$ is not contained in a 1-simplex.  The graph $(\simp\mathbf X)^1$ is equipped with the standard path-metric, and two 0-simplices are at distance $+\infty$ if they lie in distinct components of $\simp\mathbf X$.  If $(\simp\mathbf X)^1$ has finite diameter in this metric, we say $\simp\mathbf X$ is \emph{bounded}.  Otherwise, $\simp\mathbf X$ is \emph{unbounded}.  Define $\diam(\simp\mathbf X)$ to be the diameter of $(\simp\mathbf X)^1$ with the standard path-metric.
\end{defn}

\begin{exmp}
The simplicial boundary of an eighth-flat or a nondiagonal quarter flat is a single 1-simplex, that of a diagonal quarter flat or non-diagonal half-flat is a subdivided interval of length 2, that of a diagonal half-flat is a subdivided interval of length 3, and that of a flat is a 4-cycle.  The simplicial boundary of a tree with more than one end is totally disconnected.
\end{exmp}

%NEW:
\begin{rem}\label{rem:slfneed}
If $\mathcal W$ contains an infinite set $\mathcal U$ of pairwise-crossing hyperplanes, then the preceding method of defining simplices is unavailable.  Indeed, $\mathcal U$ is a UBS.  However, let $\{U_n\}_{n=1}^\infty$ be a countably infinite subset of $\mathcal U$, and for each $m\geq 2$, note that $\{U_{mn}\}_{n\geq 1}$ is a UBS that is not almost-equivalent to $\mathcal U$.  In particular, $\ldots\{U_{2^kn}\}_n\subset\{U_{2^{k-1}n}\}_n\ldots\subset\{U_n\}_n\subset\mathcal U$ is an infinite chain of UBSs, no two of which are almost-equivalent.  Thus $\mathcal U$ contains no minimal UBS and it is not clear how to associate a simplex with the almost-equivalence class of $\mathcal U$ in a way that is compatible with the assignment of simplices made available by Theorem~\ref{thm:structureofboundarysets}.  It would be undesirable to ignore UBSs of this type, however, because one must then accept that there are geodesic rays that do not represent simplices of $\simp\mathbf X$.  This is why we disallow such $\mathbf X$ when defining the simplicial boundary.  \emph{Hence $\simp\mathbf X$ is defined as long as every family of pairwise-crossing hyperplanes is finite, and nonempty as long as some halfspace is deep.}
\end{rem}
%%%%%%%%%%%%%%%%%%%%%%%%%%%%%%%%%%%%%%%%%%%%%%%%%%%%%%%%%%
\begin{thm}[Basic properties of $\simp\mathbf X$]\label{thm:boundaryproperties}
If $\simp\mathbf X$ is defined, then:
\begin{compactenum}
\item $\simp\mathbf X=\emptyset$ if and only if for all hyperplanes $H$, each halfspace $\mathfrak h(H),\mathfrak h^*(H)$ is contained in a uniform neighborhood of $H$, i.e. $\mathfrak h(H)$ is \emph{shallow}.
\item Each simplex of $\simp\mathbf X$ is contained in a finite-dimensional maximal simplex.
\item $\simp\mathbf X$ is a flag complex.
\item If $\mathbf X$ is locally finite and is compactly decomposable into at least two unbounded components, then $\simp\mathbf X$ is disconnected.
\item If $\mathbf X$ is unbounded and hyperbolic, then $\simp\mathbf X$ is a totally disconnected set.
\end{compactenum}
\end{thm}

\begin{proof}
If each halfspace is shallow, then $\mathcal W$ cannot contain a UBS, and thus $\simp\mathbf X$ is empty.  If $\simp\mathbf X$ is empty, then $\mathcal W$ does not contain a UBS.  Lemma~\ref{lem:UBSesexist} implies that for each halfspace $\mathfrak h(H)$, every hyperplane $V\subset\mathfrak h(H)$ lies at finite Hausdorff distance from $H$, so that $\mathfrak h(H)$ is shallow.  This proves assertion~$(1)$.
To prove assertion~$(4)$, apply~$(1)$ and Theorem~\ref{thm:boundarysubcomplexes} below.

Assertion~$(2)$ follows from Theorem~\ref{thm:structureofboundarysets}.  Indeed, the proof of Theorem~\ref{thm:structureofboundarysets} shows that if $\simp\mathbf X$ contains an infinite, increasing union of simplices, then $\mathbf X$ contains an infinite family of pairwise-crossing hyperplanes.  In particular, if $\dimension(\mathbf X)=D$, then $\dimension(\simp\mathbf X)\leq D-1$.  Assertion~$(3)$ then follows from~$(2)$ by induction on the size of a clique in $(\simp\mathbf X)^1$, using the definition of a UBS.  Finally, if $\simp\mathbf X$ contains a 1-simplex, then $\crossing X$ contains $K_{p,p}$ for $p$ arbitrarily large, by Theorem~\ref{thm:structureofboundarysets}.  Such an $\mathbf X$ is not hyperbolic, by the proof of~\cite[Theorem~7.6]{HagenQuasiArb}.
\end{proof}

The simplicial boundary behaves in the same way with respect to convex subcomplexes as the visual boundary of a CAT(0) space does with respect to convex subspaces.

\begin{thm}\label{thm:boundarysubcomplexes}
Let $\mathbf Y\subseteq\mathbf X$ be a convex subcomplex.  Then $\simp\mathbf Y$ is a subcomplex of $\simp\mathbf X$.  Moreover, if $\mathbf X=\mathbf X_1\cup\mathbf X_2$, where $\mathbf X_1$ and $\mathbf X_2$ are convex subcomplexes, then $\simp\mathbf X\cong A_1\cup A_2$, where $A_1\cong\simp\mathbf X_1,\,A_2\cong\simp\mathbf X_2$ and $A_1\cap A_2\cong\simp(\mathbf X_1\cap\mathbf X_2)$.
\end{thm}

\begin{proof}
Let $u$ be a simplex of $\simp\mathbf Y$.  Let $\mathcal U$ be a UBS of hyperplanes of $\mathbf Y$ representing $u$.  Each $U\in\mathcal U$ is of the form $U'\cap\mathbf Y$, where $U'$ is a hyperplane of $\mathbf X$ crossing $\mathbf Y$.  Now, by convexity, $U'\cap\mathbf Y$ and $U''\cap\mathbf Y$ contact in $\mathbf Y$ if and only if $U'$ and $U''$ contact (in the same way) in $\mathbf X$.  Moreover, $U'''\cap\mathbf Y$ separates $U'\cap\mathbf Y$ and $U''\cap\mathbf Y$ if and only if $U'''$ separates $U'$ and $U''$ in $\mathbf X$.  Hence $\mathcal U'=\{U'\in\mathcal W\mid U'\cap\mathbf Y\in\mathcal U\}$ is a UBS in $\mathbf X$, representing a simplex $u'$ of $\simp\mathbf X$.  Since elements of $\mathcal U'$ that cross in $\mathbf X$ cross in $\mathbf Y$, Theorem~\ref{thm:structureofboundarysets} implies that $\dimension(u')=\dimension(u)$.

The assignment $u\mapsto u'$ therefore yields a simplicial map $\simp\mathbf Y\rightarrow\simp\mathbf X$ since the intersection of two simplices in $\simp\mathbf Y$ yields a simplex of $\simp\mathbf X$ in the same way as above.  This map is injective, since UBSs in $\mathbf Y$ that correspond to almost-equivalent UBSs in $\mathbf X$ are almost-equivalent.

In particular, if $\mathbf X=\mathbf X_1\cup\mathbf X_2$ is the union of two convex subcomplexes, let $A_1\cong\simp\mathbf X_1,\,A_2\cong\simp\mathbf X_2$.  Then every hyperplane crosses $\mathbf X_1$ or $\mathbf X_2$, and hence $\simp\mathbf X\cong A_1\cup A_2$.  Now, $\mathbf X_1\cap\mathbf X_2$ is convex, being the intersection of convex subcomplexes, and thus $\simp\mathbf X$ contains $\simp(\mathbf X_1\cap\mathbf X_2)=B$.  If $u\subset A_1\cap A_2$ is a simplex, then it is represented by a UBS, all of whose elements cross $\mathbf X_1\cap\mathbf X_2$.  Hence $A_1\cap A_2\subseteq B$, and the other inclusion is similar.
\end{proof}

\subsection{Visible simplices, visible pairs, and combinatorial geodesic completeness}\label{sec:visibility}
Recall that if $\gamma:[0,\infty)\rightarrow\mathbf X$ is a combinatorial geodesic ray, then $\mathcal W(\gamma)$ is a UBS.  The converse is not true, as illustrated by Example~\ref{exmp:invisible}.

\begin{defn}[Visible simplex, visible pair, fully visible, optical space]\label{defn:visible}
The simplex $v\subseteq\simp\mathbf X$ is \emph{visible} if there exists a combinatorial geodesic ray $\gamma$ such that $\mathcal W(\gamma)$ represents $v$.  The pair $u,v\subset\simp\mathbf X$ of visible simplices is a \emph{visible pair} if there exists a combinatorial geodesic $\gamma:\reals\rightarrow\infty$ such that $\mathcal W(\gamma((-\infty,0]))$ represents $u$ and $\mathcal W(\gamma([0,\infty))$ represents $v$. $\mathbf X$ is \emph{fully visible} if each simplex of $\simp\mathbf X$ is visible.  $\mathbf X$ is an \emph{optical space} if, whenever the pair $u,v$ of 0-simplices at infinity is invisible, $u=v$.
\end{defn}

\begin{exmp}[An invisible 0-simplex at infinity]\label{exmp:invisible}
In an eighth-flat, the set of ``horizontal hyperplanes'' is a UBS representing an invisible 0-simplex at infinity: every geodesic ray crosses infinitely many vertical hyperplanes.
\end{exmp}

The following lemma is a simple application of Sageev's construction.

\begin{lem}\label{lem:segments}
Let $\mathcal U=\{U_i\}$ be a nonempty, finite, inseparable set of hyperplanes.  Suppose that for each $U_i$, 
we can label the two distinct halfspaces associated to $U_i$ by $U^+_i,U^-_i$ in such a way that $U_i^+\cap 
U^+_j\neq\emptyset$ and $U^-_i\cap U^-_j\neq\emptyset$ for all $i,j$.  Then there exists a geodesic segment 
$P\rightarrow{\mathbf X}$ such that $\mathcal W(P)=\mathcal U$.
% Let $\mathcal U$ be a nonempty, finite, inseparable set of hyperplanes containing no facing triple.  Then there exists a geodesic segment $P\rightarrow{\mathbf X}$ such that $\mathcal W(P)=\mathcal U$.
\end{lem}

\begin{proof}
 Choose $0$--cubes $x\in\bigcap_iU^-_i,y\in\bigcap_iU^+_i$.  Let $P'$ be a geodesic from $x$ to $y$.  
Since $\mathcal U$ is inseparable, we can modify $P'$, fixing its endpoints, so that if $e,f$ are $1$--cubes 
of $P'$ dual to hyperplanes in $\mathcal U$, and $e'$ is a $1$--cube of $P'$ between $e$ and $f$, then $e'$ is 
dual to a $1$--cube of $\mathcal U$.  Hence there is a subgeodesic $P$ of $P'$ with the desired property.
\end{proof}

Every simplex of $\simp\mathbf X$ is contained in a visible simplex:

\begin{thm}\label{thm:visiblesimplex}
Let $\mathbf X$ be a CAT(0) cube complex with no infinite family of pairwise-crossing hyperplanes.  Then every maximal simplex of $\simp\mathbf X$ is visible.  In particular, each isolated 0-simplex is visible.
\end{thm}

\begin{proof}
Let $\mathcal V$ be a UBS, semi-nested by a 0-cube $f$ at infinity, inducing a partial ordering $\prec$ on 
$\mathcal V$ such that any two elements of $\mathcal V$ cross or are comparable, no infinite $\prec$-chain has 
a maximum, and each infinite $\prec$-chain has a minimum.

By Theorem~\ref{thm:structureofboundarysets}, up to adding or discarding finitely many hyperplanes, we can 
assume that $\mathcal V=\bigsqcup_{i=1}^k\mathcal V_i$, where each $\mathcal V_i$ is a minimal UBS, and every 
element of $\mathcal V_j$ crosses all but finitely many elements of $\mathcal V_i$ when $i<j$.  

Since $\mathbf X$ has no infinite family of pairwise-crossing hyperplanes, each $\mathcal V_j$ contains a 
chain $\{V_j^i\}_{j\ge0}$, and hence contains the inseparable closure of this chain.  
The proof of Lemma~\ref{lem:UBSesexist} implies that this inseparable closure contains a UBS, so by minimality 
of $\mathcal V_j$, we can take $\mathcal V_j$ to be the inseparable closure of $\{V_j^i\}_{j\ge0}$.

By discarding finitely many hyperplanes, we can assume that $V^i_0$ crosses $V^{i'}_0$ 
whenever $i>i'$, i.e. the hyperplanes $V^i_0$ pairwise cross.

For each $U,V\in\mathcal V$, we have $f(U)\cap f(V)\neq\emptyset$.  On the other hand, by the above 
assumptions, $(\mathbf X-f(U))\cap(\mathbf X-f(V))\neq\emptyset$.

Indeed, let $U\in\mathcal V_i$ and $V\in\mathcal V_j$, with $i\leq j$.  Then $V^i_0\subseteq \mathbf X-f(U)$ 
and $V^j_0\subseteq\mathbf X-f(V)$.  Since $V^i_0\cap V^j_0\neq\emptyset$, we have that $(\mathbf 
X-f(U))\cap(\mathbf X-f(V))\neq\emptyset$, as required.

\textbf{A sequence of segments:}  For any $n$, the inseparable closure of 
$\bigcup_{j=1}^k\{V_j^0,\ldots,V_j^n\}$ therefore satisfies the hypotheses of Lemma~\ref{lem:segments}.  Hence 
we can choose a $\prec$--minimal hyperplane $V_0\in\mathcal V$ and subsets $\mathcal V_0\subset\mathcal 
V_1\subset\ldots\subset\mathcal V_r\subset\ldots$ such that: $|\mathcal V_r|=r+1$, and $\bigcup_{r\geq 
0}\mathcal V_r=\mathcal V$, and for each $r\geq 0$, there exists a 0-cube $c_r\in N(V_0)$ and a geodesic 
segment $P_r\rightarrow{\mathbf X}$ emanating from $c_r$ such that $\mathcal W(P_r)=\mathcal V_r$.

% 
% By definition, there exists at least one $\prec$-minimum $V_0\in\mathcal V$.  Fix $V_0$ and choose $\mathcal 
% V_0\subset\mathcal V_1\subset\ldots\subset\mathcal V_r\subset\ldots$ such that $|\mathcal V_r|=r+1$, each 
% $\mathcal V_r$ is inseparable, and $\bigcup_{r\geq 0}\mathcal V_r=\mathcal V$. Also, make this choice so that 
% the hyperplane $V_0\in\mathcal V_0$ is $\prec$-minimal in $\mathcal V_0$ and has the property that every 
% element of $\bigcup_r\mathcal V_r$ either lies in $f(V_0)$ or crosses $V_0$.  Note that if $V_0$ is maximal and 
% minimal, then every hyperplane crosses $V_0$, and we can choose $\mathcal V_0=\{V_0\}$.
% 
% By Lemma~\ref{lem:segments}, 

Let $\Omega$ be the graph whose vertex-set is $\bigcup_{r\geq 0}\mathcal R_r$, where $\mathcal R_r$ is the set of geodesic segments $P_r$ whose 1-cubes are dual to exactly the hyperplanes in $\mathcal V_r$.  Let an edge join $P_r\in\mathcal R_r$ to $P_{r+1}\in\mathcal R_{r+1}$ if and only if $P_r\subset P_{r+1}$.  The graph $\Omega$ is infinite and locally finite.  Indeed, local finiteness follows from the absence of infinite cubes in ${\mathbf X}$. If $P_r$ is a geodesic segment crossing $\mathcal V_r$ and extending to $P_{r+1}$, then let $c'$ be the 1-cube of $P_{r+1}-P_r$.  If $V$ is dual to the terminal 1-cube of $P_r$ adjacent to $c'$, and $V_1,\ldots,V_n\in\mathcal V_{r+1}-\mathcal V_r$ are dual to 1-cubes $c_1,\ldots,c_n$ extending $P_r$, then since $\mathcal V$ contains no facing triple, the set $V,V_1,\ldots,V_n$ contains a collection of $n$ pairwise-crossing hyperplanes, and thus $n$ is bounded above by the dimension of the maximal cube containing $c'$.  %\footnote{Note that the above argument shows that $\Omega$ is \emph{uniformly} locally finite when ${\mathbf X}$ is finite-dimensional.  To prove the claim as stated, we do not need the pairwise-crossing conclusion above; we need only observe that there are finitely many cubes containing $c'$.}
Hence each $P_r$ extends in finitely many ways in each direction, and thus $\Omega$ is an infinite, locally finite graph.

\textbf{Taking a limit to obtain a ray:}  If $\Omega$ has an infinite connected component, then by K\"{o}nig's lemma, there is an infinite path $\gamma\rightarrow{\mathbf X}$ that crosses every hyperplane of $\bigcup_{r\geq 0}\mathcal V_r=\mathcal V$ exactly once.  Hence $\gamma$ is a geodesic ray whose initial 1-cube is dual to the given hyperplane $V_0$.  By construction, any hyperplane crossing $\gamma$ belongs to some (and hence all but finite many) $\mathcal V_r$, whence $\mathcal V$ represents a visible simplex.

\textbf{When the limit does not exist:}  If every component of $\Omega$ is finite, then for each $n$ and each $\phi\in\mathcal R_n$, there exists a maximal $k(\phi)\geq n$ such that $\phi$ extends to an element of $\mathcal R_{k(\phi)}$.  Suppose there is a hyperplane $U(\phi)$ separating $\phi$ from infinitely many $\psi\in\bigcup_{m>k(\phi)}\mathcal R_m$.  Since it separates the initial point of $\phi$ from the initial point of infinitely many such segments $\psi$, and $\phi$ and each such $\psi$ emanate from $N(V_0)$, the hyperplane $U(\phi)$ crosses $V_0$.  Moreover, since $U(\phi)$ separates $\phi$ from each $\psi\in\bigcup_{m\geq k(\phi)}\mathcal R_m$, the hyperplane $U(\phi)\not\in\mathcal V$, since each hyperplane of $\mathcal V$ is dual to a 1-cube of infinitely many of the segments $\psi$.

Hence we can choose a subsequence $(i_j)_{j\geq 1}$ and, for each $j\geq 1$, a segment $\phi_j\in\mathcal R_{i_j}$ and a hyperplane $U_j\not\in\mathcal V$ that separates $\phi_j$ from $\phi_k$ for all $k>j$, and thus crosses $V_0$.  See the left side of Figure~\ref{fig:diagonalquarterflat}.
\begin{figure}[h]
\begin{center}
  \includegraphics[width=3.5in]{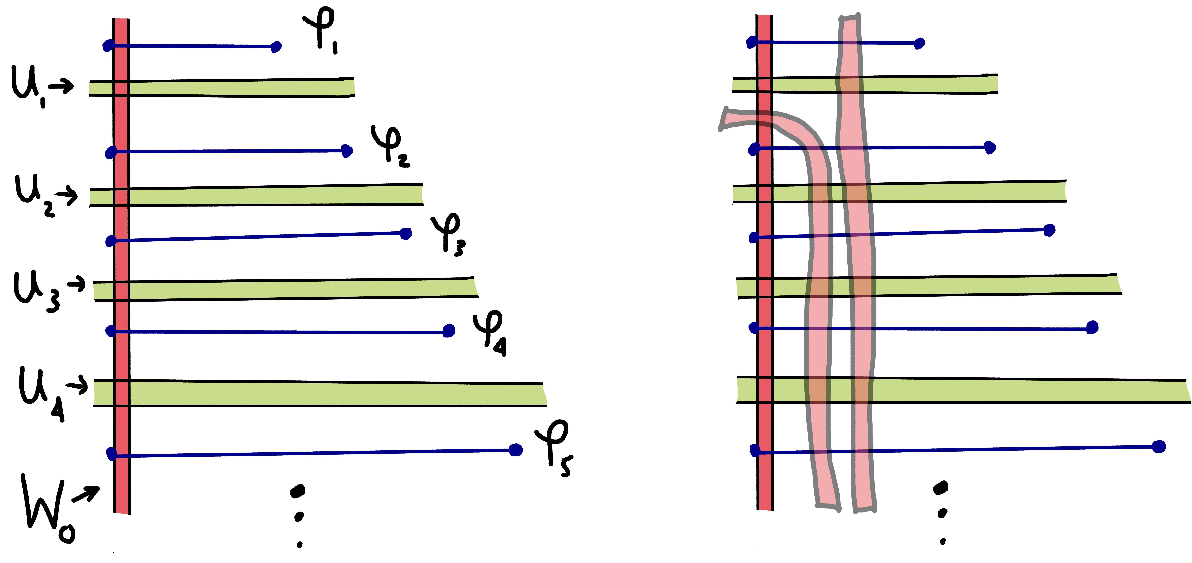}\\
  \caption{The situation arising in the absence of a ray associated to a simplex at infinity.}\label{fig:diagonalquarterflat}
\end{center}
\end{figure}
Note that if $V\in\mathcal V_{i_j}$, then $V$ crosses $\phi_k$ for all $k\geq j$, and hence $V$ crosses $U_k$ for all $k\geq j$.  On the other hand, if some hyperplane $U$ separates $U_i$ from $U_j$ with $i<j$, then $U$ must cross every hyperplane crossed by both $U_i$ and $U_j$.  Hence there is an infinite, inseparable, partially-ordered set $\mathcal U$, containing all of the $U_j$, such that, for all but finitely many $V\in\mathcal V$, $V$ crosses all but finitely many $U\in\mathcal U$.

\textbf{Finding a larger simplex:} The set $\mathcal V'=\mathcal V\cup\mathcal U$ is a UBS containing $\mathcal V$ and representing a simplex $v'$ of $\simp\mathbf X$ that properly contains the simplex $v$ represented by $\mathcal V$.  Hence each maximal simplex of $\mathbf X$ is visible.  By Theorem~\ref{thm:boundaryproperties}, the simplex $v$ is contained in a maximal simplex and hence $\mathcal V$ extends to a UBS almost-equivalent to $\mathcal W(\gamma)$ for some geodesic ray $\gamma$.

\textbf{Proof that $U(\phi)$ exists:}  Suppose $\phi\in\mathcal R_n$ does not extend to any $\psi\in\mathcal R_m$ for $m>k(n)$.  Either infinitely many $\psi\in\bigcup_{m>k(n)}\mathcal R_m$ are separated from $\phi$ by a hyperplane $U\not\in\mathcal V$, or we are in the following situation for all but finitely many such $\psi$.  Let $m>k(n)\geq n+1$, so that $\psi=\phi'\alpha$, where $\phi'$ is a length-$n$ geodesic segment crossing $\mathcal V_n$ and $\alpha$ is a geodesic segment crossing $\mathcal V_m-\mathcal V_n$.  Then $\phi(n)$ and $\phi'(n)$ are separated by a hyperplane $U$ that crosses each of the hyperplanes in $\mathcal V_m-\mathcal V_n$, for otherwise $\phi$ would extend to a segment $\phi\alpha\in\mathcal R_m$.  If $U\not\in\mathcal V$, then we are done, since $U$ must then separate $\phi$ from $\psi$.  Otherwise, $U$ crosses $\phi$ or $\phi'$ and thus $U\in\mathcal V_n$.  Hence $\phi\cap\phi'\neq\emptyset$, so that some sub-segment of $\phi$ extends to a segment in $\mathcal R_m$.  If this holds for arbitrarily large values of $m$, with $\phi$ fixed, then K\"{o}nig's lemma yields a ray whose initial 0-cube lies on $\phi$ and whose set of dual hyperplanes is $\mathcal V$.  Thus we may suppose that there exists $U\not\in\mathcal V$ that separates $\phi$ and $\psi$.
\end{proof}

Theorems~\ref{thm:boundaryproperties} and~\ref{thm:visiblesimplex} imply that a hyperbolic cube complex is fully visible.  More generally:

\begin{cor}\label{cor:nowayoffthetrain}
Let $\gamma$ be a rank-one combinatorial geodesic ray representing a simplex $v$ of $\simp\mathbf X$.  Then $v$ is an isolated 0-simplex.
\end{cor}

\begin{proof}
Suppose that $\mathcal W(\gamma)=\mathcal V_1\sqcup\mathcal V_2$, where $\mathcal V_1$ and $\mathcal V_2$ are disjoint UBSs.  By Theorem~\ref{thm:structureofboundarysets}, each element of $\mathcal V_1$ crosses all but finitely many elements of $\mathcal V_2$.  By the proof of Theorem~\ref{thm:trichotomy2}, $\gamma$ bounds an eighth-flat, a contradiction.
\end{proof}

\begin{lem}[Equating basepoints]\label{lem:samestart}
Let $\gamma_1,\gamma_2:[0,\infty)\rightarrow\mathbf X$ be combinatorial geodesic rays.  Then there exist combinatorial geodesic rays $\gamma'_1,\gamma'_2:[0,\infty)\rightarrow\mathbf X$ such that $\gamma'_1(0)=\gamma'_2(0)$ and $\gamma'_i$ is almost-equivalent to $\gamma_i$ for $i\in\{1,2\}$.
\end{lem}

\begin{proof}
Since only finitely many hyperplanes separate $\gamma_1(0)$ from $\gamma_2(0)$, there exists $T\geq 0$ such that for all $t\geq T$, the hyperplane dual to $\gamma_1([t,t+1])$ does not separate $\gamma_1(0)$ from $\gamma_2(0)$.  Choose a segment $P$ joining $\gamma_2(0)$ to $\gamma_1(T)$.  The concatenation $P\gamma_1([T,\infty))=\gamma'_1$ is almost-equivalent to $\gamma_1$, and $\gamma'_1(0)=\gamma_2(0)$.  Any hyperplane $H$ crossing $\gamma([T,\infty))$ and $P$ separates $\gamma_2(0)$ from $\gamma_1(T)$.  Since $H$ crosses $\gamma_1$, $H$ cannot separate $\gamma_1(0)$ from $\gamma_1(T)$ and hence separates $\gamma_1(0)$ from $\gamma_2(0)$, but this contradicts the choice of $T$.  Thus $\gamma'_1$ has the desired properties, and we take $\gamma'_2=\gamma_2$.  See Figure~\ref{fig:equating}.
\begin{figure}[h]
  \includegraphics[width=2.25in]{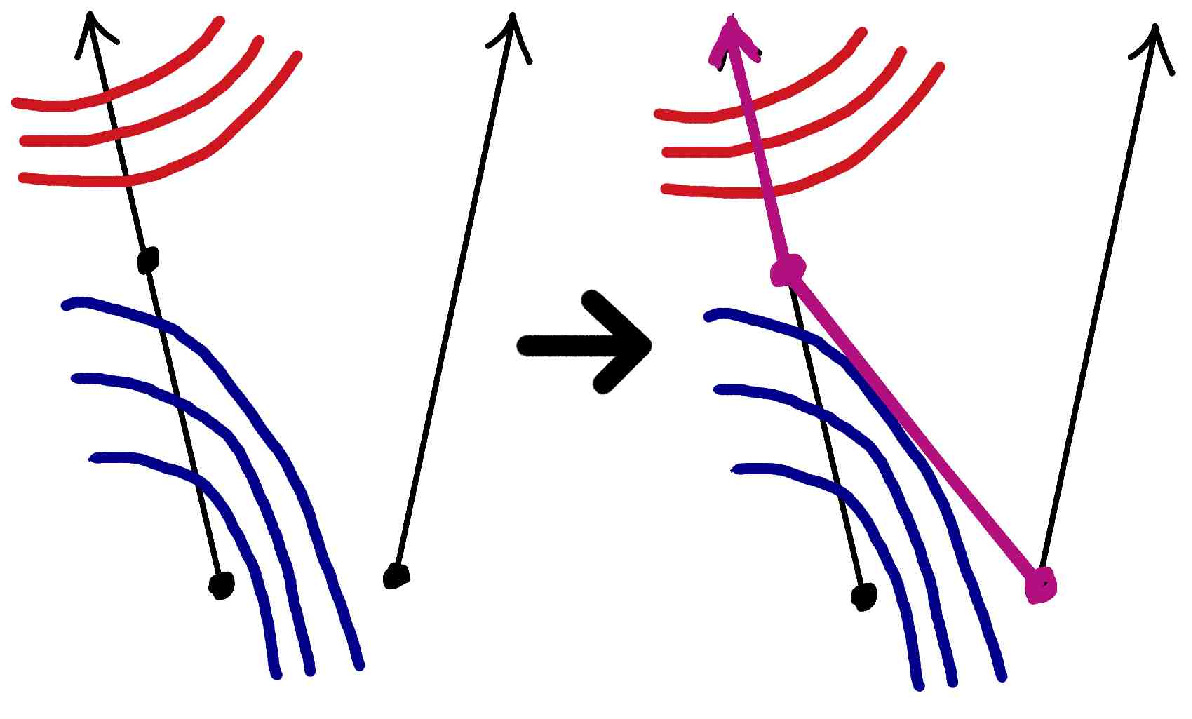}\\
  \caption{Equating basepoints.}\label{fig:equating}
\end{figure}
\end{proof}

\begin{lem}[Folding Lemma]\label{lem:folding}
Let $\gamma_1,\gamma_2:[0,\infty)\rightarrow\mathbf X$ be geodesic rays with common initial 0-cube.  Suppose that some hyperplane $H$ crosses $\gamma_1$ and $\gamma_2$, so that $H$ is dual to $\gamma_1([s,s+1])$ and $\gamma_2([t,t+1])$ with $0\leq s\leq t$.  Then there exist geodesic rays $\gamma_1',\gamma_2':[0,\infty)\rightarrow\mathbf X$ such that $\gamma_1([s+1,\infty))\subset\gamma'_1,\,\gamma_2([s+1,\infty))\subset\gamma'_2$, and $\gamma_1',\gamma_2'$ coincide in an initial segment $P$ whose initial 0-cube is $\gamma_1(0)$ and whose terminal 1-cube is dual to $H$.
\end{lem}

\begin{proof}
If $a_1=\gamma_1(s+1)=\gamma_2(t+1)=a_2$, then replace $\gamma_1([0,s+1])$ and $\gamma_2([0,t+1])$ by a geodesic segment $P$ joining $a_1$ to $\gamma_1(0)=\gamma_2(0)=b$.  If not, then let $m$ be the median of $b, a_1,a_2$.  Since $H$ separates both $a_1$ and $a_2$ from $b$, $H$ must separate $m$ from $b$.  Since $a_1,a_2\in N(H)$ and that subcomplex is convex, $m\in N(H)$.  Let $P$ be a geodesic path joining $b$ to $m$, so that the terminal 1-cube of $P$ is dual to $H$.  Choose geodesic segments $P_1,P_2\rightarrow N(H)$ joining $m$ to $a_1$ and to $a_2$ respectively.  Then replace $\gamma_1([0,s+t])$ by $PP_1$ and $\gamma_2([0,t+1])$ by $PP_2$.  See Figure~\ref{fig:folding}.
\end{proof}

\begin{figure}[h]
  \includegraphics[width=2.25in]{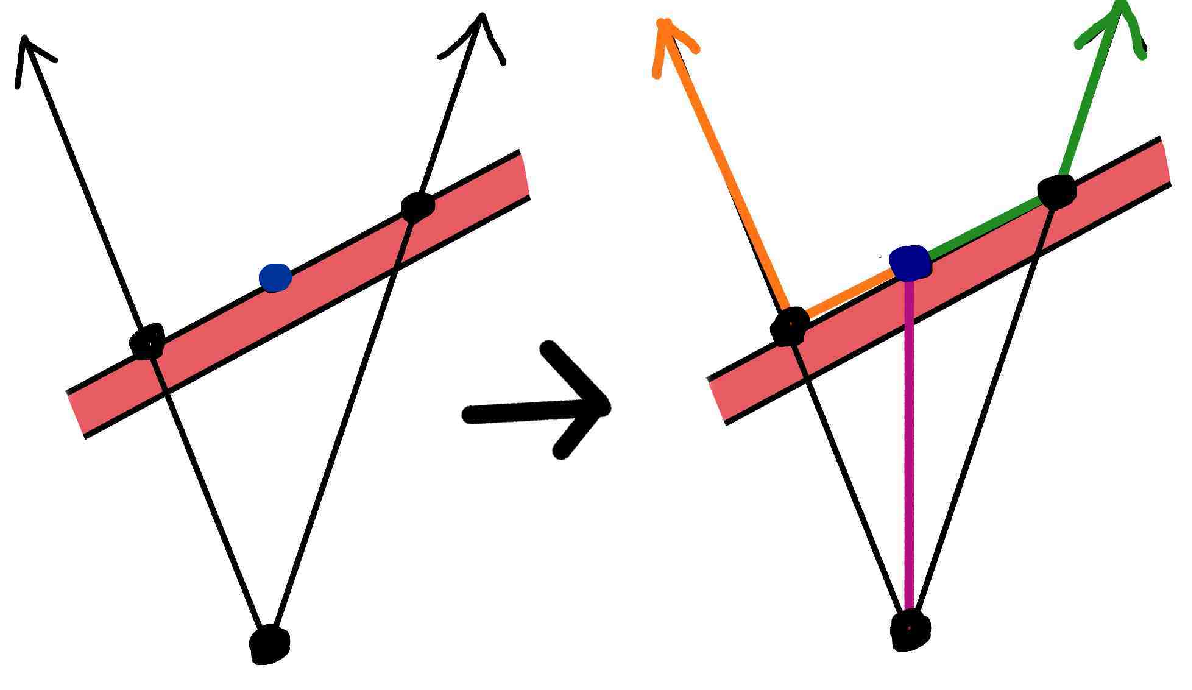}\\
  \caption{Folding.}\label{fig:folding}
\end{figure}

\begin{thm}\label{thm:visiblefracflat}
Let $v\subset\mathbf X$ be an $n$-simplex of $\simp\mathbf X$, each of whose 0-simplices is visible.  Then there exists an isometric embedding $\mathbf R\rightarrow\mathbf X$, where $\mathbf R$ is the standard tiling of $[0,\infty)^{n+1}$ by $(n+1)$-cubes.
\end{thm}

\begin{proof}
Let $v=[v_0,v_1,\ldots,v_n]$ and let $\gamma_0,\gamma_1,\ldots,\gamma_n$ be combinatorial geodesic rays, with $\gamma_i$ representing $v_i$.  By finitely many applications of Lemma~\ref{lem:samestart}, we can assume that the $\gamma_i$ all emanate from a common 0-cube $x_0$.  Since $v_i\cap v_j=\emptyset$ for $i\neq j$, we have $|\mathcal W(\gamma_i)\cap\mathcal W(\gamma_j)|<\infty$ for all $i\neq j$, so that by finitely many applications of the folding lemma, followed by the removal of a finite common initial segment of the $\gamma_i$, we can assume that $\mathcal W(\gamma_i)\cap\mathcal W(\gamma_j)=\emptyset$ when $i\neq j$.

For any $i\neq j$, let $H_k^i$ be dual to the $k^{th}$ 1-cube of $\gamma_i$ and let $H_m^j$ be dual to the $m^{th}$ 1-cube of $\gamma_j$.  If $H_k^i$ does not cross $H_m^j$, then, since $v_j$ is a 0-simplex, there are infinitely many hyperplanes $H$ crossing $\gamma_j$ and separated from $H_k^i$ by $H_m^j$.  But every hyperplane crossing $\gamma_i$ crosses all but finitely many of the hyperplanes crossing $\gamma_j$, and thus $H_m^j\bot H_k^i$ for all $m,k\geq 1$ and all $i\neq j$.

By induction, $\mathbf R'\cong\prod_{i=0}^{n-1}\gamma_i$ is an isometrically embedded subcomplex of $\mathbf X$ (the base case amounts to the existence of $\gamma_0$, which is assumed).  Indeed, every hyperplane $H$ either crosses exactly one of $\mathbf R'$ or $\gamma_n$, or $H$ crosses neither of those subcomplexes.  Each hyperplane crossing $\mathbf R'$ crosses each hyperplane crossing $\gamma_n$.  Hence $\mathbf R\cong\mathbf R'\times\gamma_n$ isometrically embeds in $\mathbf X$.
\end{proof}

\subsubsection{Visible pairs}  A notable difference between the combinatorial metric and the CAT(0) metric is that distinct 0-simplices of the same simplex of $\simp\mathbf X$ are joined by a bi-infinite combinatorial geodesic in $\mathbf X$, but the corresponding points on the visual boundary (see Section~\ref{sec:titscompare}) are not joined by a CAT(0) geodesic.  More precisely:

\begin{thm}\label{thm:opticalspace}
Let $\mathbf X$ be fully visible.  Then $\mathbf X$ is an optical space.
\end{thm}

\begin{proof}
Let $v,w$ be visible simplices of $\simp\mathbf X$.  We shall show that either $u,v$ is a visible pair, or $u\cap v\neq\emptyset$.  In the latter case, if $u,v$ are 0-simplices, then $u=v$, and the claim follows.

Let $\gamma_u$ and $\gamma_v$ be combinatorial geodesic rays representing $u$ and $v$ respectively.  By Lemma~\ref{lem:samestart}, we can take $\gamma_u$ and $\gamma_v$ to have the same initial 0-cube $\gamma_u(0)=\gamma_v(0)$.  Let $\mathcal C=\mathcal W(\gamma_u)\cap\mathcal W(\gamma_v)$.  Then $\mathcal C$ contains no facing triple, being the intersection of UBSs.  If $C_1,C_2\in\mathcal C$ are separated by $H$, then $H$ crosses $\gamma_u$ and $\gamma_v$, and hence belongs to $\mathcal C$.  Therefore, if $\mathcal C$ is infinite, then it is a UBS representing a simplex $c\subset u\cap v$, and we are done.  Hence suppose that $\mathcal C$ is finite.  If $|\mathcal C|=0$, then define $\alpha:\reals\rightarrow\infty$ by $\alpha(t)=\gamma_u(t)$ for $t\geq0$ and $\gamma_v(t)$ otherwise, yielding a bi-infinite geodesic showing that $u,v$ is a visible pair.  If $|\mathcal C|=n>0$, then choose $C\in\mathcal C$.  By the folding lemma, we can assume that there exists $T\geq 0$ such that $\gamma_u(t)=\gamma_v(t)$ for $t\leq T$ and that $C$ is dual to the 1-cube $\gamma_u([T,T+1])$.  Since $\mathcal C$ is finite, we can choose $C$ so that $\mathcal W(\gamma_u([T+1,\infty)))\cap\mathcal W(\gamma_v([T+1,\infty)))=\emptyset$.  Now simply remove the path $\gamma_u([0,T))$ from $\gamma_u\cup\gamma_v$ and define $\alpha$ as before.
\end{proof}

\begin{rem}
If $\gamma_v,\gamma_u$ are combinatorial geodesic rays representing simplices $u,v$ of $\mathbf X$, and $u\subset v$, then by Lemma~\ref{lem:samestart}, we can choose $\gamma_u$ and $\gamma_v$ so that $\mathcal W(\gamma_u)\subset\mathcal W(\gamma_v)$.
\end{rem}

\subsubsection{Combinatorial geodesic completeness}  The notion of combinatorial geodesic completeness is needed in stating Theorem~\ref{thm:divergenceofgroup}.  In that context, it coincides with CAT(0) geodesic completeness (see~\cite[Chapter~7]{HagenPhD} for a discussion of how combinatorial geodesic completeness is used in a mostly combinatorial proof of that theorem).  However, in general, the notions differ, and their relationship warrants discussion.

\begin{defn}[Combinatorial geodesic completeness]\label{defn:combgeodesiccomplete}
Suppose $\mathbf X$ has no infinite family of pairwise-crossing hyperplanes.  Then $\mathbf X$ is \emph{combinatorially geodesically complete} if the following holds for every maximal set $\{W_1,\ldots,W_n\}$ of pairwise-crossing hyperplanes.  For each $1\leq i\leq n$, let $\chi(W_i)$ be a halfspace associated to $W_i$.  Then $\bigcap_{i=1}^n\chi(W_i)$ contains 0-cubes arbitrarily far from $\bigcap_{i=1}^nN(W_i)$, for any choice of $\chi$.
\end{defn}

\begin{rem}[Combinatorial geodesic completeness vs. essentiality and visibility]
The cube complex $\reals\times[-\frac{1}{2},\frac{1}{2}]$ shows that combinatorial geodesic completeness does not imply essentiality.  A diagonal half-flat shows that essentiality does not imply combinatorial geodesic completeness, though the two notions are equivalent for trees.  Gluing a 2-cube to $\reals\times[-\frac{1}{2},\frac{1}{2}]$ along any proper face yields a an optical space that is not combinatorially geodesically complete.  Let $\mathbf E$ be an eight-flat.  At each 0-cube $e$ of $\mathbf E$ that is contained in a single 2-cube, attach a ray by identifying its initial point with $e$ to form a cube complex that is combinatorially geodesically complete but not fully visible.
\end{rem}

\begin{rem}[Combinatorial vs. CAT(0) geodesic completeness]\label{rem:geodesicallycomplete}
A CAT(0) space $Y$ is \emph{geodesically complete} if each geodesic path $P$ extends to a bi-infinite geodesic $\gamma:\reals\rightarrow Y$.

If $\mathbf X$ is CAT(0) geodesically complete, then $\mathbf X$ is combinatorially geodesically complete, since each diagonal in a maximal cube is a CAT(0) geodesic path.  However, there exists a combinatorially geodesically complete cube complex $\mathbf C$ that is not CAT(0) geodesically complete: let $\mathbf C_0\cong[0,\infty)^2$ be a non-diagonal quarter-flat, and let $\mathbf C_1,\mathbf C_2$ be copies of $[0,\infty)\times[0,1]$.  $\mathbf C$ is formed by identifying $\{0\}\times[0,1]\subset\mathbf C_1$ with $\{0\}\times[0,1]\subset\mathbf C_0$ and $\{0\}\times[0,1]\subset\mathbf C_2$ with $[0,1]\times\{0\}$ in the obvious cubical way, as shown in Figure~\ref{fig:noextend}.  For $a\geq 1$, the CAT(0) geodesic line segment joining $(0,a)$ to $(a,0)$ does not extend to a bi-infinite CAT(0) geodesic.  $\mathbf C$ is, however, combinatorially geodesically complete.
\begin{figure}[h]
  \includegraphics[width=2.25in]{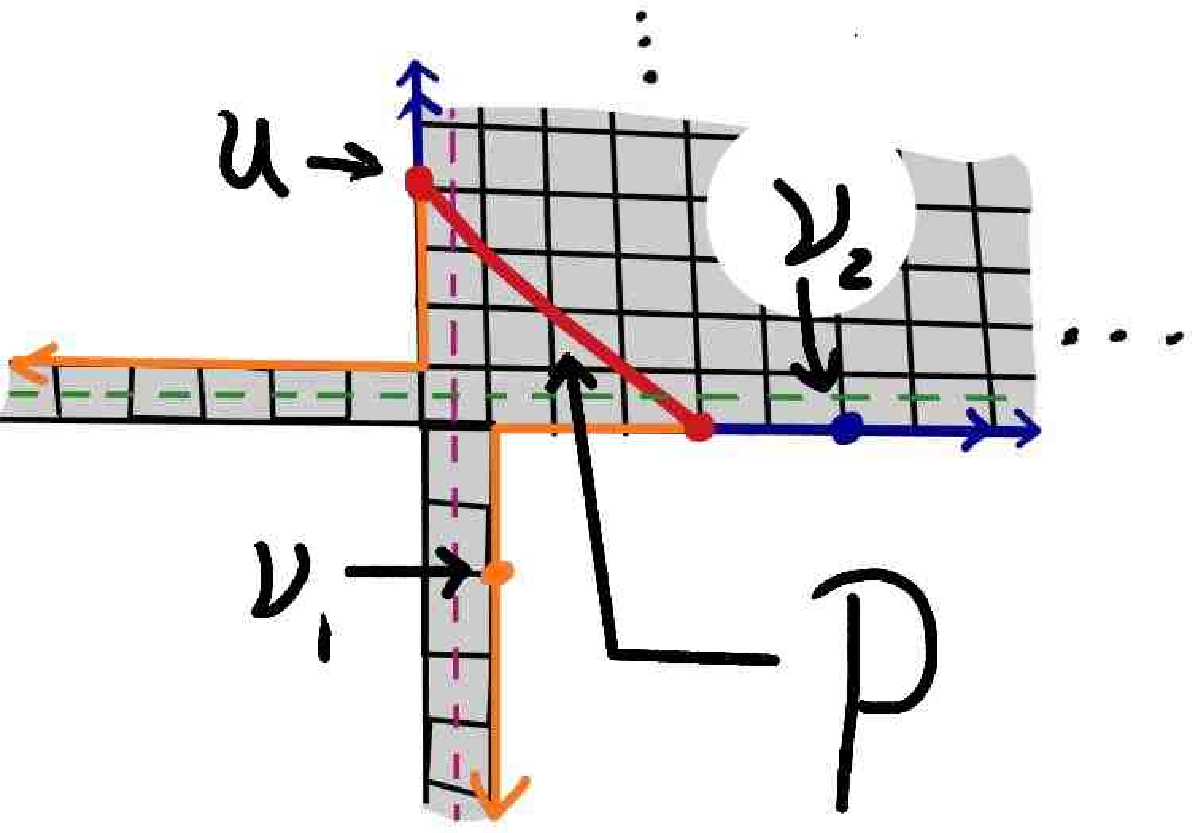}\\
  \caption{The CAT(0) geodesic segment $P$ does not extend to a geodesic line. In any of the possible lines extending $P$, there exists a point ($v_1$ or $v_2$) closer to $u$ in $\mathbf C$ than in the putative geodesic.  However, each 2-cube $s$ determines four quarter-spaces, each of which contains 0-cubes arbitrarily far from $s$.}\label{fig:noextend}
\end{figure}

\end{rem}

The choice of nomenclature comes from the following characterization of combinatorial geodesic completeness.

\begin{lem}\label{lem:combgeodesiccomplete}
If $\mathcal W\neq\emptyset$, then $\mathbf X$ is combinatorially geodesically complete if and only if, for each combinatorial geodesic segment or ray $P\rightarrow\mathbf X$, there exists a bi-infinite combinatorial geodesic $\gamma:\reals\rightarrow\mathbf X$ such that $P\subset\gamma$.
\end{lem}

\begin{proof}
Suppose that every combinatorial geodesic segment in $\mathbf X$ extends to a geodesic.  Let $W_1,\ldots,W_n$ be a finite set of pairwise-crossing hyperplanes and let $\chi$ choose a halfspace associated to each $W_i$.  Let $A_0$ be a combinatorial geodesic segment in the unique maximal cube of $\bigcap_{i=1}^nN(W_i)$, so that the initial 0-cube of $A_0$ lies in $\bigcap_i(\mathbf X-\chi(W_i))$ and the terminal 0-cube lies in $\bigcap_{i=1}^n\chi(W_i)$.  By assumption, there is a combinatorial geodesic $\gamma:\reals\rightarrow\infty$ that contains $A_0$ and hence crosses each $W_i$.  Parameterize $\gamma$ so that $\gamma(0)$ is the terminal 0-cube of $A_0$ and $\gamma(-n)$ is the initial 0-cube of $A_0$.  Then $\gamma([0,\infty))$ contains 0-cubes arbitrarily far from $\gamma(0)$ and hence from the cube $\bigcap_iN(W_i)$.  Moreover, $\gamma([0,\infty))$ cannot cross any $W_i$, since $\gamma$ is a geodesic, and thus $\mathbf X$ is combinatorially geodesically complete.

Now suppose that $\mathbf X$ is combinatorially geodesically complete.  Let $P:[0,n]\rightarrow\infty$ be a finite combinatorial geodesic path.  Suppose first that $n=0$.  Let $c$ be a maximal cube containing $P=P(0)$.  Now, $\bigcap_{W\in\mathcal W(c)}P(0)(W)$ is a maximal finite intersection of halfspaces, and so contains 0-cubes arbitrarily far from $c$ and hence from $P(0)$, and the same holds for $\bigcap_{W\in\mathcal W(c)}(\mathbf X-P(0)(W))$.  Thus, for any $N\geq 0$, $P$ is contained in a geodesic segment $\gamma_N$ such that each of the above two intersections 0-cubes of $\gamma_N$ at distance $N$ from $P(0)$.  Applying the same argument to the paths consisting of the endpoints of $\gamma_N$ shows that $\gamma_{N+1}$ may be chosen with $\gamma_N\subsetneq\gamma_{N+1}$, whence $P$ extends to a bi-infinite combinatorial geodesic.

Suppose for some $n\geq 0$ that every geodesic segment of length at most $n$ extends to a bi-infinite geodesic.  Let $P:[0,n+1]\rightarrow\mathbf X$ be a combinatorial geodesic segment, so that $P=P'c$, where $c$ is the terminal 1-cube.  By induction, there is a combinatorial geodesic $\gamma$ containing $P'$.  Let $\alpha$ be the sub-ray of $\gamma$ whose initial 0-cube is the initial 0-cube of $c$ and which contains $P'$.  Let $\chi$ be a (finite) maximal cube containing $c$, and let $\chi$ be the terminal 0-cube of $c$.  If $\alpha$ contains $c$, then we are done.  Hence suppose that $\chi$ differs from every 0-cube of $\alpha$ on the hyperplane $W$ dual to $c$.  Since $\bigcap_{V\in\mathcal W(c)}\chi(V)$ contains 0-cubes arbitrarily far from $\chi$, there is a geodesic ray $\beta\subset\bigcap_{V\in\mathcal W(c)}\chi(V)$ whose initial 0-cube is $\chi$.  Let $\gamma'=\alpha\cup c\cup\beta$.  Any hyperplane $U$ that meets $\gamma$ in two distinct 1-cubes must cross $\alpha$ and $c\cup\beta$.  Now, each element of $\mathcal W(c)$ separates $\alpha$ from $\beta$, and so $U$ must cross each $V\in\mathcal W(C)$.  Since $\mathcal W(c)$ is a maximal family of pairwise-crossing hyperplanes, this is impossible.  This argument also shows, mutatis mutandis, that every combinatorial geodesic ray extends to a geodesic.
\end{proof}

\subsection{Products and joins}\label{sec:productsandjoins}
Theorem~II.9.24 of~\cite{BridsonHaefliger} equates the existence of a spherical join decomposition of the Tits boundary of a CAT(0) space with the existence of a nontrivial product decomposition of that space.  This generalizes a result of Schroeder on Hadamard manifolds (see Appendix~4 of~\cite{BaGS}).  Here we prove an analogous result about simplicial join decompositions of $\simp\mathbf X$.  Remark~\ref{rem:PandJ} compares Theorem~\ref{thm:productsandjoins} to the corresponding CAT(0) result.

Given simplicial complexes $A,B$, we denote by $A\star B$ their \emph{(simplicial) join}, i.e. the simplicial complex formed from $A\sqcup B$ by adding a simplex $[a_1,\ldots,a_n,b_1,\ldots,b_m]$ whenever $a_1,\ldots,a_n\in A$ and $b_1,\ldots,b_m\in B$ are pairwise-distinct 0-simplices.

\begin{thm}\label{thm:productsandjoins}
Let $\mathbf X$ be an essential CAT(0) cube complex with no infinite collection of pairwise-crossing hyperplanes.  Consider the following two statements:
\begin{compactenum}
\item There exist unbounded convex subcomplexes $\mathbf X_1$ and $\mathbf X_2$ such that $\mathbf X\cong\mathbf X_1\times\mathbf X_2$.
\item There exist disjoint, nonempty subcomplexes $A_1,A_2\subset\simp\mathbf X$ such that $\simp\mathbf X\cong A_1\star A_2$.
\end{compactenum}
Then $(1)\Rightarrow(2)$.  If, in addition, $\mathbf X$ is fully visible, then $(2)\Rightarrow(1)$.
\end{thm}

\begin{proof}
\textbf{$(1)\Rightarrow(2)$:}  Let $\mathcal W_i$ be the set of hyperplanes of ${\mathbf X}_i$.  Then $\mathcal W=\mathcal W'_1\sqcup\mathcal W'_2$, where $\mathcal W'_1=\{V\times{\mathbf X}_2\,:\,V\in\mathcal W_1\}$ and, similarly, $\mathcal W'_2=\{{\mathbf X}_1\times H\,:\,H\in\mathcal W_2\},$ and for all $V\in\mathcal W'_1$ and $H\in\mathcal W'_2$, the hyperplanes $V$ and $H$ cross.  This implies that any 0-simplex of $\simp{\mathbf X}$ lies in the image of $\simp{\mathbf X}_i$ for some $i\in\{1,2\}$.  From the above description of the crossings, the claimed join structure is immediate: every 0-simplex in the image of $\simp\mathbf X_1$ is adjacent to every simplex in the image of $\simp\mathbf X_2$.  Essentiality ensures that the $A_i$ are nonempty.

\textbf{$(2)\Rightarrow(1)$:}  This is proven by loose analogy with the proof of~\cite[Proposition~II.9.25]{BridsonHaefliger}.  Fix a base 0-cube $x$.  For each 0-simplex $u\in\simp{\mathbf X}$, appealing to full visibility, choose a geodesic ray $\gamma:[0,\infty)\rightarrow\mathbf X$ such that $\gamma(0)=x$ and $\mathcal W(\gamma)$ represents the 0-simplex $u.$  For $i\in\{1,2\}$, let $\mathbb T_i$ be the graph in ${\mathbf X}^{1}$ obtained by taking the union of all rays $\gamma$ such that $\gamma(0)=x$ and $\gamma$ represents a 0-simplex of $A_1$.  Let ${\mathbf X}_i(x)=\mathbf X_i$ be the convex hull of $\mathbb T_i$.

\textbf{${\mathbf X}_i(x)$ is independent of basepoint in ${\mathbf X}_1(x)$:}  Let $x'\in{\mathbf X}_1(x)$ be a 0-cube, and construct the complex ${\mathbf X}_i(x')$ from rays $\gamma'$ based at $x'$ and representing 0-simplices of $A_i$, exactly as was done for ${\mathbf X}_i$.  Since ${\mathbf X}_1(x)$ is convex, any hyperplane $V$ that separates $x$ from $x'$ crosses ${\mathbf X}_1(x)$ and ${\mathbf X}_1(x')$.  Indeed, if $P$ is a geodesic segment joining $x$ to $x'$, then $P\subset\mathbf X_1(x)\cap\mathbf X_1(x')$, since $x,x'\in\mathbf X_1(x)\cap\mathbf X_1(x')$.

If $V$ is a hyperplane crossing $\mathbf X_1(x)$, then $V$ crosses a geodesic ray $\gamma$ emanating from $x$ and representing a 0-simplex of $A_1$.  There is a geodesic ray $\gamma'$ emanating from $x'$ that is almost-equivalent to $\gamma$.  If $V$ crosses $\gamma$ and $\gamma'$, then $V$ crosses $\mathbf X_1(x)\cap\mathbf X_1(x')$.  If $V$ separates $x$ from $x'$, then $V$ crosses $\mathbf X_1(x)\cap\mathbf X_1(x')$, as shown above.  If not, then $V$ separates $\gamma'$ from an infinite subray $\alpha$ of $\gamma$.  Hence $V$ crosses all but finitely many of the hyperplanes that cross $\gamma$ and $\gamma'$.  Choose a ray $\gamma_1$ such that $\gamma_1(0)=x'$ and $\gamma_1$ is almost-equivalent to $\gamma$, and $\gamma_1$ contains a 1-cube dual to $V$.  Hence $V$ crosses $\mathbf X_1(x)$.  Thus $\mathcal W(\mathbf X_1(x))=\mathcal W(\mathbf X_1(x'))$, since $\mathbf X_1(x)$ is the smallest convex subcomplex containing all rays representing 0-simplices in $A_1$ and emanating from $x'$.

\textbf{Verification that ${\mathbf X}_1\cap{\mathbf X}_2=\{x\}$:}  Let $V$ be a hyperplane that crosses ${\mathbf X}_1\cap{\mathbf X}_2$.  Then ${\mathbf X}_1-{\mathbf X_2}$ and ${\mathbf X_2}-{\mathbf X}_1$ lie in distinct halfspaces associated to $V$.  By essentiality, there exist minimal UBSs $\mathcal V_1,\mathcal V_2$ such that $\mathcal V_i$ consists of hyperplanes that cross ${\mathbf X}_i$ only, and $V$ separates $\mathcal V_1$ from $\mathcal V_2$.  Let $v_1,v_2$ be the 0-simplices of $A_1,A_2$, respectively, represented by $\mathcal V_1$ and $\mathcal V_2$.  Since $\mathcal V_1$ and $\mathcal V_2$ are separated by $V$, we have $d_{\simp{\mathbf X}}(v_1,v_2)>1.$  Indeed, if $v_1,v_2$ are adjacent, then they are represented by UBS $\mathcal V_1'\sqcup\mathcal V_2'$ such that each $V'\in\mathcal V_1'$ crosses all but finitely many $V''\in\mathcal V_2'$, and hence crosses $V$, a contradiction.  On the other hand, since $v_1\in A_1$ and $v_2\in A_2$ and $\simp{\mathbf X}\cong A_1\star A_2$, the simplices $v_1$ and $v_2$ are adjacent, i.e. $d_{\simp{\mathbf X}}(v_1,v_2)=1,$ a contradiction.  Hence no hyperplane crosses ${\mathbf X}_1\cap{\mathbf X}_2$, which therefore consists entirely of the 0-cube $x$.

\textbf{Defining orthogonal projections:}  Define \emph{orthogonal projections} $\rho_i:{\mathbf X}^{(0)}\rightarrow{\mathbf X}^{(0)}_i$ by letting $\rho_i(y)$ be the gate of $y$ in ${\mathbf X}_i$.  In particular, $\rho_i$ is the identity on the 0-skeleton of ${\mathbf X}_i$.

Note that if $p,q\in\mathbf X$ are adjacent 0-cubes, then one of the following situations occurs.  If $p,q\in{\mathbf X}_i$, then $\rho_i(p)=p,\,\rho_i(q)=q$ and in particular the projections are adjacent.  If $p\in{\mathbf X}_i$ and $q\not\in{\mathbf X}_i$, then $\rho_i(p)=\rho_i(q)=q$ since $q$ is at distance at least 1 from ${\mathbf X}_i$ and at distance exactly 1 from $p$.  If $p,q\not\in{\mathbf X}_i$, then either $p$ is closer to ${\mathbf X}_i$ than is $q$, in which case $\rho_i(p)=\rho_i(q)$, or they are at equal distance from ${\mathbf X}_i$, in which case $p,q$ are separated by a single hyperplane $V$ that does not separate either from ${\mathbf X}_i$, whence $V$ crosses ${\mathbf X}_i$ and is thus the unique hyperplane separating $\rho_i(p)$ from $\rho_i(q)$.  Hence we obtain a map $\rho_i:{\mathbf X}^{1}\rightarrow{\mathbf X}^{1}_i$ that sends 0-cubes to 0-cubes and sends each 1-cube $c$ isometrically to a 1-cube $\rho_i(c)$ or collapses $c$ to a 0-cube.

\textbf{Verification that ${\mathbf X}_2(x)=\rho_1^{-1}(x)$:}  Let $p\in{\mathbf X}_2$ and let $m$ be the unique median of $x,p$ and $\rho_1(p)$.  Then $m$ lies on a geodesic joining $x$ to $\rho_1(p)$, so that, by convexity of ${\mathbf X}_1$, we have $m\in{\mathbf X}_1$.  On the other hand, $m$ lies on a geodesic joining $x$ to $p$ so that, by convexity of ${\mathbf X}_2$, we have $m\in{\mathbf X}_2$.  Therefore, $m\in{\mathbf X}_1\cap{\mathbf X}_2$, whence $m=x$.  Thus $x\in{\mathbf X}_1$ lies on a geodesic joining $p$ to $\rho_1(p)$.  Since $\rho_1(p)$ is the closest 0-cube of ${\mathbf X}_1$ to $p,$ we have $\rho_1(p)=x$, so that ${\mathbf X}_2\subseteq\rho_1^{-1}(x)$.  On the other hand, since $\rho_1$ is the identity on ${\mathbf X}_1$, if $p\in\rho_1^{-1}(x),$ then $p=x$ or $p\in{\mathbf X}-{\mathbf X}_1$.  In the latter case, let $V$ be a hyperplane separating $p$ from ${\mathbf X}_1$.  Then by essentiality of ${\mathbf X}$, there exists a ray $\gamma,$ containing a 1-cube dual to $V$, such that $\gamma(0)=x$, $\gamma(t)=p$ for some $t>0$, and $\gamma$ represents a simplex $u\in A_2$, from which it follows that $x\in{\mathbf X}_2.$

\textbf{Conclusion:}  Define a map $j:{\mathbf X}^{(0)}\rightarrow{\mathbf X}_1^{(0)}\times{\mathbf X}_2^{(0)}$ by $j(p)=(\rho_1(p),\rho_2(p)),$ where $p\in\mathbf X$ is a 0-cube.  For all 0-cubes $p,q\in\mathbf X$, we have
\[d_{{\mathbf X}_1\times{\mathbf X}_2}=d_{{\mathbf X}_1}(\rho_1(p),\rho_1(q))+d_{{\mathbf X}_2}(\rho_2(p),\rho_2(q));\]
the first term on the right counts hyperplanes in $\mathbf X$ that cross ${\mathbf X}_1$ and separate $p,q$ and the second term counts hyperplanes that cross ${\mathbf X}_2$ and separate $p,q$.  Since each hyperplane in $\mathbf X$ crosses exactly one of the ${\mathbf X}_i$, we see that $j$ is an isometric embedding on 0-skeleta and thus extends to an isometric embedding $j:{\mathbf X}^{1}\rightarrow{\mathbf X}^{1}_1\times{\mathbf X}^{1}_2.$  Since a CAT(0) cube complex is determined by its 1-skeleton, it suffices to show that $j$ is surjective on 0-cubes.

Let $(p,q)\in{\mathbf X}_1\times{\mathbf X}_2$ be a 0-cube.  Let $\mathcal V_1(p)$ be the set of hyperplanes in ${\mathbf X}$ separating $p$ from $x$ and let $\mathcal V_2(q)$ be the set of hyperplanes separating $q$ from $x$.  Suppose that some $V\in\mathcal V_2(q)$ fails to cross some $H\in\mathcal V_1(p)$.  By essentiality, there are UBSs $\mathcal V\subset\mathcal V_2$ and $\mathcal H\subset\mathcal V_1$ that are separated by $V$ and $H$ and thus represent simplices $v\in A_2,h\in A_1$ that are non-adjacent in $\simp{\mathbf X}$, a contradiction.  Thus each hyperplane in $\mathcal V_1(p)$ crosses each hyperplane in $\mathcal V_2(q)$, so there is a 0-cube $p'$ such that the set of hyperplanes separating $p'$ from $x$ is precisely $\mathcal V_1(p)\cup\mathcal V_2(q)$.  But then $\rho_1(p')=p$ and $\rho_2(p')=q$, whence $j$ is surjective.
\end{proof}

\begin{rem}\label{rem:PandJ}
To deduce Theorem~\ref{thm:productsandjoins} from Theorem~II.9.24 of~\cite{BridsonHaefliger}, one would need to establish that each geodesic segment in $(\mathbf X,\mathfrak d_{\mathbf X})$ extends to a geodesic line when $\mathbf X$ is essential and fully visible and $\simp\mathbf X$ has a nontrivial join decomposition.  The example shown in Figure~\ref{fig:noextend} shows that the last fact would play an essential role in such a proof.  More seriously, one would need the CAT(1) realization of $\simp\mathbf X$ to be isometric to the Tits boundary (see Section~\ref{sec:titscompare}), but there are essential, fully visible product cube complexes for which this is not the case.
\end{rem}

%%%%%%%%%%%%%%%%%%%%%%%%%%%%%%%%%%%%%%%%%% ADDED 24-26 APRIL 2012 %%%%%%%%%%%%%%%%%%%%%%%%%%%%%%%%%%%%%%%%%%%%%%%%%%%%%%%%%%%%%%%%%%%%%%%%%%%%%%

\subsection{Comparison to other boundaries}\label{sec:titscompare}
Throughout this section, $\partial_{\infty}\mathbf X$ denotes the visual boundary of $(\mathbf X,\mathfrak d_{\mathbf X})$, endowed with the cone topology.  The Tits boundary $\partial_T\mathbf X$ is obtained by endowing the visual boundary with the Tits metric.  In this section, we discuss the relationship between $\simp\mathbf X$, the visual boundary, the Tits boundary, and the \emph{Roller boundary} which $\mathbf X$ possesses as a CAT(0) cube complex.

The construction of $\simp\mathbf X$ somewhat resembles that of the Roller boundary $\mathfrak R\mathbf X$ of $\mathbf X$ (of which a compact description appears in Section~2.2 of~\cite{NevoSageev}).  However, $\mathfrak R\mathbf X$ is necessarily compact, so that by Theorem~\ref{thm:boundaryproperties}, even a tree suffices to distinguish the simplicial boundary from the Roller boundary.  Likewise, $\mathbf X$ differs in general from $\partial_{\infty}\mathbf X$.

Topologically, $\partial_{\infty}\mathbf X$ is sensitive to perturbations of the CAT(0) metric, as shown in~\cite{CrokeKleiner}: Croke and Kleiner demonstrated that the visual boundary of the CAT(0) cube complex associated to the right-angled Artin group $\langle a_1,a_2,a_3,a_4\mid[a_i,a_{i+1}],1\leq i\leq 3\rangle$ varies in response to changes in the angles at the corners of the 2-cubes. By contrast, $\simp\mathbf X$ depends only on $\mathbf X^{(1)}$ -- it does not ``notice'' the CAT(0) metric -- and is thus invariant under such perturbations.

\subsubsection{Simplices and the visual boundary}
Let $\sigma:[o,\infty)\rightarrow\mathbf X$ be a CAT(0) geodesic ray, and let $\mathcal W(\sigma)$ be the set of hyperplanes $H$ with $H\cap\sigma=\emptyset$.  Then $\mathcal W(\sigma)$ is a UBS.  If the geodesic ray $\sigma'$ lies at uniform Hausdorff distance from $\sigma$, then clearly $|\mathcal W(\sigma)\cap\mathcal W(\sigma')|<\infty$, so that the simplices of $\simp\mathbf X$ represented by $\mathcal W(\sigma)$ and $\mathcal W(\sigma')$ coincide.  Denote this simplex by $w(\sigma)$.  Let $[\sigma]$ be the point of $\partial_{\infty}\mathbf X$ represented by $\sigma$ and let $\mathfrak S\mathbf X$ be the set of simplices of $\simp\mathbf X$.  The earlier observation shows that the assignment $[\sigma]\mapsto w(\sigma)$ gives a well-defined map $\mathfrak s:\partial_{\infty}\mathbf X\rightarrow\mathfrak S\mathbf X$.  The following simple lemma is crucial here and in the next section.

\begin{lem}[Representing 0-simplices with CAT(0) geodesic rays]\label{lem:CAT0represent}
Let $\gamma$ be a combinatorial geodesic ray such that $\mathcal W(\gamma)$ is minimal.  Then there exists a combinatorial geodesic ray $\gamma'$ such that $|\mathcal W(\gamma)\triangle\mathcal W(\gamma')|<\infty$, and a CAT(0) geodesic ray $\sigma$ such that $\gamma'$ is contained in the union of cubes $c$ such that $c\cap\sigma\neq\emptyset$.
\end{lem}

%\begin{rem}
%The stronger assertion, that each 0-simplex is represented by a combinatorial geodesic ray that is also CAT(0)-geodesic, is false, as illustrated in Figure~\ref{fig:norep}.
%\begin{figure}[h]
%  \includegraphics[width=1.5in]{norep.eps}\\
%  \caption{The simplicial boundary of this fully visible cube complex is a single 0-simplex; some of the representative combinatorial geodesic rays are shown, arrowed.  No combinatorial geodesic ray is also a CAT(0) geodesic ray.}\label{fig:norep}
%\end{figure}

%\end{rem}

\begin{proof}[Proof of Lemma~\ref{lem:CAT0represent}]
Let $\mathbf C$ be the convex hull of $\gamma$, so that $\mathcal W(\mathbf C)=\mathcal W(\gamma)$.  Since $\mathbf C$ is CAT(0)-convex, it suffices to prove the claim for $\mathbf C$.  Now, $\mathcal W(\gamma)$ is the inseparable closure of a set $\{H_i\}_{i=1}^{\infty}$ of pairwise non-crossing hyperplanes such that $H_{i-1}$ and $H_{i+1}$ are separated by $H_i$ for all $i\geq 2$.  Letting $x_o=\gamma(0)$, we see that for each $i\geq 1$, there exists a 0-cube $x_i$ in $\mathbf C$ that is separated from $x_o$ by $H_j$ for $j\leq i$, but not by $H_{i+1}$.

Now, $x_o$ has finite degree.  Indeed, each 1-cube of $\mathbf C$ emanating from $x_o$ is dual to a unique hyperplane, and these hyperplanes pairwise cross, since $\mathcal W(\mathbf C)=\mathcal W(\gamma)$.  Since $\mathbf X$ contains no infinite family of pairwise-crossing hyperplanes, there are finitely many such 1-cubes.  By induction, for each $i\geq 1$, there are finitely many 0-cubes $x_i$ of the type described above.  Hence there are finitely many CAT(0) geodesic segments $\sigma_i$ that emanate from $x_o$ and cross $H_i$ but not $H_{i+1}$.  By K\"{o}nig's lemma, there is a CAT(0) geodesic $\sigma$ emanating from $x_o$ that crosses each $H_i$.  Since $\mathcal W(\gamma)$ is the inseparable closure of $\{H_i\}_i$, we have $\mathcal W(\sigma)=\mathcal W(\gamma)$.  Let $\mathbf C'$ be the union of all closed maximal cubes whose interiors have nonempty intersection with $\sigma$.  Since each hyperplane crossing $\mathbf C'$ crosses $\sigma$, the intersection of $\mathbf C'$ with each hyperplane-carrier consists of a single maximal cube and is thus connected.  Hence the inclusion $\mathbf C'\hookrightarrow\mathbf C$ is a (combinatorial) isometric embedding.  We thus choose $\gamma'$ to be any combinatorial geodesic ray in $\mathbf C'$ emanating from $x_o$.
\end{proof}

\begin{cor}\label{cor:surjectivepi}
If $\mathbf X$ is fully visible, then $\mathfrak s$ is surjective.
\end{cor}

\begin{proof}
Let $v$ be an $n$-simplex of $\simp\mathbf X$.  If $n=0$, then by Lemma~\ref{lem:CAT0represent}, there exists a CAT(0) geodesic ray $\sigma$ such that $\mathcal W(\sigma)$ represents $v$, so that $\mathfrak s([\sigma])=v$.  Suppose $n\geq 1$, so that $v$ is spanned by a set $v_0,\ldots,v_n$ of 0-simplices, each of which is visible.  By Theorem~\ref{thm:visiblefracflat}, there exist combinatorial geodesic rays $\gamma_0,\ldots,\gamma_n$ with common initial 0-cube such that for all $i\neq j$ and all $H\in\mathcal W(\gamma_i),H'\in\mathcal(\gamma_j)$, the hyperplanes $H$ and $H'$ cross.  Arguing as in Lemma~\ref{lem:CAT0represent} yields a CAT(0) geodesic ray $\sigma$ with $\mathcal W(\sigma)=\bigcup_i\mathcal W(\gamma_i)$, whence $\mathfrak s([\sigma])=v$.
\end{proof}

Topologize $\mathfrak S\mathbf X$ by declaring the set $\mathfrak V\subset\mathfrak S\mathbf X$ to be open if $\mathfrak s^{-1}(\mathfrak V)\subset\partial_{\infty}\mathbf X$ is open in the cone topology.  Hence $\mathfrak S\mathbf X$, with the given quotient topology, is compact.
\begin{cor}\label{cor:0dim}
If $\simp\mathbf X$ is 0-dimensional, then $\simp\mathbf X$ has a compact topology as a quotient of $\partial_{\infty}\mathbf X$.  If $\mathbf X^{(1)}$ is hyperbolic, then the quotient $\partial_{\infty}\mathbf X\rightarrow\simp\mathbf X$ is a homeomorphism.
\end{cor}

\begin{proof}
Since each simplex of $\simp\mathbf X$ is maximal, $\mathbf X$ is fully visible by Theorem~\ref{thm:visiblesimplex}, so that the above quotient topology on $\mathfrak S\mathbf X$ exists.  The identity $\mathfrak S\mathbf X\rightarrow\simp\mathbf X$ endows $\simp\mathbf X$ with the same topology.  If $\mathbf X$ is hyperbolic and the CAT(0) geodesic rays $\gamma,\gamma'$ satisfy $w(\gamma)=w(\gamma')\in\simp\mathbf X^{(0)}$, it is not hard to see that $[\gamma]=[\gamma']$, so that $\mathfrak s$ is injective.  (See Example~\ref{exmp:titsnotisometric} for a situation in which the quotient is not injective.)
\end{proof}

\subsubsection{The CAT(1) realization of $\simp\mathbf X$ and $\partial_T\mathbf X$}  For $n\geq 0$, let $\varsigma_n$ be a right-angled spherical $n$-simplex whose 1-simplices have length $\pi/2$.  Each $n$-simplex of $\simp\mathbf X$ is metrized so as to be isometric to $\varsigma_n$.  A \emph{piecewise geodesic path} $P$ in $\simp\mathbf X$ is a path for which $P\cap v$ is a geodesic of $v$, for each simplex $v$.  The length of $P$ is $|P|=\sum_{v\in\mathfrak S\mathbf X}|P\cap v|_v$, where $|Q|_v$ denotes the length of the path $Q$ in $v$ (and $|\emptyset|_v=0$).  Given $x,y\in\simp\mathbf X$, let $\ddot d(x,y)=\inf_P|P|$, where $P$ varies over paths joining $a$ to $b$.  The simplicial boundary endowed with the metric $\ddot d:\simp\mathbf X\rightarrow\reals_{\geq0}\cup\{\infty\}$ is the \emph{CAT(1) realization} of the simplicial boundary.

According to Theorem~\ref{thm:boundaryproperties}, $\simp\mathbf X$ is a finite-dimensional flag complex, so the CAT(1) realization of $\simp\mathbf X$ is indeed a CAT(1) space.  The same is true of $\partial_T\mathbf X$ and it is natural to ask when $\simp\mathbf X$ is actually an isometric triangulation of the Tits boundary.

\begin{exmp}\label{exmp:titsnotisometric}
The following examples suggest necessary conditions for the CAT(1) realization of the simplicial boundary to be isometric to the Tits boundary.
\begin{compactenum}
\item Let $\mathbf X$ be a diagonal half-flat.  In this situation, $\simp\mathbf X$ consists of a closed interval divided into three 1-simplices, so that $\diam(\simp\mathbf X,\ddot d)=\frac{3\pi}{2}$.  However, $\diam(\partial_T\mathbf X)$ takes some value in $(\frac{\pi}{2},\frac{3\pi}{2})$, depending on the constituent eighth-flats of $\mathbf X$.  This shows that the existence of a proper, essential group action is insufficient to make $\simp\mathbf X\cong\partial_T\mathbf X$.
\item The cube complex $\mathbf X$ illustrated in Figure~\ref{fig:notisometric} has $\simp\mathbf X$ a single 0-simplex since the set of all hyperplanes is a minimal UBS\footnote{The operative property of UBSs in this example is inseparability.}.  The arrowed CAT(0) geodesic rays, however, determine points at positive distance in $\partial_T\mathbf X$.
\end{compactenum}
\end{exmp}
\begin{figure}[h]
  \includegraphics[width=1.5in]{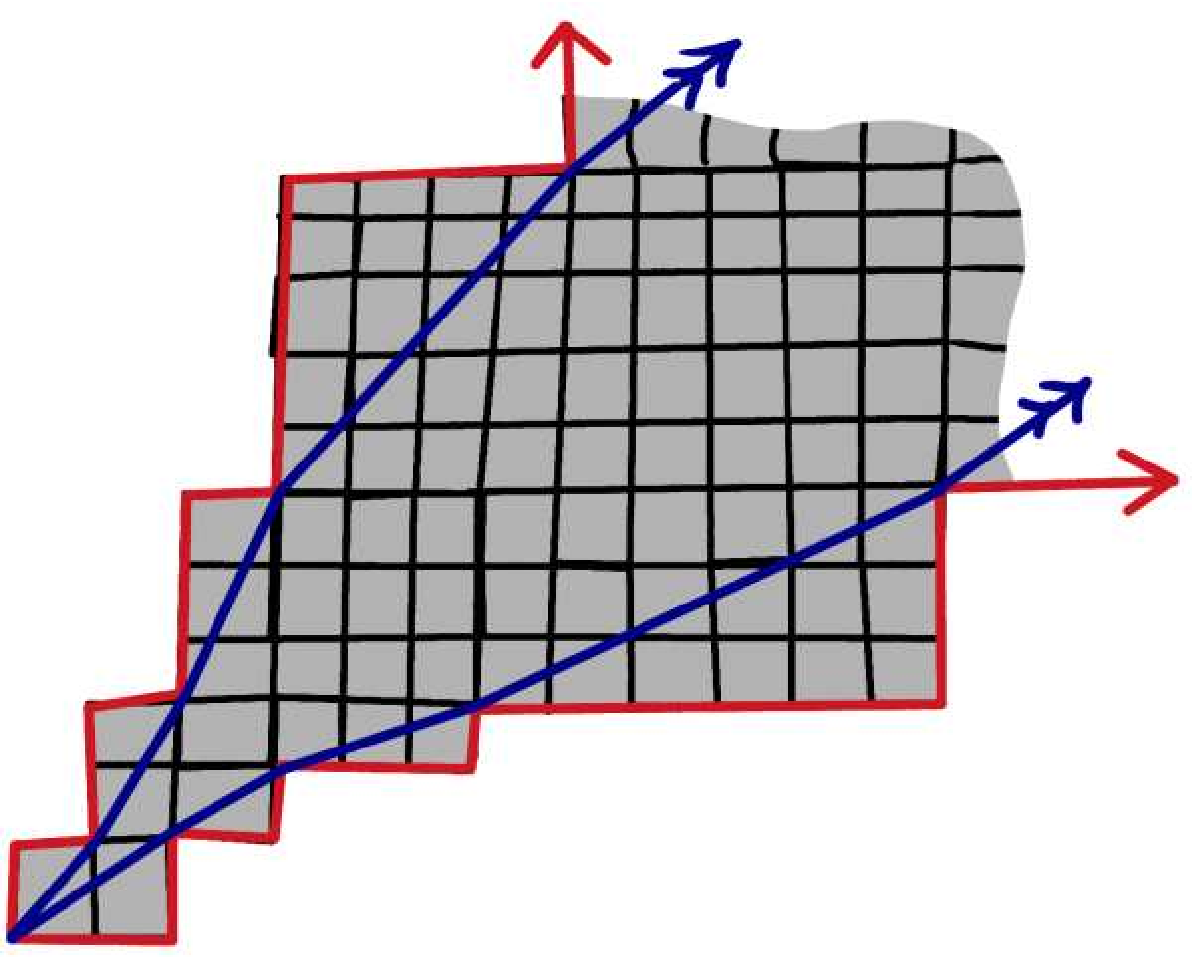}\\
  \caption{The hyperplanes in this cube complex are compact, and the set of all hyperplanes is a minimal UBS.  The single-arrowed combinatorial geodesic rays are hyperplane-equivalent, for example.  The double arrowed CAT(0) geodesic rays represent different points on the Tits boundary, however.}\label{fig:notisometric}
\end{figure}

By Theorem~\ref{thm:structureofboundarysets}, if the combinatorial geodesic ray $\gamma$ represents a positive-dimensional simplex of $\simp\mathbf X$, then $\mathcal W(\gamma)$ cannot have thin bicliques.  Example~\ref{exmp:titsnotisometric}.(2) shows that it is possible for two almost-equivalent minimal combinatorial geodesic rays to diverge, by virtue of the existence of arbitrarily large \emph{finite} bicliques in the subgraph of $\crossing X$ generated by the corresponding UBS.

\begin{defn}[Bounded bending]\label{defn:boundedbending}
The 0-simplex $v$ of $\simp\mathbf X$ has \emph{bounded bending} if for all combinatorial geodesic rays $\gamma,\gamma'$ such that $\mathcal W(\gamma)$ and $\mathcal W(\gamma')$ represent $v$, the rays $\gamma,\gamma'$ lie at finite Hausdorff distance in $\mathbf X^{(1)}$.
\end{defn}

In general, even in the presence of a proper, cocompact group action, there are 0-simplices of $\simp\mathbf X$ that do not have bounded bending; in private communication, Mike Carr brought to my attention examples of such 0-simplices when $\mathbf X$ is the universal cover of the Salvetti complex of the Croke-Kleiner group $\langle a,b,c,d\mid[a,b],[b,c],[c,d]\rangle$.  The failure of bounded bending is the obstruction to the existence of an isometry between the CAT(1) realization of $\mathbf X$ and the Tits boundary.  However, we have:

\begin{prop}\label{prop:simplicialvstits}
Let $\mathbf X$ be a fully visible CAT(0) cube complex for which $\simp\mathbf X$ is defined.  Let $(\simp\mathbf X,\ddot d)$ be the CAT(1) realization of $\simp\mathbf X$.  Then there is an isometric embedding $I:\simp\mathbf X\rightarrow\partial_T\mathbf X$ that is a section of a surjective map $R:\partial_T\mathbf X\rightarrow\simp\mathbf X$ such that:
\begin{enumerate}
 \item For each $z\in\simp\mathbf X$, the preimage $R^{-1}(z)$ has Tits diameter less than $\frac{\pi}{2}$.
 \item $I$ is $\frac{\pi}{2}$-quasi-surjective, and has quasi-inverse $R$.
 \item $I$ is an isometry if and only if every 0-simplex of $\simp\mathbf X$ has bounded bending.
 \item If $\gamma:\reals\rightarrow\mathbf X$ is an axis of an isometry $g\in\Aut(\mathbf X)$, with endpoints $p_{\pm}\in\partial_T\mathbf X$, and $\mathcal H(\gamma)$ represents a pair of 0-simplices, then $p_{\pm}\in\image(I)$.
\end{enumerate}
\end{prop}

\begin{proof}
The plan of the proof is as follows.  We first define an embedding $I$ of the 0-skeleton of $\simp\mathbf X$ in $\partial_T\mathbf X$.  In the absence of bounded bending, this involves some arbitrary choices.  We next show that for any admissible choice of $I$, we can extend $I$ in a unique way to a map $\simp\mathbf X\rightarrow\partial_T\mathbf X$ that is an isometric embedding on each simplex.  We then verify that, if $I,I'$ are maps constructed in this way, $\image(I)$ is contained in the open $\frac{\pi}{2}$-neighborhood of $\image(I')$ and vice versa.  From this, it follows that $I$ is quasi-surjective, and we also argue using this fact to show that $I$ is an isometric embedding.  At this point, it is easy to construct $R$ and verify the remaining claims.

We use the following notation: if $\gamma$ is a CAT(0) geodesic ray, then $[\gamma]$ is the corresponding point on the Tits boundary, $\mathcal W(\gamma)$ is the UBS consisting of hyperplanes crossing $\gamma$, and $v_{\gamma}$ is the simplex of $\simp\mathbf X$ represented by $\mathcal W(\gamma)$.

\textbf{Defining $I$ on the 0-skeleton:}  Let $v\in\simp\mathbf X$ be a 0-simplex.  By full visibility of $\mathbf X$, and Lemma~\ref{lem:CAT0represent}, there exists a CAT(0) geodesic ray $\gamma$ with $v_{\gamma}=v$.  Let $I(v)=[\gamma]$.  Clearly, if $I(v)=I(v')$ for some 0-simplex $v'$, then since rays that are asymptotic are hyperplane-equivalent up to a finite set of hyperplanes, $v=v'$, i.e. $I:(\simp\mathbf X)^{(0)}\rightarrow\partial_T\mathbf X$ is injective.

Now, if $v$ does not have bounded bending, then there is an arbitrary choice involved in defining $I$, since there may be a ray $\gamma'$ such that $[\gamma']\neq[\gamma]$ but $v_{\gamma'}=v$.  However, in such a case, $\partial_T([\gamma],[\gamma'])<\frac{\pi}{2}$.  Indeed, suppose to the contrary that $\partial_T([\gamma],[\gamma'])\geq\frac{\pi}{2}$.  Without loss of generality, $\gamma(0)=\gamma'(0)$.  If $\partial_T([\gamma],[\gamma'])\geq\pi$, then the concatenation of $\gamma$ and $\gamma'$ fellow-travels with a bi-infinite geodesic, so that $v_{\gamma}\neq v_{\gamma'}$, a contradiction.  Otherwise, examining triangles determined by $\gamma(t)$ and $\gamma'(t)$ shows that $|\mathcal W(\gamma)\triangle\mathcal W(\gamma')|$ is infinite if $\gamma$ and $\gamma'$ meet at an angle at least $\frac{\pi}{2}$.  Since this argument does not use minimality of $\mathcal W(\gamma)$, the same conclusion holds even if $v_{\gamma}$ and $v_{\gamma'}$ are not necessarily 0-simplices: either $\partial_T([\gamma],[\gamma'])<\frac{\pi}{2}$ or $v_{\gamma}\neq v_{\gamma'}$.  Conversely, it is easily checked that if the 0-simplex $v$ has bounded bending, then any two rays representing $v$ are asymptotic.

\textbf{Extending $I$ to $\simp\mathbf X$:}  Let $w$ be a $d$-simplex of $\simp\mathbf X$, spanned by the 0-simplices $v_0,\ldots,v_d$.  For $0\leq i\leq d$, let $\sigma_i$ be the CAT(0) geodesic ray chosen above, representing $I(v_i)$.  By induction on $d$, there exists $x\in\mathbf X$ and CAT(0) geodesic rays $\sigma'_0,\ldots,\sigma'_{d-1}$ such that each $\sigma'_i$ fellow-travels with $\sigma_i$, each $\sigma'_i$ emanates from $x$, and $\mathbf X$ contains an isometrically embedded subspace $E\cong\prod_{i=0}^{d-1}\sigma'_i$.  In the base case, when $d=0$, this is clear, since we can take $E=\sigma_0$.

Now, by removing a finite initial subpath from each $\sigma'_i$, and from $\sigma_d$, if necessary, we can assume that if $H$ is a hyperplane crossing $E$, and $V$ is a hyperplane crossing $\sigma_d$, then $V$ and $H$ cross (see the proof of Theorem~\ref{thm:visiblefracflat}).  Therefore, for each hyperplane $H$ crossing $E$, any hyperplane $W$ that separates a point of $\sigma_d$ from $H$ must separate all of $\sigma_d$ from $H$.  Hence $\sigma_d$ fellow-travels with a geodesic ray $\sigma'_d$ that emanates from $x$ and lies in a hyperplane crossing $E$.  Arguing as in the proof of Theorem~\ref{thm:visiblefracflat}, we find that the Euclidean orthant $\sigma_d\times E$ embeds isometrically in $\mathbf X$.

Hence $\mathbf X$ contains a $(d+1)$-dimensional flat orthant $F\cong\prod_{i=0}^d\sigma'_i$, where each $\sigma'_i$ is a CAT(0) geodesic ray representing $I(v_i)$ in the Tits boundary.  Each point $z\in w$ is determined by barycentric coordinates $(\alpha_j(z))_{j=0}^d$, where $\alpha_j(z)\in[0,1]$ and $\sum_{i=0}^d\alpha_j^2(z)=1$.  Let $\gamma_z$ be the geodesic ray in $F$ emanating from the origin $x$ in the direcion of the vector $(\alpha_i(z))_{i=0}^d$, and let $I(z)$ be the point of $\partial_T\mathbf X$ represented by $\gamma_z$ (which is a geodesic of $\mathbf X$ since $F$ is isometrically embedded).  This gives an isometric embedding $I:w\rightarrow\partial_T\mathbf X$ for each right-angled spherical simplex $w$ of $\simp\mathbf X$.  Note that $I(w)$ is completely determined by $\{I(v_i)\}_{i=0}^d$.

(What we have actually proved is that the cubical convex hull of $\sigma'_0,\ldots,\sigma'_d$ contains an orthant isometric to the product of those rays.)

\textbf{$I$ is quasi-surjective:}  Let $\Lambda$ be the set of all maps $\lambda:(\simp\mathbf X)^{(0)}\rightarrow\partial_T\mathbf X$ defined as in the first part of the proof.  Note that for each $\lambda\in\Lambda$, there exists a unique isometric embedding $I_{\lambda}:\simp\mathbf X\rightarrow\partial_T\mathbf X$ that extends the map $\lambda$ and is constructed in the above manner.

Now, let $[\gamma]\in\partial_T\mathbf X$.  If $v_{\gamma}$ is a 0-simplex, we can choose $\lambda\in\Lambda$ such that $\lambda(v_{\gamma})=[\gamma]$, and then extend to $I_{\lambda}$ as above.  Hence $[\gamma]\in\image(I_{\lambda})$ for some $\lambda$.

More generally, let $[\gamma]\in\partial_T\mathbf X$ and let $v_{\gamma}$ be a $d$-simplex spanned by 0-simplices $v_0,\ldots,v_d$.  For some choice of $\lambda$, the ray $\gamma$ is asymptotic to a ray in the orthant $\prod_i\lambda(v_i)$, and thus we can choose $\lambda$ so that $\image(I_{\lambda})$.

We have shown that for each $[\gamma]\in\partial_T\mathbf X$, there exists $\lambda\in\Lambda$ such that $[\gamma]\in\image(I_{\lambda})$.  On the other hand, let $z\in\simp\mathbf X$ and let $\lambda,\lambda'\in\Lambda$.  Then $I_{\lambda}(z)$ and $I_{\lambda'}(z)$ are represented by rays that are hyperplane-equivalent up to a finite set of hyperplanes; we have seen that this implies that $d_T(I_{\lambda}(x),I_{\lambda'}(x))<\frac{\pi}{2}$.  Combining these facts, fixing $\lambda\in\Lambda$ and letting $I=I_{\lambda}$, we see that for any $p\in\partial_T\mathbf X$, there exists $z\in\simp\mathbf X$ such that $d_T(p,I(z))<\frac{\pi}{2}$.  Thus $I$ is $\frac{\pi}{2}$-quasi-surjective.

\textbf{$I$ is an isometric embedding:}  Let $\mathfrak M$ be the set of all maximal simplices in $\simp\mathbf X$, and for each $v\in\mathfrak M$, let $\hat v$ be the open $\frac{\pi}{2}$-neighbourhood of $I(v)$ in $\partial_T\mathbf X$, so that $\cup_{v\in\mathfrak M}\hat v=\partial_T\mathbf X$.  Obviously, if $u\cap v\neq\emptyset$, then $\hat u\cap\hat v\neq\emptyset$.  Conversely, using full visibility, let $u=v_{\gamma_u}$ and $v=v_{\gamma_v}$ for rays $\gamma,\gamma'$.  Let $\hat u\cap\hat v$ contain a point $[\gamma]$.  Then $\mathcal W(\gamma_u)\cap\mathcal W(\gamma)$ is infinite, for otherwise $[\gamma_u],[\gamma]$ are either the endpoints of a bi-infinite \emph{combinatorial} geodesic, which is impossible if $d_T([\gamma_u],[\gamma])<\frac{\pi}{2}$.  Similarly,  $\mathcal W(\gamma_v)\cap\mathcal W(\gamma)$ is infinite.  Hence either $u\cap v\neq\emptyset$, or $u\cap v_{\gamma}$ and $v\cap v_{\gamma}$ are nonempty (or both).  If $u\cap v=\emptyset$, then $\mathcal W(\gamma_v)\cap\mathcal W(\gamma)$ and $\mathcal W(\gamma_u)\cap\mathcal W(\gamma)$ are disjoint (up to finite subsets) UBSes in the UBS $\mathcal W(\gamma)$, whence $u$ and $v$ are contained in a common simplex.  In short: if $\hat u\cap\hat v\neq\emptyset$, then either $u\cap v\neq\emptyset$, or $u$ and $v$ are contained in a common simplex.

Now, let $z,z'\in\simp\mathbf X$.  The above argument shows that $d_T(I(z),I(z'))<\infty$ if and only if $\ddot d(z,z')<\infty$, as follows.  If $\ddot d(z,z')$ is infinite, i.e. if $z,z'$ lie in distinct path-components of $\simp\mathbf X$, then $I(z),I(z')$ cannot be joined by a sequence of subsets of the form $\hat v$, by the preceding argument, and thus $\partial_T(I(z),I(z'))=\infty$; the converse is almost identical.  Hence it suffices to show that $I$ is an isometric embedding on each path-component of $\simp\mathbf X$, i.e. we can assume that $\ddot d(z,z')<\infty$.

For any $\epsilon>0$, we can choose a path $Q$ in $\partial_T\mathbf X$ such that $d_T(I(z),I(z'))<|Q|+\epsilon$.  The path $Q$ is a concatenation of paths joining points at finite Tits distance, and $|Q|$ is the sum of the lengths of these subpaths.

Since $\{\hat v:v\in\mathfrak M\}$ covers $\partial_T\mathbf X$, we can choose a shortest sequence $\hat v_0,\ldots,\hat v_n$ such that $I(z)\in I(v_0),I(z')\in I(v_n)$, and $\hat v_i\cap\hat v_{i+1}\neq\emptyset$ for $0\leq i\leq n-1$, and $Q\subset\cup_i\hat V_i$.  If $n=0$, $z,z'$ lie in the same simplex, whence $Q$ can be chosen in $\image(I)$ since $I$ is an isometric embedding on simplices.

For $n\geq 0$, by induction we can write $Q=Q'Q''$, where $Q'\subset\image(I)$ and $Q''$ joins the terminal point $z''$ of $Q'$ to $z'$, with $z''\in I(v_{n-1})$.  Now, if $v_{n-1}\cap v_n=\emptyset$, then there exists a simplex $w$ containing $v_{n-1}$ and $v_n$, contradicting minimality of $n$.  Thus $I(v_n)\cap I(v_{n-1})\neq\emptyset$, by the dichotomy established above.  We shall show momentarily that, if $u,v$ are simplices with nonempty intersection, then $I|_{u\cup v}$ is an isometric embedding.  It follows that we can choose $Q''$ in the image of $I$, and thus $I(z)$ and $I(z')$ are joined by a path in $\image(I)$ of length less than $d_T(I(z),I(z'))+\epsilon$.  This completes the proof that $I$ is an isometric embedding.

It remains to establish the supporting claim about $I_{u\cup v}$.  Suppose that $a\in u,b\in v$ are interior points, with $I(a)=[\gamma_u]$ and $I(b)=[\gamma_v]$, where $\gamma_u,\gamma_v$ are geodesic rays emanating from a common point $x$.  Since $u\cap v\neq\emptyset$, infinitely many hyperplanes cross both $\gamma_u$ and $\gamma_v$, and since each of $I(u)$ and $I(v)$ is isometrically embedded in $\partial_T\mathbf X$, we have that
\[\ddot d(a,b)\leq\inf_{c\in u\cap v}[\ddot d(a,c)+\ddot d(c,b)]=\inf_{c\in u\cap v}[d_T(I(a),I(c))+d_T(I(c),I(b))].\]

Now, if $P$ is a path in $\partial_T\mathbf X$ from $I(a)$ to $I(b)$, then $P$ contains a point of the form $I_{\lambda}(c)$ where $c\in u\cap v$ and $\lambda\in\Lambda$.  Let $\sigma$ be a ray emanating from $x$ with $I(c)=[\sigma]$ and let $\xi$ be a ray emanating from $x$ with $I_{\lambda}(c)=[\xi]$.  Let $\mathcal U=\mathcal W(\sigma)=\mathcal W(\xi)$ be the infinite subset of $\mathcal W(\gamma_a)$ and $\mathcal W(\gamma_b)$ representing the simplex $u\cap v$.  For $H,H'\in\mathcal U$, write $H<_{\sigma}H'$ if a traveler leaving $x$ along $\sigma$ must cross $H$ before $H'$, and define $<_{\gamma_a},<_{\gamma_b},<_{\xi}$ analogously.  Each of these is clearly a partial ordering of $\mathcal U$.  Now, by construction, $\gamma_a$ and $\sigma$ lie uniformly close to a common orthant, one of whose proper factors is an orthant containing $\sigma$, so that for all but finitely many $H,H'\in\mathcal U$, we have $H<_{\gamma_a}H'$ if and only if $H<_{\sigma}H'$.  The same holds with $\gamma_a$ replaced by $\gamma_b$.  On the other hand, since $\sigma$ and $\xi$ are hyperplane-equivalent, either $[\sigma]=[\xi]$ or there are infinitely many $H\in\mathcal U$ and infinitely many $H'\in\mathcal U$ such that $H<_{\sigma}H'$ but $H'<_{\xi}H$.  (This is illustrated in Figure~\ref{fig:notisometric}.)  In the latter case, this implies that $\mathfrak d(\gamma_a(t),\xi(t))-\mathfrak d(\gamma_a(t),\sigma(t))$ is increasing, and the same is true with $\gamma_a$ replaced by $\gamma_b$.  Thus $|P|\geq d_T(I(a),I(c))+d_T(I(c),I(b))$, and we are done.

\textbf{The map $R$:}  Given $[\gamma]\in\partial_T\mathbf X$, there exists a closest $z\in v_{\gamma}$ such that $d_T([\gamma],I(z))<\frac{\pi}{2}$ and $I(z)$ is represented by a ray hyperplane-equivalent to $\gamma$.  Let $R([\gamma])=z$.  For all $z\in\simp\mathbf X$, we have $R\circ I(z)=z$.  For all $[\gamma]$, $I\circ R([\gamma])$ lies in $I(v_{\gamma})$, whence $R$ is a quasi-inverse for $I$.

\textbf{The bounded bending case:}  If every 0-simplex has bounded bending, then there is a unique choice of $I(v)$ for each 0-simplex $v$.  Since $I$ is uniquely determined by its restriction to the 0-skeleton, $I$ is uniquely determined, i.e. $|\Lambda|=1$.  Since each point in $\partial_T\mathbf X$ lies in the image of $I_{\lambda}$ for some $\lambda\in\Lambda$, the map $I$ is surjective, i.e. an isometry.

Conversely, if the 0-simplex $v$ does not have bounded bending, then there exist hyperplane-equivalent geodesics $\gamma,\gamma'$ that do not fellow-travel but whose sets of dual hyperplanes represent $v$.  By construction, at most one of the corresponding points on the Tits boundary lies in $\image(I)$, which is therefore not surjective.

\textbf{Axes of isometries:}  Let $g\in\Aut(\mathbf X)$ be a hyperbolic isometry and let $\gamma$ be an axis for $g$, so that $\gamma$ determines a pair $v_-,v_+$ of $\langle g\rangle$-invariant 0-simplices of $\simp\mathbf X$.  Since $\gamma$ is periodic and represents a pair of 0-simplices, the convex hull of $\gamma$ is contained in a uniform neighborhood of $\gamma$, whence every ray that is hyperplane-equivalent to a subray of $\gamma$ must fellow-travel with $\gamma$.  Hence $I(v_{\pm})$ are uniquely determined; they are the endpoints $p_{\pm}$ of $\gamma$.
\end{proof}

A useful corollary of Proposition~\ref{prop:simplicialvstits} is that $I$ always induces a bijection from the set of components of $\simp\mathbf X$ to the set of components of $\partial_T\mathbf X$.  Also, if $\mathbf X$ is hyperbolic, then each 0-simplex has bounded bending because $\contact X$ has thin bicliques, whence $\simp\mathbf X$ and $\partial_T\mathbf X$ coincide -- both are discrete sets of equal cardinality.

It seems that the case in which every 0-simplex has bounded bending is in some sense rare, and we close with the following:

\begin{question}
Let $G$ act properly, cocomapctly, and essentially on the CAT(0) cube complex $\mathbf X$.  Under what additional conditions does every 0-simplex in $\simp\mathbf X$ have bouned bending?
\end{question}

%%%%%%%%%%%%%%%%%%%%%%%%%%%%%%%%%%%%%%%%%%%%%%%%%%%%%%%%%%%%%%%%%%%%%%%%%%%%%%%%%%%%%%%%%%%%%%%%%%%%%%%%%%%%%%%%%%%%%%%%%%%%%

\section{Bounded contact graphs}\label{sec:boundedcontact}
Although trees with a single point are easily classified (!), locally infinite quasi-points are considerably less so, and it is thus natural to ask for conditions on $\mathbf X$ that make the quasi-tree $\contact X$ bounded.  For the moment, we study this question without using any group action on $\mathbf X$.

\subsection{Applying the projection trichotomy}\label{sec:quick}
We begin with a simple corollary of Theorem~\ref{thm:trichotomy2}.

\begin{cor}\label{cor:boundedcontactsimple}
Let $\mathbf X$ be a CAT(0) cube complex with no infinite collection of pairwise-crossing hyperplanes.  If $\diam(\contact X)<\infty$, then for every isolated 0-simplex $v\in\simp\mathbf X$, there exists a hyperplane $H$ such that $v$ lies in the image of $\partial_{\triangle}N(H)\hookrightarrow\simp\mathbf X$.
\end{cor}

\begin{proof}
By Corollary~\ref{cor:nowayoffthetrain}, a combinatorial geodesic ray $\gamma$ representing an isolated 0-simplex $v$ at infinity is rank-one.  Hence, by Theorem~\ref{thm:trichotomy2}, either $\gamma$ projects to an unbounded path in $\contact X$, or $\gamma$ lies in a uniform neighborhood of a hyperplane $H$.  In the latter case, it follows that all but finitely many hyperplanes crossing $\gamma$ cross $H$.  In this case, $v$ lies in the image of $\simp N(H)$.  If $\simp\mathbf X$ contains some isolated 0-simplex $v$, then Theorem~\ref{thm:visiblesimplex} implies that $v$ is visible, and the preceding argument shows that either $\contact X$ has infinite diameter or $v$ lies in the image of the simplicial boundary of a hyperplane.
\end{proof}

\begin{exmp}[Spiral of quarter-flats]\label{exmp:spiral}
For $i\geq 1$, let ${\mathbf F}_i$ be a non-diagonal quarter-flat, bounded by rays $I_i$ and $O_i$.  Let ${\mathbf X}_1={\mathbf F}_1$ and for $i\geq 2$, construct ${\mathbf X}_i$ by attaching ${\mathbf F}_{i}$ to ${\mathbf X}_{i-1}$ by identifying the ray $I_i$ with the sub-ray $O'_{i-1}\subset O_{i-1}$ beginning at the second 0-cube.  ${\mathbf X}_i$ is strongly locally finite and fully visible.  For each $i<\infty$, the simplicial boundary ${\simp{\mathbf X}_i}$ is a subdivided line segment of length $i$, and it is easily seen that $\diam(\crossing{X}_i)=i+1.$  The \emph{spiral of quarter-flats} $\mathbf S$ obtained by gluing two copies of $\bigcup_{i\geq 1}\mathbf X_i$ along $I_0$ has an unbounded contact graph, and $\simp\mathbf S$ consists of $\reals$ together with a pair of isolated 0-simplices corresponding to a ``spiraling'' geodesic ray.  See Figure~\ref{fig:spiral}.
\end{exmp}
\begin{figure}[h]
\begin{center}
  \includegraphics[width=1.75in]{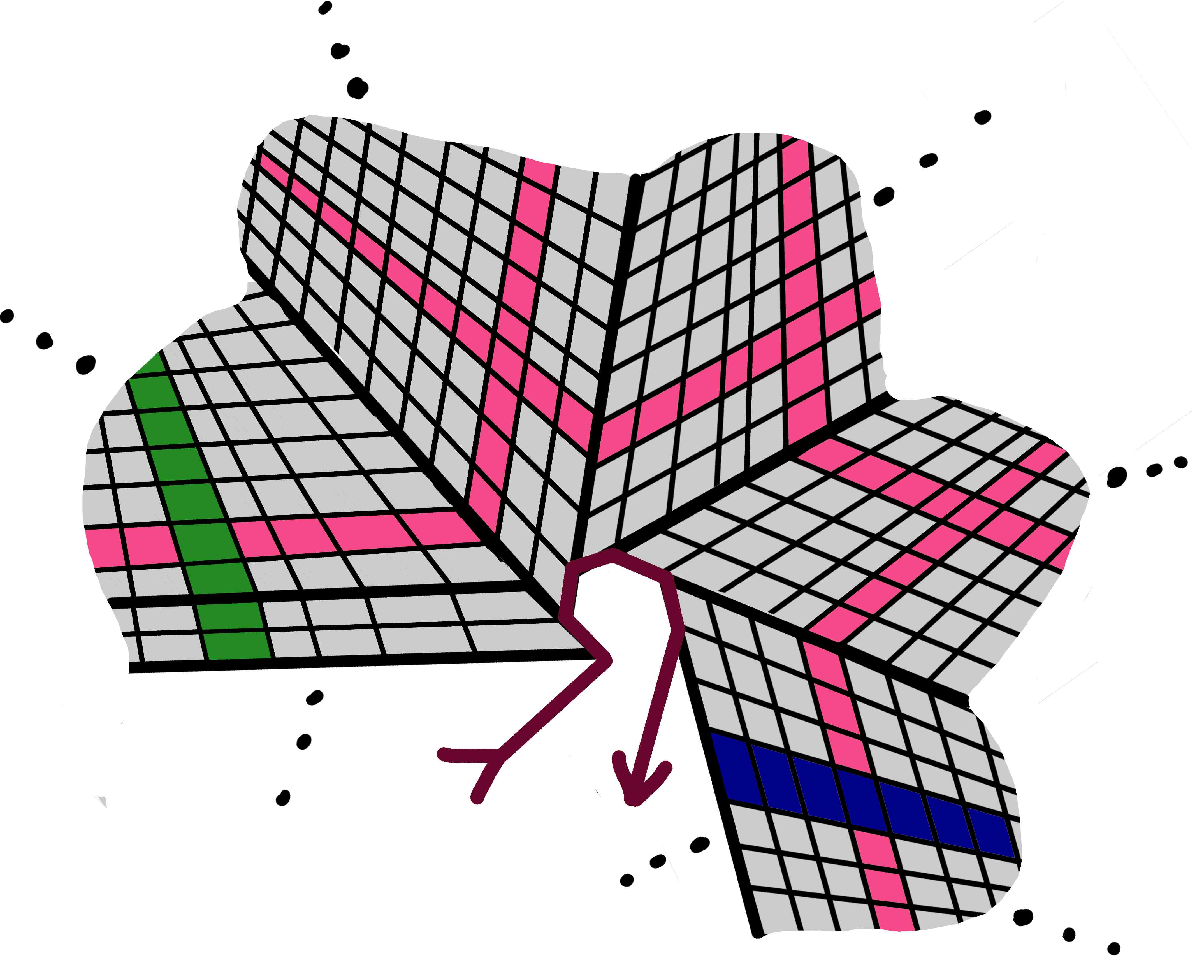}\\
  \caption{An infinite spiral of quarter-flats.  The arrowed geodesic spends finite time in each quarter-flat, and represents a pair of isolated 0-simplices in $\simp\mathbf S$.}\label{fig:spiral}
\end{center}
\end{figure}
\vspace{-1em}
\subsection{Relating $\contact X$ and $\simp\mathbf X$}\label{sec:boundedcontactfull}
The main theorem of this section gives conditions under which $\contact X$ and $\crossing X$ are bounded.

\begin{thm}\label{thm:boundedcontact}
Let $\mathbf X$ be a strongly locally finite, essential, compactly indecomposable CAT(0) cube complex with degree $D\in\naturals\cup\{\infty\}$.  Then
\[\diam(\contact X)\leq\diam(\crossing X)\leq\diam(\simp\mathbf X)\leq 2\diam(\crossing X)-2\leq 2D\diam(\contact X)-2.\]
Moreover:
\begin{compactenum}
\item If $\simp\mathbf X$ is bounded, then $\diam(\contact X)<\infty$ and $\mathbf X$ contains no rank-one geodesic ray.
\item $\mathbf X$ decomposes as an iterated pseudoproduct if and only if $\crossing X$ is bounded.
\item If $\mathbf X$ decomposes as an iterated pseudoproduct, then $\simp\mathbf X$ is bounded.
\item If $D<\infty$, then $(2)$ holds with $\crossing X$ replaced by $\contact X$.
\end{compactenum}
\end{thm}

\begin{proof}
Assertion $(1)$ follows from Lemma~\ref{lem:boundedcrossingrankone} and Lemma~\ref{lem:finalinequality}.  Assertion~$(2)$ is the content of Lemma~\ref{lem:findingpseudoproduct}.  From Lemma~\ref{lem:pseudoproductboundedboundary}, which establishes the bound $\diam(\simp\mathbf X)\leq 2\diam(\crossing X)-2,$ we get Assertion~$(3)$; the lower bounds on $\diam(\simp\mathbf X)$ come from Lemma~\ref{lem:finalinequality} and the fact that $\crossing X$ is a spanning subgraph of $\contact X$.  Finally, Lemma~\ref{lem:contactcrossing} says that, if $D<\infty$, then $\contact X$ is bounded if and only if $\crossing X$ is bounded, which yields assertion~$(4)$.
\end{proof}

\begin{rem}
Example~\ref{exmp:spiral} shows that the bounds on $\diam(\simp\mathbf X)$ in Theorem~\ref{thm:boundedcontact} are attained.  Also, if $\mathbf X$ is the standard tiling of $\reals^2$ by 2-cubes, then $\diam\crossing X=2$ and $\diam(\simp\mathbf X)=2$, since $\simp\mathbf X$ is a 4-cycle.  Looking ahead, one also sees that the combinatorial geodesic ray in $\mathbf S$ that ``spirals'', spending finite time in each quarter-flat, has quadratic divergence.

The assumption that $\mathbf X$ is strongly locally finite is needed to give access to $\simp\mathbf X$ and to guarantee that any infinite set of hyperplanes contains a UBS, which is needed to compare $\simp\mathbf X$ to $\crossing X$.
\end{rem}

The following notion is used to relate $\crossing X$ to $\simp\mathbf X$.
\begin{defn}[Pseudoproduct, iterated pseudoproduct]\label{defn:pseudoproduct}
$\mathbf X$ is the \emph{pseudo-product} of the cube complexes $\mathbf Q_1$ and $\mathbf Q_2$, denoted by $\mathbf X\cong\mathbf Q_1\psprod\mathbf Q_2$, if all of the following hold:
\begin{compactenum}
\item There is a convex embedding $a:\mathbf Q_1\rightarrow\mathbf X$, and the image of $\mathbf Q_1$ is crossed by an infinite set $\mathcal V_1$ of hyperplanes.  There is a restriction quotient $q_1:\mathbf X\rightarrow\mathbf Q_1$ such that $a$ is a section of $q_1$, and $q_1$ is obtained by restricting to the set $\mathcal V_1$ of hyperplanes.
\item There is a restriction quotient $q_2:\mathbf X\rightarrow\mathbf Q_2$ obtained by restricting to the set $\mathcal V_2=\mathcal W-\mathcal V_1$ of hyperplanes.  Moreover, $\mathcal V_2$ is infinite.
\item There is a cubical isometric embedding $e:\mathbf X\rightarrow\mathbf Q_1\times\mathbf Q_2$ for which the convex hull of $e(\mathbf X)$ is all of $\mathbf Q_1\times\mathbf Q_2$.  Also, if $\pi_1:\mathbf Q_1\times\mathbf Q_2\rightarrow\mathbf Q_1$ is the projection, then $\pi_1\circ e\circ a$ is the identity on $\mathbf Q_1$.
\item For each $H\in\mathcal V_2$, there are infinitely many $V\in\mathcal V_1$ such that $H\bot V$ in $\mathbf X.$
\end{compactenum}
$\mathbf X$ is an \emph{iterated pseudoproduct} if there exists a hyperplane $H_0\subset\mathbf X$, an integer $k\geq 0$ called the \emph{depth}, and a sequence of cube complexes $Q_1^i,Q_2^i$, for $0\leq i\leq k$, such that $\mathbf X\cong Q^0_1\psprod Q^0_2$, for all $0\leq i\leq k-2$, and $Q^i_1\cong Q^{i+1}_1\psprod Q^{i+1}_2$, and finally $Q^k_1\cong H_0,\,Q^k_2\cong[-\frac{1}{2},\frac{1}{2}]$ and $Q^{k-1}_1\cong Q_1^k\times Q^k_2$.
\end{defn}

If $\mathbf X\cong\mathbf Q_1\times\mathbf Q_2,$ then $\mathbf X\cong\mathbf Q_1\psprod\mathbf Q_2$.  However, for example, an eighth-flat is a pseudo-product of two infinite rays but is not a product of any two cube complexes since its crossing graph has diameter at least 3.

Fix a base hyperplane $H_0$.  The \emph{grade $g(H)$} of the hyperplane $H$ with respect to $H_0$ is the distance of $H$ to $H_0$ in $\crossing X$.  In other words, $H_0$ has grade 0 and $H$ has grade $n$ if $H$ \emph{crosses} a grade-$(n-1)$ hyperplane but does not cross a grade-$(n-2)$ hyperplane.

\subsubsection{Proof of Theorem~\ref{thm:boundedcontact}}
Theorem~\ref{thm:boundedcontact} is proved as the following sequence of lemmas.  Throughout, $\mathbf X$ is supposed to be strongly locally finite, leafless, and compactly indecomposable.

\begin{lem}\label{lem:boundedcrossingrankone}
If $\simp\mathbf X$ is bounded, then $\mathbf X$ contains no rank-one combinatorial geodesic ray.
\end{lem}

\begin{proof}
If $\simp\mathbf X$ is bounded, then it is connected.  Hence either $\simp\mathbf X$ is a single 0-simplex, or $\simp\mathbf X$ contains no isolated 0-simplex.  In the latter case, $\mathbf X$ contains no rank-one ray by Corollary~\ref{cor:nowayoffthetrain}.  If $\simp\mathbf X$ is a single 0-simplex $v$, then let $\mathcal V$ be a minimal UBS representing $v$.  But there is a hyperplane $V$ that is initial in $\mathcal V$: every element of $\mathcal V-\{V\}$ lies in a single halfspace $V^+$ associated to $V$, or crosses $V$.  But by essentiality, $V^-$ contains a UBS representing a 0-simplex at infinity distinct from $v$, a contradiction.
\end{proof}

\begin{lem}\label{lem:findingpseudoproduct}
$\crossing X$ is bounded if and only if $\mathbf X$ is an iterated pseudoproduct.
\end{lem}

\begin{proof}
Choose $R\leq\diam\crossing X<\infty$ and a hyperplane $H_0$ such that, for all $H\in\mathcal W$, the grade $g(H)$ of $H$ in $\crossing X$, with respect to $H_0,$ is at most $R$.  By essentiality, we can assume $R\geq 2$, for otherwise $H_0$ crosses each $H\neq H_0$, and thus $H_0$ is not essential.  Moreover, choose $R$ so that there exists $H$ with $g(H)=R.$

Let $\mathcal V_1$ be the set of hyperplanes $H$ with $0\leq g(H)\leq R-1$, and let $\mathcal V_2$ be the set of grade-$R$ hyperplanes.  Since $\mathbf X$ is compactly indecomposable, $N(H_0)$ cannot be compact, and thus there are infinitely many hyperplanes crossing $H_0$.  Hence $\mathcal V_1$ is infinite.

On the other hand, let $H_R$ be a grade-$R$ hyperplane.  Then since $R\geq 2$, the hyperplanes $H_R$ and $H_0$ do not cross.  Since $H_R$ is essential, there exists a hyperplane $H'$ such that $H_R$ separates $H'$ from $H_0$.  Any path in $\crossing X$ from $H'$ to $H_0$ contains $H_R$ or some hyperplane crossing $H_R$, and thus $H'$ has grade at least $R$, whence, since every hyperplane has grade at most $R$, we have $H'\in\mathcal V_2$.  Now $H'$ is essential, and thus separates some $H''$ from $H_0$.  Thus $H''$ has grade $R$, by the same argument.  Hence $\mathcal V_2$ is infinite.

\textbf{The quotients $\mathbf X\rightarrow\mathbf Q_i$:}  For $i\in\{1,2\}$, let $\mathbf Q_i$ be the cube complex dual to the wallspace $(\mathbf X^0,\mathcal V_i)$ whose walls are those induced by the hyperplanes in $\mathcal V_i$.  This gives a cubical quotient $q_i:\mathbf X\rightarrow\mathbf Q_i$; this is just a restriction quotient that collapses the hyperplanes in $\mathcal W-\mathcal V_i$.

Let $\mathbb V_i$ be the set of hyperplanes in $\mathbf Q_i$, so that there is a natural bijection $\mathcal V_i\rightarrow\mathbb V_i$ given by $V\mapsto q_i(V)$.  Note that, if $V,V'\in\mathcal V_i$ are hyperplanes, then $V\bot V'$ if and only if $q_i(V)\bot q_i(V')$.  Moreover, if $V,V'$ osculate, then $q_i(V)$ and $q_i(V')$ osculate.  For $V,V'\in\mathcal V_1$, the converse is true.  If no hyperplane separates $q_1(V)$ from $q_1(V')$, then any hyperplane $U$ separating $V,V'$ must belong to $\mathcal V_2$.  But if $g(V),g(V')\leq R-1$ and $U$ separates $V,V'$, then $g(U)\leq R-1$, a contradiction.  In other words, $\mathcal V_1$ is inseparable, and hence $V,V'\in\mathcal V_1$ contact if and only if their images in $\mathbf Q_1$ do.  However, there could exist non-contacting $V,V'\in\mathcal V_2$ such that $q_2(V)$ osculates with $q_2(V')$, if every hyperplane separating $V,V'$ has grade at most $R-1$.

\textbf{The embedding $e:\mathbf X\rightarrow\mathbf Q_1\times\mathbf Q_2$:}  Fix a base 0-cube $x_o\in\mathbf X$.  Recall that this is a section $x_o:\mathcal W\rightarrow\mathcal W^{\pm}$ of the map $\pi:\mathcal W^{\pm}\rightarrow\mathcal W$ that associates each halfspace of $\mathbf X$ to the corresponding hyperplane.  Moreover, $x_o$ is automatically consistent and canonical.  Fix a base 0-cube $(y^1_o,y_o^2)\in\mathbf Q_1\times\mathbf Q_2$.  Then $y^1_o$ is a consistent, canonical orientation of the hyperplanes $q_1(V)\in\mathbb V_1$ and $y_o^2$ a consistent canonical orientation of the hyperplanes in $\mathbb V_2$.  Define $y_o^i(q_i(V))=q_i(x_o(V))$, i.e. the image in $\mathbf Q_i$ of the halfspace associated to $V$ that contains $x_o$.  This is the halfspace of $\mathbf Q_i$ containing $q_i(x_o)$.  In other words, let $y_o^i=q_i(x_o)$.

We now define a map $\mathbf X^0\ni x\mapsto\psi(x)=(\psi^1(x),\psi^2(x))\in\mathbf Q_1\times\mathbf Q_2$ as follows.  Let $\psi^i(x)(q_i(V))=q_i(x(V))$.  Then, since $x$ is consistent, for all $V,V'\in\mathcal V_i$, we have $x(V)\cap x(V')\neq\emptyset$ and hence $q_i(x(V))\cap q_i(x(V'))\supseteq q_i(x(V)\cap x'(V))\neq\emptyset$.  Thus $\psi^i(x)$ is a consistent orientation of the hyperplanes of $\mathbf Q_i$.  Since $x$ differs from $x_o$ on finitely many hyperplanes in $\mathcal V_i$, and $\psi^i(x)(q_i(V))\neq\psi^i(x_o)(q_i(V))$ only if $x(V)\neq x_o(V)$, the 0-cube $\psi^i(x)$ is canonical. Hence $\psi(x)$ is a genuine 0-cube of $\mathbf Q_1\times\mathbf Q_2$.

Now $\psi^i(x)(q_i(V))$ differs from $\psi^i(y)(q_i(V))$ when $q_i(x(V))$ and $q_i(y(V))$ are different halfspaces associated to $q_i(V)$.  This happens if and only if $x(V)\neq y(V)$, so that $d_{\mathbf Q_1\times\mathbf Q_2}(\psi(x),\psi(y))$ counts the hyperplanes in $\mathcal V_1\sqcup\mathcal V_2=\mathcal W$ on which $x,y$ differ.  Thus $\psi:\mathbf X^0\rightarrow(\mathbf Q_1\times\mathbf Q_2)^0$ is an isometric embedding, and hence the map $\psi$ extends to an isometric embedding $e:\mathbf X\rightarrow\mathbf Q_1\times\mathbf Q_2$, since each median graph uniquely determines a CAT(0) cube complex.

\textbf{The embedding $a:\mathbf Q_1\rightarrow\mathbf X$:}  Choose a base 0-cube $x_o\in N(H_0)$.  Let $\mathbf Y^0$ be the set of 0-cubes $y$ in $\mathbf X$ such that every hyperplane separating $y$ from $x_o$ belongs to $\mathcal V_1$.  Let $\mathbf Y\subset\mathbf X$ be the convex hull of $\mathbf Y^0$.  Each 0-cube $y\in\mathbf Y$, viewed as a section of $\pi$, has the property that $y(W)=x_o(W)$ for $W\in\mathcal V_2$, since $W$ cannot separate $y$ from $x_o$.  In other words, $\mathbf Y$ is isomorphic to the restriction quotient obtained by restricting each 0-cube of $\mathbf X$ to $\mathcal V_1$!  More precisely, for each $q(x)\in\mathbf Q_1$, define $a(q(x)):\mathcal W\rightarrow\mathcal W^{\pm}$ by $a(q(x))(W)=x(W)$ for $W\in\mathcal V_1$ and, for all $W\in\mathcal V_2$, let $a(q(x))(W)$ be the halfspace containing $H_0$ (and thus $x_o$).  This is well-defined since no $W\in\mathcal W_2$ crosses $H_0.$  This orientation is consistent, since $x$ is consistent on $\mathcal W_1$.  If $a(q(x))$ orients $W\in\mathcal W_1$ inconsistently with $V\in\mathcal W_2$, then $V$ separates $W$ from $H_0,$ but we have seen this to be impossible.  Finally, $a(q(x))$ agrees with $x_o$ on all but finitely many hyperplanes, namely those hyperplanes in $\mathcal V_1$ on which $x,x_o$ disagree.  Hence $a(q(x))$ is canonical and belongs to $\mathbf Y$, and, if $x\in\mathbf Y$, then $a(q(x))=x$.  In the usual way, one checks that $q(x)\mapsto a(q(x))$ determines an isometric embedding $a:\mathbf Q_1\rightarrow\mathbf X$.  (In general, there is no isometric embedding $\mathbf Q_2\rightarrow\mathbf X$.)

\textbf{The convex hull of $e(\mathbf X)$:}  Let $\mathbf Z\subseteq\mathbf Q_1\times\mathbf Q_2$ be the convex hull of $e(\mathbf X)$.  Every hyperplane of $\mathbf Q_1\times\mathbf Q_2$ is of the form $q_1(V)\times\mathbf Q_2$ or $\mathbf Q_1\times q_2(H)$, where $V\in\mathcal V_1$ and $H\in\mathcal V_2$, and each hyperplane of the former type crosses each hyperplane of the latter type.  It follows that every hyperplane of $\mathbf Q_1\times\mathbf Q_2$ crosses $e(\mathbf X)$, whence $\mathbf Z=\mathbf Q_1\times\mathbf Q_2$.  Indeed, $\mathbf Z$ is either all of $\mathbf Q_1\times\mathbf Q_2$, or $\mathbf Z$ is the intersection of all halfspaces that contain $e(\mathbf X)$.  The latter is impossible since every hyperplane crosses $e(\mathbf X)$.

\textbf{The 2-cubes:}  Let $H\in\mathcal V_2$ be a grade-$R$ hyperplane.  By essentiality, there exists an infinite set $U_0,U_1,\ldots$ of grade-$R$ hyperplanes such that $H$ separates $U_i$ from $H_0$, for all $i\geq 0$, and for all $i\geq 1$, the hyperplanes $U_{i+1}$ and $U_{i-1}$ are separated by $U_i$.  Every $U_i$ crosses a set $\{W_i^j\}_j$ of grade-$R-1$ hyperplanes, each of which must cross $H$.  Let $\Omega$ be the full subgraph of $\crossing X$ generated by $\{W_i^j\}\cup\{H\}$.  Then every path in $\crossing X$ from $U_i$ to $H_0$ passes through $H$ or through one of the $W_i^j$, since $H$ separates $U_i$ from $H_0$.  Thus $\crossing X-\Omega$ has at least two nonempty components, namely the component containing $U_0$ and the component containing $H_0$.  Hence, by Lemma~\ref{lem:separation} and compact indecomposability, $\Omega$ is infinite, whence infinitely many grade-$(R-1)$ hyperplanes cross $H$ in $\mathbf X$.  We conclude that $\mathbf X\cong\mathbf Q_1\psprod\mathbf Q_2$.

\textbf{Induction:}  $\mathbf Q_1$ has infinitely many hyperplanes and $\crossing Q_1$ has finite diameter -- grading from $H_0$, every hyperplane has grade at most $R-1$.  If $R=2$, then $Q_1\cong H_0\times[-\frac{1}{2},\frac{1}{2}]$.  Otherwise, by essentiality, the set $\mathcal V_2^1$ of grade $(R-1)$ hyperplanes is infinite, and, letting $\mathcal V_1^1$ be the (infinite) set of hyperplanes of grade at most $R-2$, we proceed as before, to find that $\mathbf Q_1\cong\mathbf Q_1^1\psprod\mathbf Q_2^1.$
Continuing in this way, since the highest grade decreases at each step, we find that as long as $R>2$, we can continue the pseudo-product decompositions, using the above argument.  This terminates with $Q^{R-1}_1\cong H_0\times[-\frac{1}{2},\frac{1}{2}]\cong N(H_0)$.  Hence $\mathbf X$ is an iterated pseudoproduct.

\textbf{Bounding $\diam(\crossing X)$:}  Conversely, suppose that $\mathbf X$ is an iterated pseudoproduct, and let $k$ be the constant from Definition~\ref{defn:pseudoproduct}.  Then, for each hyperplane $H\neq H_0$ of $\mathbf X$, either there exists a maximal $m\leq k-1$ such that $H$ survives in the restriction quotient $\mathbf X\rightarrow\mathbf Q_1^0\rightarrow\mathbf Q_1^1\rightarrow\ldots\rightarrow\mathbf Q_1^m$, or $H$ crosses a hyperplane crossing $\mathbf Q_1^0$.  In the latter case, $Q^{m}_1\cong Q^{m+1}_1\psprod Q_{m+1}^2$, so that, by the definition of a pseudo-product, $H$ crosses a hyperplane that crosses $Q^{m+1}_1$.  Hence $H$ is at distance at most $k+1$ from $H_0$ in $\crossing X$.
\end{proof}

\begin{lem}\label{lem:pseudoproductboundedboundary}
If $\mathbf X$ is an iterated pseudoproduct, then $\diam(\simp\mathbf X)\leq2(\diam(\crossing X)-1).$
\end{lem}

\begin{proof}
Let $R<\infty$ be the maximal grade in $\crossing X$ with respect to a hyperplane $H_0$, and let $\mathbf X\cong\mathbf Q_1\psprod\mathbf Q_2$ be the pseudoproduct decomposition from Lemma~\ref{lem:findingpseudoproduct}.  Since $a(\mathbf Q_1)\subset\mathbf X$ is convex, we have a simplicial embedding $\simp a:\simp\mathbf Q_1\rightarrow\simp\mathbf X$, by Theorem~\ref{thm:boundarysubcomplexes}.  Let $A_1=\simp a(\simp\mathbf X)$ be its image.  Hence, by Theorem~\ref{thm:productsandjoins}, we have $\simp(\mathbf Q_1\times\mathbf Q_2)\cong A_1\star\simp\mathbf Q_2,$ since $e\circ a:\mathbf Q_1\rightarrow\mathbf Q_2$ embeds $\mathbf Q_1$ convexly in $\mathbf Q_1\times\mathbf Q_2$ (with image $\mathbf Q_1\times y_o^2$).

Let $A_2\cong\simp\mathbf Q_2$, so that $\simp(\mathbf Q_1\times\mathbf Q_2)=A_1\star A_2$.  In particular, $A_1\cap A_2=\emptyset$.  Thus every simplex of $\simp(\mathbf Q_1\times\mathbf Q_2)$ is of the form $[u_0,u_1,\ldots,u_p,v_0,\ldots,v_q],$ where each $u_i$ represented in $\simp\mathbf X$ by a UBS $\mathcal U_i\subset\mathcal V_1$ and each $v_j$ is represented by a UBS $q_2(\mathcal W_j)$, where $\mathcal W_j\subset\mathcal V_2$ is a UBS.

Let $\mathcal Q^2$ be a UBS of hyperplanes in $\mathbf Q_2$ that represents the 0-simplex $q$.  Let $\{W_i\}_{i\geq 0}\subseteq\mathcal Q^2$ be a set of hyperplanes such that for all $i\geq 1$, the hyperplane $W_i$ separates $W_{i-1}$ from $W_{i+1}$.  By the pseudoproduct decomposition, for each $i\geq 0,$ there exist infinitely many hyperplanes $H$ of grade at most $R-1$ that cross $W_i$ in ${\mathbf X}$, and thus cross $W_j$ in ${\mathbf X}$, for all $j<i$.  Hence there is a UBS $\mathcal Q^1$ representing a 0-simplex $q^1\in\simp{\mathbf Q}_1$ such that $[q^1,q]$ is a 1-simplex of $\image (\simp{\mathbf X}\rightarrow\simp(\mathbf Q_1\times\mathbf Q_2)).$  In other words, each 0-simplex of $\simp{\mathbf Q}_2$ is adjacent in the image of $\simp\mathbf X$ to a 0-simplex of $\simp{\mathbf Q}_1.$

We now argue by induction on $R$.  First suppose that $R\geq 3$ and, by induction, that there is a function $K:\naturals\rightarrow\naturals$ such that, for all 0-simplices $s,s'\in\simp{\mathbf Q}_1,$ there is a path in $\simp\mathbf X$ of length at most $K(R-1)$ joining $s$ to $s'$.  Then, if $q,q'\in\simp{\mathbf Q}_2$ are 0-simplices, they are respectively at distance 1 from some 0-simplices $s,s'\in\simp{\mathbf Q}_1,$ by the preceding paragraph, and hence there is a path in $\simp{\mathbf X}$ of length at most $K(R)=K(R-1)+2$ joining $q,q'$.  Hence we have
\[\diam(\simp{\mathbf X})\leq K(R)=K(2)+2(R-2).\]

Next, let $R=2$ and let $s^1,s^2$ be 0-simplices of $\simp{\mathbf Q}_1,$ so that for $i\in\{1,2\}$, the simplex $s^i$ is represented by a UBS $\mathcal S^i$ consisting only of grade-1 hyperplanes.  Let $\{S^i_j\}_{j\geq 0}\subseteq\mathcal S^i$ be a maximal set of pairwise non-crossing hyperplanes.  Either $[s^1,s^2]$ is a 1-simplex, and we are done, or we can choose $\mathcal S^i$ so that $\mathcal S^1\cap\mathcal S^2=\emptyset.$

In the latter case, for all $j,k\geq 0$, there is a geodesic path $S^1_j\bot H_0\bot S^2_k$ in $\crossing X$, where $H_0$ is the base hyperplane of the grading.  Thus the set of hyperplanes $H$ such that $H\bot S_j^1$ and $H\bot S^2_k$ is nonempty.  On the other hand, the set of such $H,$ together with the 2-element set $\{S_j^1,S_k^2\}$, generates a subgraph $\Lambda\subset\crossing X$ such that $\crossing X-\Lambda$ has at least two infinite components, since $\mathbf X$ is essential.  Since $\mathbf X$ is compactly indecomposable, there are infinitely many such $H$, by Lemma~\ref{lem:separation}, and hence there is a minimal UBS $\mathcal U$ such that, for all $U\in\mathcal U$, we have a path $S^1_j\bot U\bot S^2_k$ in $\crossing X$.  But then $s^1$ and $s^2$ are both adjacent to the 0-simplex $u$ represented by $\mathcal U.$  Thus $K(2)=2$ and $\diam(\simp{\mathbf X})\leq 2R-2\leq 2\diam(\crossing X)-2,$ as required.
\end{proof}

\begin{lem}\label{lem:finalinequality}
If $\mathbf X$ is essential and strongly locally finite, then $\diam(\crossing X)\leq\diam(\simp\mathbf X)$.
\end{lem}

\begin{proof}
Choose $R\leq\diam(\crossing X)$.  Also, assume that we have chosen $R\geq 2$, for if we can not make this choice then essentiality is contradicted.  Choose hyperplanes $U_0$ and $V_0$ such that $d_{_{\crossing X}}(U_0,V_0)=R\geq 2$.  By essentiality, there exists a minimal UBS $\mathcal U$ and a minimal UBS $\mathcal V$ such that $U_0\in\mathcal U$ and $V_0\in\mathcal V$ are initial hyperplanes in the given UBSs: for all $U'\in\mathcal U$, $U_0$ separates $U'$ from $V_0$, and for all $V'\in\mathcal V$, $V_0$ separates $V'$ from $U_0$.

Let $u,v$ be the 0-simplices at infinity represented by $\mathcal U$ and $\mathcal V$ respectively.  If $u,v$ lie in different components of $\simp\mathbf X$, then $\diam(\simp\mathbf X)=\infty$ and we are done.  Therefore, let $u,s_1,\ldots,s_k,v$ be a path in $(\simp{\mathbf X})^1$ joining $u$ to $v$.  Let $\mathcal S_i$ be a UBS representing the 0-simplex $s_i$.  Then there is a path $\sigma=U\bot S_1\bot\ldots\bot S_k\bot V$ in $\crossing X$, where $U\in\mathcal U,V\in\mathcal V$ and for each $1\leq t\leq k$, we have $S_t\in\mathcal S_t$.  Without loss of generality, $U_0$ and $V_0$ both separate $U$ from $V$.  Hence $\sigma$ must either contain $U_0$ (respectively, $V_0$) or $\sigma$ must contain a hyperplane $S_t$ that crosses $U_0$ (respectively, a hyperplane $S_r$ that crosses $V_0$).  Hence, in the worst case, we have a path $U_0\bot S_t\bot S_{t+1}\bot\ldots\bot S_r\bot V_0$ of length $r-t+2$ joining $U_0$ to $V_0$.  Thus $k+1\geq r-t+2\geq R$, whence $\diam(\simp{\mathbf X})\geq\diam(\crossing X)$.
\end{proof}

Lemma~\ref{lem:separation} allows us to assume that no finite subgraph separates the crossing graph; this is how compact indecomposability manifests itself in $\crossing X$.

\begin{lem}\label{lem:separation}
Let $\mathbf X$ be strongly locally finite and compactly indecomposable.  Then for any finite subgraph $\Lambda$ of $\crossing X$, the complement $\crossing X-\Lambda$ is connected.
\end{lem}

\begin{proof}
\textbf{Proof that $\crossing X$ is connected:}  If $\crossing X$ is disconnected, $\mathbf X$ is the wedge sum of two proper subcomplexes, and is thus not compactly indecomposable.  Indeed, let $H,H'$ be hyperplanes in different components of the crossing graph.  If $H''$ separates $H$ from $H'$, then $H''$ cannot lie in the same component of $\crossing X$ as both $H$ and $H'$, so we may choose these hyperplanes so that $H\coll H'$.  Now $N(H)\cap N(H')$ consists of a single 0-cube $x$, since, if $N(H)\cap N(H')$ contained distinct 0-cubes $x,y$, each hyperplane separating $x$ from $y$ would cross $H$ and $H'$.

Neither of $H,H'$ is compact, since $\mathbf X$ is compactly indecomposable.  Thus we can choose 0-cubes $y\in N(H)$ and $y'\in N(H')$ so that $x,y,y'$ are pairwise-distinct.  Suppose that $P\rightarrow\mathbf X$ is a path joining $y$ to $y'$ with $x\not\in P$.  Let $A\rightarrow N(H)$ be a shortest path joining $x$ to $y$ and let $B\rightarrow N(H')$ be a shortest path joining $y'$ to $x$, and consider a minimal-area disc diagram $D\rightarrow\mathbf X$ bounded by $APB$, as in Figure~\ref{fig:disconnecteddelta}.  Let $K$ be the dual curve emanating from the initial 1-cube of $A$ and let $K'$ be the dual curve emanating from the terminal 1-cube of $B$.  Let $U$ be the hyperplane to which $K$ maps, and let $U'$ be the hyperplane to which $K'$ maps.  If $U=U'$, then $H\bot U\bot H'$, a contradiction.  If $K$ crosses $K'$, then $H\bot U\bot U'\bot H'$, another contradiction.  By minimality of area, $K$ and $K'$ end on $P$.  Let $P'$ be the subpath of $P$ between and including the 1-cubes $c$ and $c'$ dual to $K$ and $K'$, as in Figure~\ref{fig:disconnecteddelta}, so that $P'=cQc'$, where $Q$ is a path such that $|Q|<|P|$ and $x\not\in Q$.  On the other hand, $x\in N(U)\cap N(U')$, and no hyperplane $W$ crosses $U$ and $U'$, since otherwise the path $H\bot U\bot W\bot U'\bot H'$ would join $H$ and $H'$ in $\crossing X$.  Thus we can replace $H$ and $H'$ by $U$ and $U'$ respectively, and replace $P$ by $Q$, and obtain a lower-length counterexample.  Eventually, we are in a situation where $|Q|=0$ and $x\not\in Q$.  But then $Q\in N(U)\cap N(U')$ is separated from $x\in N(U)\cap N(U')$, and thus $U$ and $U'$ are crossed by a common hyperplane, a contradiction.  Hence no such $P$ can exist, i.e. every path from $N(H)$ to $N(H')$ passes through $x$.  It is now obvious that $x$ disconnects $\mathbf X$.

\begin{figure}[h]
  \includegraphics[width=1.5in]{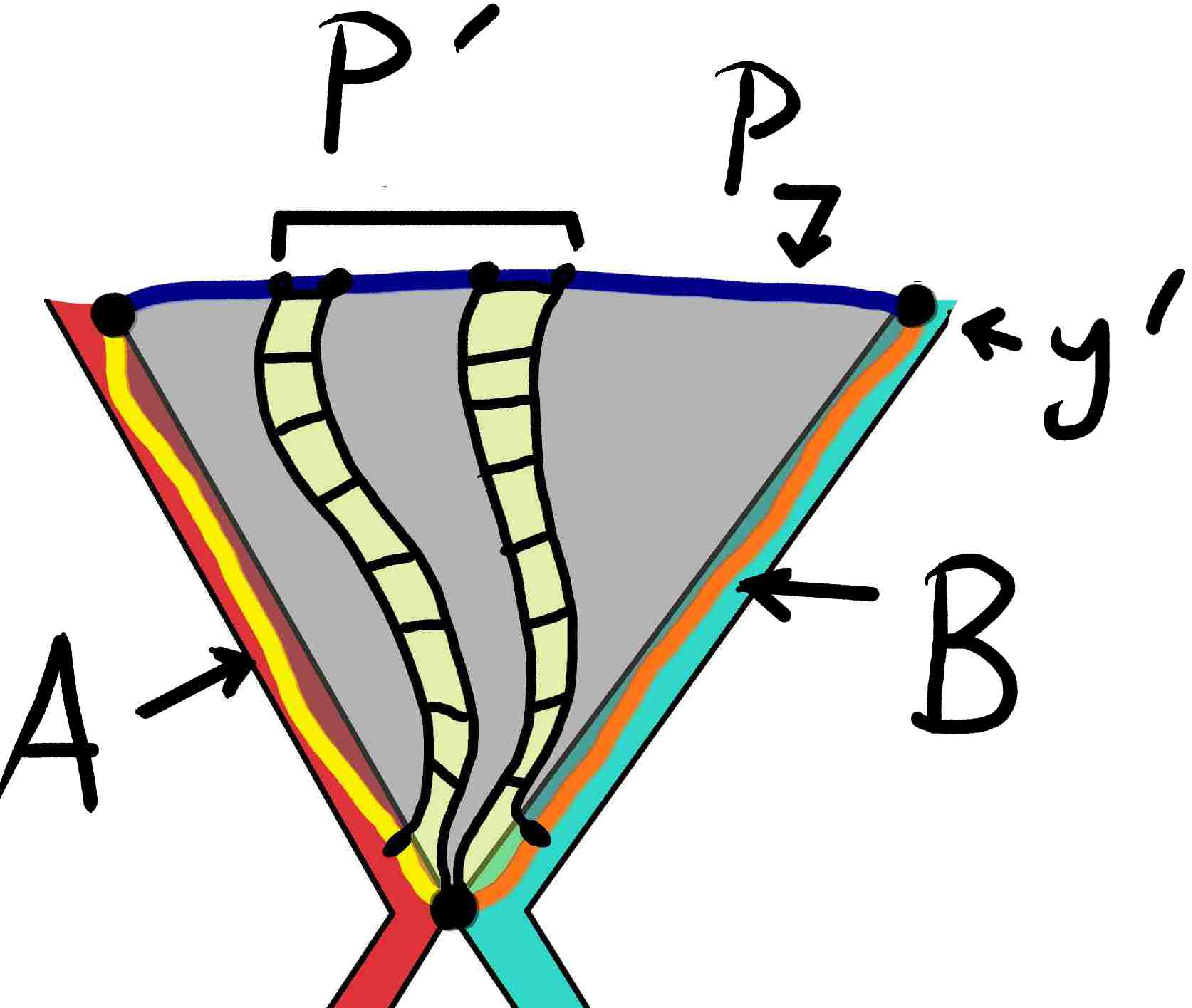}\\
  \caption{If $\crossing X$ is disconnected, then $\mathbf X$ has a cut-0-cube.}\label{fig:disconnecteddelta}
\end{figure}

\textbf{Disconnecting $\crossing X$ in a restriction quotient:}  Now suppose $\Lambda$ is a finite graph that separates $\crossing X$.  Let $\mathcal W'\subset\mathcal W$ be the set of hyperplanes that do not correspond to vertices of $\Lambda$.  Let $\mathbf X'=\mathbf X(\mathcal W')$ be the restriction quotient corresponding to this subset, and let $q:\mathbf X\rightarrow\mathbf X'$ be the quotient map.  If $W,W'$ are hyperplanes of $\mathbf X$, then $W,W'$ cross if and only if $q(W)$ crosses $q(W')$ in $\mathbf X'$.  Therefore, if $\crossing X-\Lambda$ is disconnected, it follows that $\crossing X'$ is disconnected, and therefore that $\mathbf X'\cong\mathbf Y_1\cup_{\{a\}}\mathbf Y_2$, where $\mathbf Y_i$ is a nonempty subcomplex containing at least one hyperplane, and $a$ is a single 0-cube.  Let $K=q^{-1}(\mathbf Y_1)\cap q^{-1}(\mathbf Y_2)=q^{-1}(a)$.  Since $q$ is continuous, $K$ separates $\mathbf X$ into two nonempty subcomplexes.  On the other hand, if $b,b'\in K$ are 0-cubes, then $q(b)=q(a)=q(b')$, so that $b(H)\neq b(H')$ if and only if $H$ is a hyperplane corresponding to a vertex of $\Lambda$.  Thus $K$ is finite and disconnects $\mathbf X$, contradicting compact indecomposability.
\end{proof}

\begin{lem}\label{lem:contactcrossing}
Let $\mathbf X$ be compactly indecomposable, with degree $D<\infty$.  If $\diam(\contact X)<\infty$, then $\diam(\crossing X)<\infty$.
\end{lem}

\begin{proof}
Since $\mathbf X$ is compactly indecomposable, $\crossing X$ is path-connected, by Lemma~\ref{lem:separation}. If $H$ and $H'$ are osculating hyperplanes and $d_{_{\crossing X}}(H,H')=R$, then a simple disc diagram argument shows that there is a shortest path $H=H_0\bot H_1\bot\ldots\bot H_{R-1}\bot H_R=H'$ in $\crossing X$ with the property that $H_i\coll H_j$ for all $i,j$, and hence $N(H)\cap N(H')$ contains a 0-cube of degree at least $R+1$.  Thus, if $H,H'$ osculate, then $d_{_{\crossing X}}(H,H')\leq D-1$.  Thus $d_{_{\crossing X}}(H,H')\leq D$ whenever $H\coll H'$, and it follows from the triangle inequality that $\diam(\crossing X)\leq (D-1)\diam(\contact X)$.
\end{proof}

\vspace{-0.5em}

\section{Rank-one isometries, $\contact X$, and $\simp\mathbf X$}\label{sec:rankoneapp}
In this section, we first use Theorem~\ref{thm:trichotomy1} and Theorem~\ref{thm:trichotomy2} to analyze the action of hyperbolic isometries of $\mathbf X$ on $\contact X$, and then re-interpret the rank-rigidity theorem of Caprace-Sageev~\cite{CapraceSageev} in terms of the contact graph and the simplicial boundary.  Rank-one isometries of a $\mathbf X$ whose axes diverge from every hyperplane are exactly those that act as hyperbolically on $\contact X$:

\begin{prop}\label{prop:rankonecontacteasy}
Let $G$ act on the CAT(0) cube complex $\mathbf X$ and let $g\in G$.  If $g$ acts on $\contact X$ with an unbounded orbit, then $g$ is rank-one.
\end{prop}

\begin{proof}
By Theorem~\ref{thm:semisimpleHAGLUND}, $g$ has a combinatorial axis $\gamma$.  If $g$ is not rank-one, then for any hyperplane $H$ crossing $\gamma$, we have $d_{_{\contact X}}(g^nH,H)\leq 3$ for all $n\in\integers$, since $g^nH$ and $H$ cross a common isometric half-flat by Proposition~\ref{prop:combinatorialrankone}. If $U$ is any other hyperplane, then $d_{_{\contact X}}(g^nU,U)\leq 3+2d_{_{\contact X}}(H,U)$, and thus every $g$-orbit in $\contact X$ is bounded.
\end{proof}

\begin{thm}\label{thm:rankonecontact}
Let $G$ act on the strongly locally finite CAT(0) cube complex $\mathbf X$ and let $g\in G$ be a combinatorially rank-one element.  Suppose that for all $n>0$, and for all hyperplanes $H$, $g^n\not\in\stabilizer(H)$.  Then $g$ has a quasi-geodesic axis in $\contact X$.
\end{thm}

\begin{proof}
Apply Theorem~\ref{thm:trichotomy1} and Theorem~\ref{thm:trichotomy2} to a combinatorial geodesic axis for $g$.
\end{proof}

We now restate the rank-rigidity theorem~\cite[Theorem~A, Theorem~B]{CapraceSageev} in terms of the contact graph and the simplicial boundary. Theorem~\ref{thm:RRcontact} follows directly from rank-rigidity, since the contact graph of a nontrivial product is a nontrivial join, and a hyperbolic element obtained as in~\cite{CapraceSageev} by applying double-skewering to a pair of \emph{strongly separated} (i.e. at distance at least 3 in $\crossing X$) hyperplanes necessarily acts with an unbounded orbit on $\contact X$.  We first need:

\begin{rem}
A group $G$ acting on $\mathbf X$ acts by simplicial automorphisms on $\simp\mathbf X$.  Also, the set of hyperplanes crossing a CAT(0) geodesic ray is a UBS, so that $g\in G$ has a fixed point in $\partial_{\infty}\mathbf X$ if and only if $g$ fixes the barycenter of a simplex of $\simp\mathbf X$.
\end{rem}

\begin{thm}[Rank-rigidity, contact graph form]\label{thm:RRcontact}
Let $G$ act essentially on the finite-dimensional CAT(0)cube complex $\mathbf X$.  Suppose in addition that either $G$ acts properly and cocompactly, or $G$ does not stabilize a simplex of $\simp\mathbf X$.  If $\contact X$ is unbounded, then $G$ contains a combinatorially rank-one element $g$, no power of which stabilizes a hyperplane.  Otherwise, $\contact X$ has diameter 2 and there exist unbounded convex subcomplexes $\mathbf Q_1,\mathbf Q_2\subset\mathbf X$ such that $\mathbf X\cong\mathbf Q_1\times\mathbf Q_2$.
\end{thm}

By Theorem~\ref{thm:productsandjoins}, a product decomposition of $\mathbf X$ with unbounded factors corresponds to a join decomposition of $\simp\mathbf X$, while an axis for a rank-one isometry yields a pair of isolated 0-simplices of $\simp\mathbf X$.  Rank-rigidity therefore implies:

\begin{thm}[Rank-rigidity, simplicial boundary form]\label{thm:RRsimplicial}
Under the same hypotheses on $G$ and $\mathbf X$ as in Theorem~\ref{thm:RRcontact}, if $\simp\mathbf X$ is connected, there exist unbounded convex subcomplexes $\mathbf Q_1,\mathbf Q_2\subset\mathbf X$ such that $\mathbf X\cong\mathbf Q_1\times\mathbf Q_2$, and $\simp\mathbf X$ decomposes as a nontrivial simplicial join.  Otherwise, $G$ contains a rank-one element $g$, no power of which stabilizes a hyperplane.
\end{thm}

\begin{prob}\label{prob:thmDRR}
Without using~\cite[Proposition~5.1]{CapraceSageev}, show that if $G$ acts on $\mathbf X$, satisfying the hypotheses of either form of the rank-rigidity theorem, and $\mathbf X$ decomposes as an iterated pseudoproduct of depth 2, then $\mathbf X\cong\mathbf Q_1\times\mathbf Q_2$ for unbounded subcomplexes $\mathbf Q_1,\mathbf Q_2$.  A solution to this problem would yield an alternate proof of rank-rigidity (in the locally finite case) using Theorem~\ref{thm:boundedcontact} together with double-skewering; this is discussed in~\cite[Chapter~6]{HagenPhD}.
\end{prob}

\section{Cubical divergence}\label{sec:divergence}
We study combinatorial divergence of geodesic rays in an arbitrary strongly locally finite CAT(0) cube complex $\mathbf X$, and obtain characterizations of linear and super-linear divergence and divergence of geodesics in terms of $\simp\mathbf X$.  In Section~\ref{sec:groupdiverge}, we state a criterion for linear divergence of cocompactly cubulated groups stated in terms of $\simp\mathbf X$ or $\contact X$.

\subsection{Simplex-paths and combinatorial divergence}\label{sec:simplexpaths}
Choose a 0-cube $x\in\mathbf X$.  Note that, by Lemma~\ref{lem:samestart}, each visible simplex of $\simp\mathbf X$ is represented by the set $\mathcal W(\gamma)$, where $\gamma:[0,\infty)\rightarrow\mathbf X$ is a combinatorial geodesic ray with $\gamma(0)=x$.  Given such a ray $\gamma$, let $v(\gamma)$ denote the simplex at infinity represented by the UBS $\mathcal W(\gamma)$.

\begin{defn}[Simplex path, simplicial Tits distance]\label{defn:simplexpath}
Let $u,v$ be simplices of $\simp\mathbf X$.  A \emph{simplex-path} is a sequence $u=u_0,u_1,\ldots,u_n=v$ such that $u_i\cap u_{i+1}\neq\emptyset$ for $0\leq i\leq n-1$.  The \emph{length} of this simplex-path is $n$. The \emph{simplicial Tits distance} $\eta(u,v)$ from $u$ to $v$ is the length of a shortest simplex-path whose initial simplex is $u$ and whose terminal simplex is $v$.  Note that $\eta(u,v)\leq1$ if and only if $u\cap v\neq\emptyset$, and $\eta(u,v)\leq2$ if $u,v$ are contained in a common simplex.
\end{defn}

The next lemma allows us to induct on the simplicial Tits distance between visible simplices without assuming that $\mathbf X$ is fully visible.

\begin{lem}\label{lem:shortestpathvisible}
Let $u_0,u_1,\ldots,u_n$ be a shortest simplex-path joining $u_0$ to $u_n$.  Then, for $1\leq i\leq n$, the $u_i$ can be chosen so that the simplex $u_i$ is visible.
\end{lem}

\begin{proof}
Each simplex is contained in a visible simplex by Theorem~\ref{thm:boundaryproperties} and Theorem~\ref{thm:visiblesimplex}.
%For each $i$, we have $u_i\cap u_j\neq\emptyset$.  Suppose that $u_i$ is contained in a simplex $u'_i$.  Then $u_{i-1}\cap u_i\subseteq u'_i$ and $u_i\cap u_{i+1}\subseteq u'_i$.  In particular, $u'_i\cap u_{i\pm1}\neq\emptyset$, and we can replace $u_i$ by $u'_i$ without increasing the length of the path.  Since $u_i$ is contained in a finite maximal simplex, by Theorem~\ref{thm:boundaryproperties}, we can choose this path so that each simplex (except the fixed initial and terminal simplices) is maximal.  Maximal simplices are visible, by Theorem~\ref{thm:visiblesimplex}.
\end{proof}

\begin{defn}[$r$-avoiding path, combinatorial divergence]\label{defn:ravoiding}
With respect to $x\in\mathbf X^{(0)}$, for $r\geq 0$, the combinatorial path $P\rightarrow\mathbf X$ is \emph{$r$-avoiding} if, for all 0-cubes $p\in P$, $d_{_{\mathbf X}}(p,x)\geq r$.  Let $\gamma,\gamma':[0,\infty)\rightarrow\mathbf X$ be combinatorial geodesic rays with $\gamma(0)=\gamma'(0)=x$.  For $r\geq 0$, let $\dive{\gamma}{\gamma'}(r)=\inf_P|P|,$ where $P$ varies over all $r$-avoiding paths joining $\gamma(r)$ to $\gamma'(r)$.  This is the \emph{combinatorial divergence} of the rays $\gamma,\gamma'$.

The rays $\gamma,\gamma'$ diverge \emph{linearly} if there exist $A>0,B\in\reals$ such that $\dive{\gamma}{\gamma'}(r)\leq Ar+B$ for all $r\geq 0$.  The bi-infinite combinatorial geodesic $\gamma:\reals\rightarrow\mathbf X$ has \emph{linear (combinatorial) divergence} if the rays $\gamma([0,\infty))$ and $\gamma((-\infty,0])$ diverge linearly.  $\mathbf X$ has \emph{linear divergence of rays} if each pair of combinatorial geodesic rays in $\mathbf X$ diverges linearly, and $\mathbf X$ has \emph{(uniformly) linear divergence} if there exist $A\geq 0,B\in\reals$ such that $\dive{\gamma}{\gamma'}(r)\leq Ar+B$ for all combinatorial geodesic rays $\gamma,\gamma'$ emanating from $x$ and for all $r\geq 0$.
\end{defn}

From now on, all geodesic rays are understood to emanate from a fixed initial 0-cube $x_0$.  If $\gamma,\gamma',\gamma''$ are three such rays, then $\dive{\gamma}{\gamma'}(r)\leq\dive{\gamma}{\gamma''}(r)+\dive{\gamma'}{\gamma''}(r)$ for all $r\geq 0$.

\subsection{Divergence of pairs of combinatorial geodesic rays}
Let $\gamma,\gamma:[0,\infty)\rightarrow\mathbf X$ be combinatorial geodesic rays emanating from $x_0$ and let $v,v'\subset\simp\mathbf X$ be the simplices at infinity represented by $\gamma$ and $\gamma'$ respectively.

%\subsubsection{Bounding divergence from above}
\begin{lem}\label{lem:divergeupper1}
If $v\subseteq v'$, then there exists $B\in\reals$ such that $\dive{\gamma}{\gamma'}(r)\leq 2r+B$ for all $r\geq 0$.  If $v$ and $v'$ lie in a common simplex of $\simp\mathbf X$, then there exists $B$ such that the same inequality holds for all $r\geq 0$.
\end{lem}
\begin{proof}
Suppose that $v\subseteq v'$.  Then the set $\mathcal H=\mathcal W(\gamma)-\mathcal W(\gamma')$ of hyperplanes that cross $\gamma$ but not $\gamma'$ is finite.  The set $\mathcal H'=\mathcal W(\gamma')-\mathcal W(\gamma)$ of hyperplanes that cross $\gamma'$ but not $\gamma$ is finite if and only if $v=v'$.

First, if $\mathcal H=\emptyset$, then for all $r\geq 0$, $\dive{\gamma}{\gamma'}(r)=d_{_{\mathbf X}}(\gamma(r),\gamma'(r))\leq 2r.$ To see this, first fix $r\geq 0$.  Consider the set of hyperplanes crossing the geodesic segment $P_r=\gamma([0,r])$.  Since $\mathcal H=\emptyset$, this set of hyperplanes is partitioned into two subsets: $\mathcal A_r$ is the set of hyperplanes that cross $P_r$ and $P'_r=\gamma'([0,r])$.  The set of remaining hyperplanes, $\mathcal B_r$, all cross $P_r$ and $\gamma'-P'_r$, since every hyperplane that crosses $\gamma$ also crosses $\gamma'$.

The set of hyperplanes crossing $P'_r$ consists of $\mathcal A_r$, together with a set $\mathcal C_r$ of hyperplanes $U$ such that each $U$ crosses $\gamma-P_r$, or separates an infinite subray of $\gamma$ from an infinite subray of $\gamma'$, and thus crosses all but finitely many elements of $\mathcal H$.  Note that $|\mathcal A_r|+|\mathcal B_r|=|\mathcal A_r|+|\mathcal C_r|=r$.  Note also that if $B\in\mathcal B_r$ and $C\in\mathcal C_r$, then $B\bot C$, so that the hyperplanes in $\mathcal B_r$ are orientable independently of those in $\mathcal C_r$.
\begin{figure}[h]
  \includegraphics[width=4in]{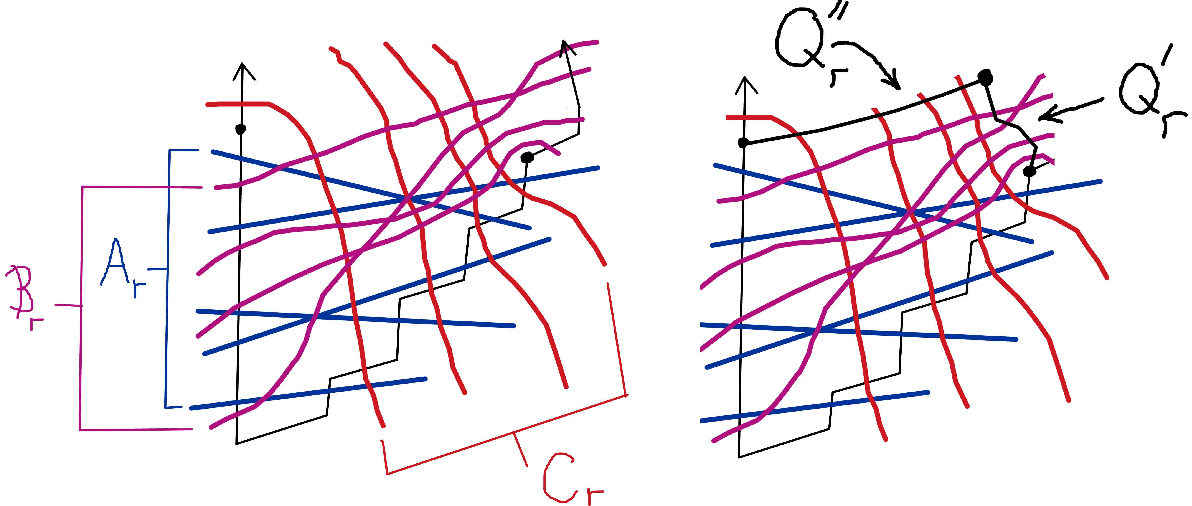}\\
  \caption{The sets $\mathcal A_r,\mathcal B_r,\mathcal C_r$ are shown at left. At right is the path $Q_r$.}\label{fig:hyperplanepartition}
\end{figure}

We now construct an $r$-avoiding geodesic segment $Q_r$ joining $x'_r=\gamma'(r)$ to $x_r=\gamma(r)$, from which the claimed equality follows.  Let $Q'_r(0)=x'_r$.  Let $B_0,\ldots,B_b$ be the hyperplanes in $\mathcal B_r$, numbered so that $x'_r\in N(B_0)$ and $B_0\coll B_1\coll\ldots\coll B_b$.  This assumption is justified, since any hyperplane separating $B_i$ from $B_{i+1}$ must cross $P_r$ and thus belong to $\mathcal B_r$.  Any hyperplane separating $x'_r$ from $B_0$ is likewise in $\mathcal B_r$.  For $0\leq t\leq b$, let $Q'_r(t)$ orient all hyperplanes $U\in\mathcal W-\mathcal B_r$ toward the halfspace $x'_r(U)$, and let $Q'_r(t)(B_s)=x'_r(B_s)$ for $s>t$.  For $s\leq t$, let $Q'_r(t)(B_s)=\mathbf X-x'_r(B_s)$.  This creates a geodesic segment from $x'_r$ to $Q'_r(t)$ that crosses exactly the set $\mathcal B_r$ of hyperplanes.  Note that each $Q'_r(t)$ is separated from $x_o$ by $\mathcal A_r\cup\mathcal C_r$, and thus $d_{_{\mathbf X}}(Q'_r(t),x_o)\geq r$ for all $t$.  In the same manner, we flip the hyperplanes $\mathcal C_r$ successively to yield a geodesic segment $Q''_r$ from $Q'_r(b)$ to $x_r$, with $Q''_r$ separated from $x_o$ by $\mathcal A_r\cup\mathcal B_r$.  The segment $Q_r=Q'_rQ''_r$ is the desired $r$-avoiding geodesic segment.  This yields the desired bound on the divergence.

Now, if $\mathcal H$ is nonempty, then there exists a combinatorial geodesic ray $\gamma''$ with $\gamma''(0)=x_0$, such that $d_{_{\mathbf X}}(\gamma(r),\gamma''(r))\leq 2|\mathcal H|$ for each $r\geq 0$.  Moreover, $\mathcal W(\gamma)=\mathcal W(\gamma'')\cup\mathcal H$, so that $\mathcal W(\gamma'')\subseteq\mathcal W(\gamma')$.  Given such a ray $\gamma''$, we see that
\[\dive{\gamma}{\gamma'}(r)\leq\dive{\gamma}{\gamma''}(r)+\dive{\gamma'}{\gamma''}(r)\leq 2|\mathcal H|+2r,\]
and the first assertion of the lemma follows, with $B=2|\mathcal H|$.

It thus suffices to produce the ray $\gamma''$.  Let $\mathcal A_r,\mathcal B_r,\mathcal C_r$ be as above and let $\mathcal H_r$ be the set of $H\in\mathcal H$ that cross $P_r$.  Since $\mathcal H$ is finite, $\mathcal H_r$ has uniformly bounded cardinality $h\leq|\mathcal H|$.  Let $\mathcal H=\{H_0,\ldots,H_h\}$, where for $i<j$, the hyperplane $H_i$ separates the 1-cube $N(H_j)\cap\gamma$ from $x_o$.  Let $R$ be the subpath of $\gamma$ joining $x_o$ to the 0-cube of $N(H_h)\cap\gamma$ separated from $x_o$ by $H_h$.  Let $\alpha$ be the combinatorial geodesic subray of $\gamma$ emanating from the terminal 0-cube of $R$.  Then $H_h$ crosses every hyperplane dual to $\alpha$ and $\alpha\rightarrow N(H_h)$.  For $t\geq 0$, define $\alpha'(t)$ to be the 0-cube that orients each hyperplane $U$ toward $\alpha(t)(U)$, except for $U=H_h$, in which case $\alpha'(t)(H_h)=\mathbf X-\alpha(t)(H_h)$.  This orientation is consistent and canonical for each $t$, and the map $t\mapsto\alpha'(t)$ yields a geodesic ray $\alpha'\rightarrow N(H_h)$ emanating from the penultimate 0-cube of $R$.  Concatenating $\alpha'$ with the first $|R|-1$ 1-cubes of $R$ yields a geodesic $\gamma''_h\rightarrow\mathbf X$ such that $\gamma''_h(t)$ is separated from $\gamma(t+1)$ by at most one hyperplane, namely $H_h$, for each $t\geq 0$.  The claim now follows by induction on $|\mathcal H|$.  Note that by our previous claim, for all $r\geq 0$, $\dive{\gamma}{\gamma''}(r)\leq 2|\mathcal H|,$ since $\mathcal W(\gamma'')\subseteq\mathcal W(\gamma)$.  On the other hand, $\mathcal W(\gamma'')\subseteq\mathcal W(\gamma')$ by construction.

Suppose that $v,v'$ belong to a common simplex $v''\subseteq\simp\mathbf X$.  If $v''=v$ or $v''=v'$, then the claim follows from the first assertion of the lemma.  There are two remaining possibilities:  first, we could have that $v,v'$ are distinct 0-simplices belonging to a common 1-simplex $v''$.  Second, if $v\not\subset v'$ and $v'\not\subset v$, and $v$ is not a 0-simplex, then there exists a simplex $u=v\cap v'$ that is properly contained in both $v,v'$.

Let $\mathcal V,\mathcal V'$ be as before.  Then $\mathcal V=\mathcal V_0\sqcup\mathcal H$ and $\mathcal V'=\mathcal V'_0\sqcup\mathcal H'$, where $\mathcal H,\mathcal H'$ are finite, and $\mathcal V_0,\mathcal V'_0$ satisfy the following conditions:
\begin{compactenum}
\item If $U\in\mathcal V_0$, then either $U$ crosses both $\gamma$ and $\gamma$, or $U$ crosses every element of $\mathcal V'_0$.
\item If $U'\in\mathcal V_0$, then either $U'$ crosses both $\gamma$ and $\gamma$, or $U$ crosses every element of $\mathcal V_0$.
\end{compactenum}
In either case ($v\cap v'=u$ or $v\cap v'=\emptyset$), the proof of assertion~$(1)$ shows, mutatis mutandis, that $\dive{\gamma}{\gamma'}(r)\leq d_{_{\mathbf X}}(\gamma(r),\gamma'(r))+2(|\mathcal H|+|\mathcal H'|)$ for all $r\geq 0$.
\end{proof}

Lemma~\ref{lem:divergeupper1} enables us to bound the $\dive{\gamma}{\gamma'}(r)$ using the simplicial Tits distance, as follows.

\begin{lem}\label{lem:divergeupper2}
There exists $B\in\reals$ such that $\dive{\gamma}{\gamma'}(r)\leq 2\left[\eta(v,v')+1\right]r+B$ for all $r\geq 0$.
\end{lem}

\begin{proof}
If $\eta(v,v')=\infty$, then $A=\eta(v,v')$ suffices, so suppose $\eta(v,v')=N<\infty$.  The claim holds for $N=0$, with $A=1$, by Lemma~\ref{lem:divergeupper1}.  More generally, if $v$ and $v'$ lie in a common simplex, then the claimed inequality holds, for some $B$ and $A=1$, by Lemma~\ref{lem:divergeupper1}.

Suppose that the claim holds for some $N-1\geq 0$.  Let $\gamma,\gamma'$ satisfy $\eta(\gamma,\gamma')=N$.  Then there exist simplices $u_0,u_{N}\subset\simp\mathbf X$ such that $v=u_0,v'=u_{N}$, and there is a minimal simplex-path $P=u_0,u_1,\ldots,u_{N}$ in $\simp\mathbf X$ joining $u_0$ to $u_{N}$.  Let $\gamma''$ be a geodesic ray emanating from $x_o$ and representing $u_N\cap u_{N-1}$.  By induction, $\dive{\gamma}{\gamma''}(r)\leq 2Ar+B_1$ for all $r$ and some fixed $A$ depending on $N$, since $\eta(v,u_{N}\cap u_{N-1})=N-1$.  Since $u_{N}\cap u_{N-1}$ and $v'$ lie in a common simplex, namely $u_{N}$, Lemma~\ref{lem:divergeupper1} implies that $\dive{\gamma'}{\gamma''}(r)\leq 2r+B_1,$ so that $\dive{\gamma}{\gamma'}(r)\leq 2(A+1)r+B_1+B_2.$  Choosing $A=1$ suffices for $N=0$, so we may take $A=N+1$ and $B=B_1+B_2$, and conclude that $\dive{\gamma}{\gamma'}(r)\leq 2r\left(\eta(\gamma,\gamma')+1\right)+B.$
\end{proof}

%\subsubsection{Bounding divergence from below}
Let $\gamma,\gamma'$ be geodesic rays with $\gamma(0)=\gamma'(0)=x_0$ and for each $r\geq 0$, let $x_r=\gamma(r)$ and $x'_r=\gamma'(r)$.  Let $S_r$ be a shortest $r$-avoiding path joining $x_r$ to $x'_r$.  Then $S_r$ decomposes as a ``minimal'' concatenation of geodesic segments as follows.  Let $P_1^r$ be a subpath of $S_r$ beginning at $x_r$ that is a geodesic segment of $\mathbf X$ and let $P_1^r$ have maximal length among such paths.  If $P_1^r$ is properly contained in $S_r$, then let $P_2^r$ be a maximal geodesic subpath of $S_r-\interior{P_1^r}$ beginning at the terminal 0-cube of $P_1^r$.  In this way, we can write $S_r=P_1^rP_2^r\ldots P_{c_r}^r,$ where each $P_i^r$ is a geodesic segment and the hyperplane dual to the initial 1-cube of $P_i^r$ crosses a 1-cube of $P_{i-1}^r$.  This is a \emph{fan decomposition} of $S_r$, and, if $c_r=1$ for some fan decomposition of $S_r$, for all $r\geq 0$, then the simplices $v,v'$ of $\simp\mathbf X$ represented by $\gamma,\gamma'$ either have nonempty intersection or lie in a common simplex.

Indeed, let $C_r=\gamma([0,r])$ and let $C'_r=\gamma'([0,r])$.  Suppose that $S_r$ is a geodesic segment for each $r\geq 0$, and suppose that $v\cap v'=\emptyset$.  Let $\mathcal U=\mathcal W(\gamma)\cap\mathcal W(\gamma')$.  Now $\mathcal U$ is finite, since $v\cap v'=\emptyset$ (see Theorem~\ref{thm:opticalspace}).  Then every hyperplane separating $\gamma(r)$ from $x_0$ either belongs to $\mathcal U$ or separates $\gamma'(r)$ from $\gamma(r)$.  Analogously, every hyperplane separating $\gamma'(r)$ from $x_0$ separates $\gamma(r)$ from $\gamma'(r)$ or belongs to $\mathcal U$.  Every hyperplane separating $\gamma(r)$ from $\gamma'(r)$ separates either $\gamma(r)$ or $\gamma'(r)$ from $x_0$.  Hence $|P_1^r|=\dive{\gamma}{\gamma'}(r)=2r-2|\mathcal U|$ for all $r\geq 0$.  Applying the Folding Lemma to $\gamma$ and $\gamma'$, and then truncating the common initial segment of the resulting rays, we can assume that $\mathcal U=\emptyset$ and $|P_1^r|=r$.

Now let $P_r=QQ'$, where $Q$ is a geodesic segment of length $r$ emanating from $x_r$ and $Q'$ is a geodesic segment of length $r$ terminating at $x'_r$.  Let $D\rightarrow\mathbf X$ be a minimal-area disc diagram bounded by the closed path $C_rQQ'(C'_r)^{-1}$.  Without affecting the points $x_r,x_0,x'_r$ or the sets $\mathcal W_r,\mathcal W'_r$, we can modify $D$ and the paths $C_r,C'_r$ so that no two dual curves in $D$ emanating from $C_r$ (and no two dual curves emanating from $C'_r$) cross.

Since $\partial_pD$ is the concatenation of three geodesic segments, and any two dual curves have a total of four ends on $\partial_pD$, distinct dual curves in $D$ must map to distinct hyperplanes.  Thus $D\rightarrow\mathbf X$ is an isometric embedding and, since $\mathcal U=\emptyset$, each dual curve travels from $QQ'$ to $C_r$ or from $QQ'$ to $C'_r$.  Since every 0-cube of $QQ'$ is separated from $x_0$ by at least $r$ dual curves, it follows that each dual curve emanating from $C'_r$ crosses each dual curve emanating from $C_r$, and hence that $v$ and $v'$ lie in a common simplex.  More generally:

\begin{lem}\label{lem:divergelower2}
Suppose that there exists $K\in\reals$ such that for all $r\geq 0$, $|S_r|\leq d_{_{\mathbf X}}(x_r,x'_r)+2K.$  Then, if $v$ and $v'$ are 0-simplices, they lie in a common simplex.  If the above inequality holds and $v$ is not a 0-simplex, then $v\cap v'\neq\emptyset$.
\end{lem}

\begin{proof}
As before, suppose that $v\cap v'\neq\emptyset$ and that $\mathcal W(\gamma)\cap\mathcal W(\gamma')=\emptyset$.  To achieve the latter requires that we apply the Folding Lemma and truncate the common initial segment of the resulting rays.  This does not affect $v$ or $v'$, and modifies $\gamma$ and $\gamma'$ in only finitely many 1-cubes.

In case $K=0$, the claim follows from the previous discussion.  The hypothesis says that for all $r\geq 0$, there are at most $K$ hyperplanes that cross $S_r$ in more than one 1-cube.  Suppose that $v,v'$ do not belong to any common simplex.  Let $H_r$ be the hyperplane dual to the $r^{th}$ 1-cube of $\gamma$, and define $H'_r$ analogously for $\gamma'$.  The there exists $R,N\geq 0$ such that, for all $r\geq R$, the hyperplane $H_r$ does not cross $H'_N$. Also, we can choose $R$ such that, for all $r\geq 0$, there are at most $R$ hyperplanes $U$ that cross $H_r$ and $H'_r$.

Each 0-cube $c\in S_r$ is separated from $x_o$ by at least $r$ hyperplanes.  Let $Q_r,Q'_r$ be the initial, length-$r$ segments of $\gamma,\gamma'$ respectively, emanating from $x_o$.  Let $D_r$ be a minimal-area disc diagram bounded by $Q_rS_r(Q'_r)^{-1}$.  Since $\gamma$ and $\gamma'$ have no common dual hyperplanes, every dual curve in $D_r$ travels from $Q_r$ or $Q'_r$ to $S_r$.  Without modifying $x_o,\gamma(r),\gamma'(r)$ or $v,v'$, we may change $Q_r,Q'_r$ so that no two dual curves in $D_r$ emanating from $Q_r$ cross, and no two dual curves emanating from $Q'_r$ cross.  See Figure~\ref{fig:fanagain}.  Let $u_r$ be the number of dual curves in $D_r$ that start and end on $S_r$, so that $|S_r|=2r+2u_r$.

\begin{figure}[h]
  \includegraphics[width=2.5in]{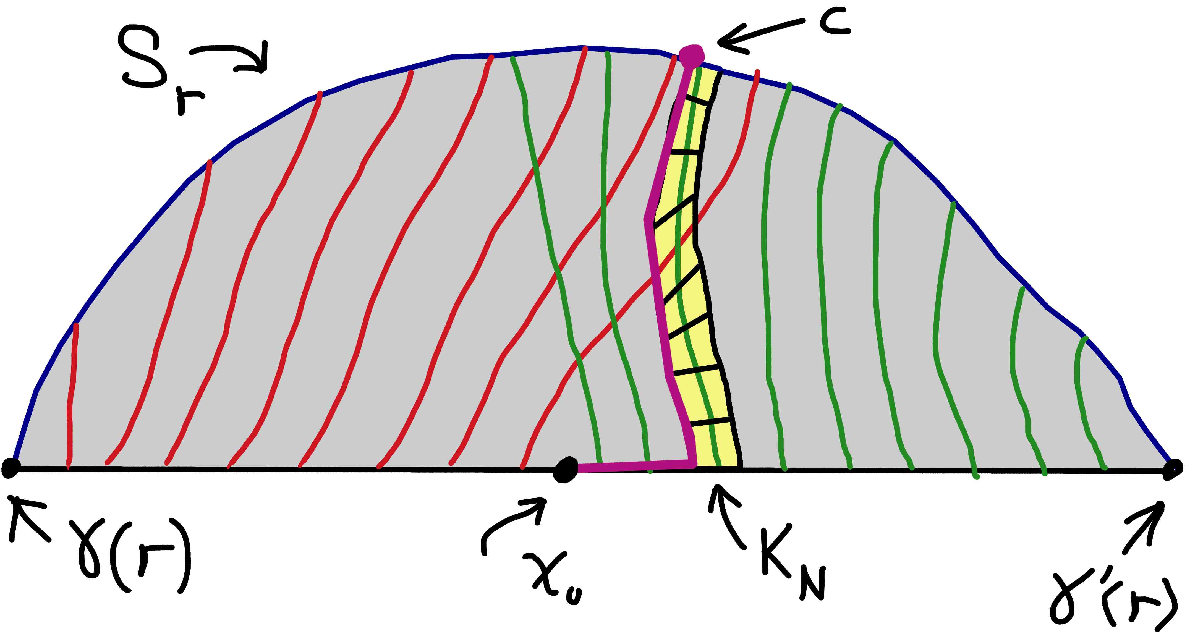}\\
  \caption{The diagram $D_r$ showing that either $\eta(v,v')\leq 1$ or $u_r\rightarrow\infty$.}\label{fig:fanagain}
\end{figure}

Let $c\in S_r$ be the 0-cube in $S_r\cap N(H'_N)$ that is separated from $\gamma'(r)$ by each $H'_i$ with $N\leq i\leq r$, as shown in Figure~\ref{fig:fanagain}.  Then there is a path in the image of $D_r$ joining $c$ to $x_o$ that crosses only $N$ of the dual curves mapping to hyperplanes $H'_i$: simply travel along the carrier of the dual curve $K_N$ mapping to $H_N$, until arriving at $\gamma'(N-1)$.  Then travel along $\gamma'$ to $x_o$.  Hence every dual curve in $D_r$ separating $c$ from $x_o$ maps to a hyperplane that crosses $H_0$.  Thus $u_r\geq r-R-N$ for all $r\geq 0$, and thus $|S_r|-d_{_{\mathbf X}}(\gamma(r),\gamma'(r))$ is unbounded as $r\rightarrow\infty$, a contradiction.  We conclude that $v$ and $v'$ either intersect or lie in a common simplex.
\end{proof}

\begin{lem}\label{lem:divergelower1}
Suppose that there exist $M,N\in\reals$ such that, for all $r\geq 0$, the rays $\gamma,\gamma'$ satisfy $\dive{\gamma}{\gamma'}(r)\leq 2Nr+M.$
Then $N\geq\eta(v,v')$.
\end{lem}

\begin{proof}
First fix $r\geq0$, and let $S_r$ be a shortest $r$-avoiding path joining $\gamma(r)$ to $\gamma'(r)$, so that $|S_r|\leq 2Nr+M$.  Let $k=\left\lceil\frac{M}{N}\right\rceil$.  Let $S_r=P_1^r\ldots P_N^r$ be a concatenation of combinatorial paths such that each $P_i^r$ has length at most $2r+k$, and $||P_i^r|-|P_j^r||\leq 1$ for all $i,j$.  In other words, subdivide $S_r$ into $N$ paths, each of length as close as possible to $\frac{|S_r|}{N}$.

For $1\leq i\leq N$, let $f_i(r)$ be the terminal 0-cube of $P_i^r$, let $f_0(r)=\gamma(r),f_N(r)=\gamma'(r)$.  Let $Q_i(r)$ be a geodesic segment joining $f_i(r)$ to $x_o$, for $0\leq i\leq n$, so that $Q_0(r),Q_N(r)$ are hyperplane-equivalent to the initial length-$r$ segments of $\gamma,\gamma'$ respectively.

Let $t_i(r)$ be the number of hyperplanes crossing both $Q_i(r)$ and $Q_{i+1}(r)$.  Let $p_i(r)$ be the number of hyperplanes dual to at least 2 1-cubes of $P_i^r$.  Then $2r-t_i(r)+2p_i(r)\leq 2r+k$, so that $2t_i(r)\geq 2p_i(r)-k$.  By K\"{o}nig's lemma, there are combinatorial geodesic rays $\gamma_i$, for $1\leq i\leq N$, such that $\gamma_i(r)=f_i(r)$ for arbitrarily large values of $r$.  If $t_i(r)$ is unbounded as $r\rightarrow\infty$, then $\gamma_i,\gamma_{i+1}$ have infinitely many common hyperplanes, and thus the associated simplices $v_i,v_{i+1}\subset\simp\mathbf X$ satisfy $\eta(v_i,v_{i+1})\leq 1$.  If $t_i(r)$ is uniformly bounded as $r\rightarrow\infty$, then so is $p_i(r)$, and thus $P_i(r)$ fails to be a geodesic segment by a uniformly bounded number of 1-cubes, whence we can take $\eta(v_i,v_{i+1})\leq 1$ by Lemma~\ref{lem:divergelower2}.  Hence there is a simplex path of length at most $N$ joining $v$ to $v'$.  Indeed, if $t_i(r)$ is uniformly bounded, then either $v_i\cap v_{i+1}$ is nonempty, or $v_i$ and $v_{i+1}$ belong to a common simplex.
\end{proof}

\subsubsection{Linear divergence of rays}
$\mathbf X$ has \emph{weakly uniformly linear divergence} if there exists $A\geq 0$ such that, for all $x_0\in\mathbf X^0$ and for all combinatorial geodesic rays $\gamma,\gamma'$ with $\gamma(0)=\gamma'(0)=x_0$, we have $\dive{\gamma}{\gamma'}(r)\leq Ar+B$ for all $r\geq 0$, where $B$ (but not $A$) depends on $\gamma$ and $\gamma'$.

\begin{thm}\label{thm:lineardivergence}
Let $\mathbf X$ be strongly locally finite.  Then $\simp\mathbf X$ is bounded if and only if $\mathbf X$ has weakly uniformly linear divergence.
\end{thm}

\begin{proof}
Suppose there exists $A<\infty$ such that $\eta(v,v')\leq A$ for all simplices $v,v'$ of $\simp\mathbf X$.  By Lemma~\ref{lem:divergeupper2}, if $\gamma,\gamma'$ are rays emanating from $x_0$, then $\dive{\gamma}{\gamma'}(r)\leq 2(A+1)r+B$ for all $r\geq 0$.  Conversely, suppose that for some $A\in\reals$, and for each $\gamma,\gamma'$, there exists $B\in\reals$ such that the above inequality holds for all $r\geq 0$.  Then by Lemma~\ref{lem:divergelower2}, we have $\eta(v,v')\leq A+1$ for all visible simplices $v,v'$.

Therefore, it suffices to show that $\eta(v,v')$ is uniformly bounded as $v,v'$ vary over the set of visible simplices if and only if the graph $(\simp\mathbf X)^1$ has finite diameter.  First, suppose there exists $D<\infty$ such that, for all visible simplices $v,v'$, we have $\eta(v,v')\leq D$.  Let $u,u'$ be 0-simplices. Then, by Theorem~\ref{thm:visiblesimplex}, there exist visible simplices $v,v'$ with $u\subseteq v$ and $u'\subseteq v'$.  Let $v=v_0,v_1,\ldots,v_d=v'$ be a simplex-path joining $v$ to $v'$, where $d\leq D$.  For $0\leq i\leq d-1$, let $w_i$ be a 0-simplex of $v_i\cap v_{i+1}$.  Then $w_i$ is adjacent to $w_{i+1}$ for $0\leq i\leq d-1$.  Moreover, $w_0$ and $u$ lie in $v_0$ and are thus adjacent; similarly, $w_{d-1}$ and $u'$ are adjacent, so that $u$ and $u'$ are at distance at most $d-1+2\leq D+1$ in $(\simp\mathbf X)^1$. On the other hand, suppose that $\diam(\simp\mathbf X)\leq D<\infty$.  Let $v,v'$ be simplices.  Let $u\subseteq v,u'\subseteq v'$ be 0-simplices.  Let $P:[0,d]\rightarrow(\simp\mathbf X)^1$ be a combinatorial path, where $P_(0)=u$ and $P(d)=u'$, and $d\leq D$.  For each $i$, let $s_i=[P(i),P(i+1)]$ be a 1-simplex of $P$.  Then $v,s_0,s_1,\ldots,s_{d-1},v'$ is a simplex-path of length $d+2$ joining $v$ to $v'$, whence $\eta(v,v')\leq\diam(\simp\mathbf X)+2$.
\end{proof}

\subsubsection{Divergence of combinatorial geodesics}
Recall that $\diver{a}{b}{c}{\rho}$ is the divergence of the points $a,b\in\mathbf X^1$ with respect to the standard path-metric, the basepoint $c$, and the linear function $\rho$.  It follows from the proof of Theorem~\ref{thm:lineardivergence} that the bi-infinite combinatorial geodesic $\gamma:\reals\rightarrow\mathbf X$ has linear divergence if and only if the simplices of $\simp\mathbf X$ represented by the rays $\gamma((-\infty,0])$ and $\gamma([0,\infty))$ belong to the same component of $\simp\mathbf X$.  Note that if, for some function $f:\reals^+\rightarrow\reals^+$, the divergence $\diver{\alpha(r)}{\alpha(-r)}{\alpha(0)}{\rho}$ of the combinatorial geodesic $\alpha:\reals\rightarrow\mathbf X$ exceeds $f(r)$, then the divergence $\divers{\mathbf X}{\rho}(r)$ of $(\mathbf X,d_{_{\mathbf X}})$ exceeds $f(r)$.  The next observation is a consequence of the proof of Theorem~\ref{thm:lineardivergence}.

\begin{cor}\label{cor:superlinear1}
Let $\mathbf X$ be strongly locally finite.  Let $\alpha:\reals\rightarrow\mathbf X$ be a bi-infinite combinatorial geodesic.  Let $v^-,v^+$ be the simplices of $\simp\mathbf X$ represented by the rays $\gamma_-=\alpha((-\infty,0])$ and $\gamma_+=\alpha([0,\infty))$.  The simplices $v^-$ and $v^+$ lie in different components of $\simp\mathbf X$ if and only if the divergence of the geodesic $\alpha$ is super-linear, with respect to the combinatorial metric $d_{_{\mathbf X}}$.  In particular, if either $\gamma_+$ or $\gamma_-$ is rank-one, then $\alpha$ has super-linear divergence.
\end{cor}

\begin{proof}
By the proof of Theorem~\ref{thm:lineardivergence}, $\dive{\gamma_-}{\gamma_+}(r)$ is at most linear if and only if $v^-,v^+$ lie in the same component of $\simp\mathbf X$.  Hence, if $v^-$ and $v^+$ lie in the same component of $\simp\mathbf X$, then for all $r\geq 0$, there is an $r$-avoiding path $P_r$ of length at most $Ar+B$ joining $\gamma_-(r)$ to $\gamma_+(r)$. Thus $\diver{\alpha(r)}{\alpha(-r)}{\alpha(0)}{Cr-D}\leq Ar+B$ for any $C<1,D\geq 0$ and for all $r\geq 0$, i.e. $\alpha$ has linear divergence with respect to the combinatorial metric.

Conversely, suppose that there exist $0<C<1,D\geq 0$ such that for all $r\geq 0$, the preceding inequality holds.  Then $\gamma_+(\frac{D+r}{C})$ and $\gamma_-(\frac{D+r}{C})$ are joined by an $r$-avoiding path $P_r\rightarrow\mathbf X$, for each $r\geq 0$, with $|P_r|\leq Ar+B$.  Let $Q_r$ be the subpath of $\gamma$ joining $\gamma_+(r)$ to $\gamma_+(\frac{D+r}{C})$ and let $Q'_r$ be the subpath of $\gamma_-$ joining $\gamma_-(\frac{D+r}{C})$ to $\gamma_-(r)$.  Then $Q_rP_rQ'_r$ is an $r$-avoiding path joining $\gamma_+(r)$ to $\gamma_-(r)$, so that
\[\dive{\gamma_+}{\gamma_-}(r)\leq|Q_r|+|Q'_r|+Ar+B=\left(2(C^{-1}-1)+A\right)r+B+2DC^{-1},\]
and $v,v'$ lie in the same component of $\simp\mathbf X$.
\end{proof}

\subsection{Divergence of cubulated groups}\label{sec:groupdiverge}
Let the group $G$ act properly, essentially, and cocompactly on the combinatorially geodesically complete CAT(0) cube complex $\mathbf X$.  Denote by $\divers{G}{\rho}(r)$ the divergence of $G$ with respect to the word-metric arising from some finite generating set and with respect to the linear function $\rho$.  The following is an immediate consequence of rank-rigidity and~\cite[Proposition~3.3]{KapovichLeeb}, and the fact that nontrivial product decompositions of $\mathbf X$ correspond to nontrivial join decompositions of $\simp\mathbf X$ and $\contact X$.

\begin{thm}\label{thm:divergenceofgroup}
Let $G$ act properly, cocompactly, and essentially on the combinatorially geodesically complete CAT(0) cube complex $\mathbf X$.  If $\contact X$ is bounded, then $G$ has linear divergence if and only if $\contact X$ decomposes as the join of two infinite proper subgraphs (i.e. $\simp\mathbf X$ decomposes as the simplicial join of two proper subcomplexes).  Otherwise, $G$ has at least quadratic divergence.
\end{thm}

It also appears as though the techniques used to prove Theorem~\ref{thm:lineardivergence} can be brought to bear on the question of when the divergence of $G$ is at most quadratic.

%%%%%%%%%%%%%%%%%%%%%%%%%%%%%%%%%%%%%%%%%%%%%%%%%%%%%%%%%%%%%%%%%%%%%%%%
%%                  BIBLIOGRAPHY
%%%%%%%%%%%%%%%%%%%%%%%%%%%%%%%%%%%%%%%%%%%%%%%%%%%%%%%%%%%%%%%%%%%%%%%%
\bibliographystyle{alpha}
\bibliography{SimplicialBoundaryRevisedAGT}

\begin{thebibliography}{DGP10}

\bibitem[BC08]{BandeltChepoi_survey}
H.-J. Bandelt and V.~Chepoi.
\newblock Metric graph theory and geometry: a survey.
\newblock In J.~Pach J.~E.~Goodman and R.~Pollack, editors, {\em Surveys on
  Discrete and Computational Geometry: Twenty Years Later}, volume 453, pages
  49--86. Contemp. Math., AMS, Providence, RI, 2008.

\bibitem[BC11]{BehrstockCharney}
Jason Behrstock and Ruth Charney.
\newblock Divergence and quasimorphisms of right-angled {A}rtin groups.
\newblock {\em Mathematische Annalen}, pages 1--18, 2011.

\bibitem[BGS85]{BaGS}
W.~Ballmann, M.~Gromov, and V.~Schroeder.
\newblock Manifolds of nonpositive curvature.
\newblock {\em Prog. Math., Birkhauser, Boston}, 61, 1985.

\bibitem[BH99]{BridsonHaefliger}
Martin~R. Bridson and Andr{\'e} Haefliger.
\newblock {\em Metric spaces of non-positive curvature}.
\newblock Springer-Verlag, Berlin, 1999.

\bibitem[Bri91]{BridsonThesis}
M.R. Bridson.
\newblock Geodesics and curvature in metric simplicial complexes.
\newblock In E.~Ghys, A.~Haefliger, and A.~Verjovsky, editors, {\em Group
  theory from a geometrical viewpoint}, Proc. {I}{C}{T}{P}, {T}rieste, {I}taly,
  pages 373--463. World Scientific, Singapore, 1991.

\bibitem[CD95a]{CharneyDavis94}
Ruth Charney and Michael~W. Davis.
\newblock Finite {$K(\pi, 1)$}s for {A}rtin groups.
\newblock In {\em Prospects in topology (Princeton, NJ, 1994)}, volume 138 of
  {\em Ann. of Math. Stud.}, pages 110--124. Princeton Univ. Press, Princeton,
  NJ, 1995.

\bibitem[CD95b]{CharneyDavis95b}
Ruth Charney and Michael~W. Davis.
\newblock The ${K}(\pi,1)$-problem for hyperplane complements associated to
  infinite reflection groups.
\newblock {\em J. Amer. Math. Soc.}, 8(3):597--627, 1995.

\bibitem[CH11]{ChepoiHagen}
V.~Chepoi and M.F. Hagen.
\newblock On embeddings of {C}{A}{T}(0) cube complexes into products of trees.
\newblock {\em ar{X}iv:1107.0863v1}, pages 1--46, 2011.
\newblock Submitted.

\bibitem[Che00]{Chepoi2000}
Victor Chepoi.
\newblock Graphs of some {${\rm CAT}(0)$} complexes.
\newblock {\em Adv. in Appl. Math.}, 24(2):125--179, 2000.

\bibitem[Che11]{ChepoiNiceLabeling}
V.~Chepoi.
\newblock Nice labeling problem for event structures: a counterexample.
\newblock {\em Ar{X}iv preprint, ar{X}iv:1107.1207}, 2011.

\bibitem[CK00]{CrokeKleiner}
C.~Croke and B.~Kleiner.
\newblock Spaces with nonpositive curvature and their ideal boundaries.
\newblock {\em Topology}, 39(3):549--556, 2000.

\bibitem[CN05]{ChatterjiNiblo04}
Indira Chatterji and Graham Niblo.
\newblock From wall spaces to {$\rm CAT(0)$} cube complexes.
\newblock {\em Internat. J. Algebra Comput.}, 15(5-6):875--885, 2005.

\bibitem[CS11]{CapraceSageev}
Pierre-Emmanuel Caprace and Michah Sageev.
\newblock Rank rigidity for {CAT}(0) cube complexes.
\newblock {\em Geom. Funct. Anal.}, 21:851--891, 2011.

\bibitem[DGP10]{DahmaniGuirardelPrzytycki}
F.~Dahmani, V.~Guirardel, and P.~Przytycki.
\newblock Random groups do not split.
\newblock {\em Mathematische Annalen}, 349:657--673, 2010.

\bibitem[EFO07]{EppsteinFalmagneOvchinnikov}
D.~Eppstein, {J.-Cl.} Falmagne, and S.~Ovchinnikov.
\newblock {\em Media Theory}.
\newblock Springer-Verlag, 2007.

\bibitem[Far03]{Farley2003}
Daniel~S. Farley.
\newblock Finiteness and {$\rm CAT(0)$} properties of diagram groups.
\newblock {\em Topology}, 42(5):1065--1082, 2003.

\bibitem[Far05]{Farley2005}
Daniel~S. Farley.
\newblock Actions of picture groups on {CAT}(0) cubical complexes.
\newblock {\em Geometriae Dedicata}, 110:221--242, 2005.

\bibitem[Ger94a]{GerstenDivergence2}
S.~M. Gersten.
\newblock Divergence in 3-manifold groups.
\newblock {\em Geometric And Functional Analysis}, 4:633--647, 1994.

\bibitem[Ger94b]{GerstenDivergence}
S.~M. Gersten.
\newblock Quadratic divergence of geodesics in {C}{A}{T}(0) spaces.
\newblock {\em Geometric And Functional Analysis}, 4:37--51, 1994.

\bibitem[Gro87]{Gromov87}
M.~Gromov.
\newblock Hyperbolic groups.
\newblock In {\em Essays in group theory}, volume~8 of {\em Math. Sci. Res.
  Inst. Publ.}, pages 75--263. Springer, New York, 1987.

\bibitem[Hag07]{HaglundSemisimple}
F.~Haglund.
\newblock Isometries of {C}{A}{T}(0) cube complexes are semi-simple.
\newblock {\em ArXiv e-prints}, 2007.

\bibitem[Hag11]{HagenQuasiArb}
M.F. Hagen.
\newblock Weak hyperbolicity of cube complexes and quasi-arboreal groups.
\newblock {\em ar{X}iv:1101.5191v5}, pages 1--43, 2011.
\newblock Submitted.

\bibitem[Hag12]{HagenPhD}
M.F. Hagen.
\newblock {\em Geometry and combinatorics of cube complexes}.
\newblock PhD thesis, McGill University, 2012.

\bibitem[Hou74]{Houghton}
C.H. Houghton.
\newblock Ends of locally compact groups and their quotient spaces.
\newblock {\em J. Austral. Math. Soc.}, 17:274--284, 1974.

\bibitem[HP98]{HaglundPaulin98}
Fr{\'e}d{\'e}ric Haglund and Fr{\'e}d{\'e}ric Paulin.
\newblock Simplicit\'e de groupes d'automorphismes d'espaces \`a courbure
  n\'egative.
\newblock In {\em The Epstein birthday schrift}, pages 181--248 (electronic).
  Geom. Topol., Coventry, 1998.

\bibitem[HW10]{HruskaWiseAxioms}
Chris Hruska and Daniel~T. Wise.
\newblock Finiteness properties of cubulated groups.
\newblock {\em Preprint}, 2010.

\bibitem[IK00]{ImKl}
W.~Imrich and S.~{Klav\v{z}ar}.
\newblock {\em Product Graphs:Structure and Recognition}.
\newblock Wiley-Interscience Publication, New York, 2000.

\bibitem[Isb80]{Isbell}
J.R. Isbell.
\newblock Median algebra.
\newblock {\em Trans. Amer. Math. Soc.}, 260:319--362, 1980.

\bibitem[KL98]{KapovichLeeb}
M.~Kapovich and B.~Leeb.
\newblock 3-manifold groups and nonpositive curvature.
\newblock {\em Geom. Funct. Anal.}, 8:841--852, 1998.

\bibitem[Lea10]{LearyInfiniteCubes}
Ian Leary.
\newblock A metric {K}an-{T}hurston theorem.
\newblock {\em Preprint}, 2010.

\bibitem[Nic04]{NicaCubulating04}
Bogdan Nica.
\newblock Cubulating spaces with walls.
\newblock {\em Algebr. Geom. Topol.}, 4:297--309 (electronic), 2004.

\bibitem[NR03]{NibloReeves03}
G.~A. Niblo and L.~D. Reeves.
\newblock Coxeter groups act on {${\rm CAT}(0)$} cube complexes.
\newblock {\em J. Group Theory}, 6(3):399--413, 2003.

\bibitem[NS11]{NevoSageev}
A.~Nevo and M.~Sageev.
\newblock The {P}oisson boundary of {C}{A}{T}(0) cube complex groups.
\newblock {\em Ar{X}iv preprint 1105.1675v1}, pages 1--34, 2011.

\bibitem[OW11]{OllivierWiseDensity}
Yann Ollivier and Daniel~T. Wise.
\newblock Cubulating random groups at density~$<\frac16$.
\newblock {\em Trans. Amer. Math. Soc.}, 363:4701--4733, 2011.

\bibitem[Rol98]{Roller98}
Martin~A. Roller.
\newblock Poc-sets, median algebras and group actions. {A}n extended study of
  {D}unwoody's construction and {S}ageev{'}s theorem.
\newblock 1998.

\bibitem[Sag95]{Sageev95}
Michah Sageev.
\newblock Ends of group pairs and non-positively curved cube complexes.
\newblock {\em Proc. London Math. Soc. (3)}, 71(3):585--617, 1995.

\bibitem[Sco77]{Scott}
G.P. Scott.
\newblock Ends of pairs of groups.
\newblock {\em J. Pure Appl. Algebra}, 11:179--198, 1977.

\bibitem[vdV93]{vandeVel_book}
M.~van~de Vel.
\newblock {\em Theory of Convex Structures}.
\newblock Elsevier Science Publishers, Amsterdam, 1993.

\bibitem[Wis]{WiseIsraelHierarchy}
Daniel~T. Wise.
\newblock The structure of groups with a quasiconvex hierarchy.
\newblock 183 pp. Preprint 2011.

\bibitem[Wis04]{WiseSmallCanCube04}
Daniel~T. Wise.
\newblock Cubulating small cancellation groups.
\newblock {\em GAFA, Geom. Funct. Anal.}, 14(1):150--214, 2004.

\end{thebibliography}
%%%%%%%%%%%%%%%%%%%%%%%%%%%%%%%%%%%%%%%%%%%%%%%%%%%%%%%%%%%%%%%%%%%%%%%%%%%%%%%%%%%%%%%%%%%
%
%

\end{document}